%% file: EP_noise_cell-level_june26.tex
\documentclass{article}

\usepackage[english]{babel}
\usepackage{authblk} 

\usepackage[letterpaper,top=2cm,bottom=2cm,left=3cm,right=3cm,marginparwidth=1.75cm]{geometry}

\usepackage{cancel}

\usepackage{amsmath}
\usepackage{amsthm}
\usepackage{graphicx}
\usepackage[colorlinks=true, allcolors=blue]{hyperref}
\usepackage{comment}
\usepackage{amsfonts}
\usepackage{amssymb}
\usepackage{graphicx}
\usepackage{epstopdf}
\usepackage{algorithmic}
\usepackage{enumerate}
\usepackage{enumitem}
\usepackage{mathrsfs}
\usepackage{mathtools}
\usepackage{xcolor}
\usepackage{subcaption}
\usepackage{framed}
\usepackage{bbm} 
\allowdisplaybreaks

\begingroup
\newtheorem{theorem}{Theorem}[section]

\newtheorem{lemma}[theorem]{Lemma}
\newtheorem{definition}[theorem]{Definition}
\newtheorem{cor}[theorem]{Corollary}
\newtheorem{remark}{{Remark}}
\newtheorem{example}{Example}

\endgroup

\newcommand{\di}[1]{{\rm div}(#1)}
\newcommand{\n}{\mathbf n}

\newcommand{\ip}[1]{\textcolor{blue}{#1}}
\newcommand{\im}[1]{\textcolor{orange}{#1}}
\newcommand{\todo}[1]{\textcolor{red}{#1}}

\newcommand{\ve}{\varepsilon}

\newcommand{\dd}{\mathrm{d}}

\newcommand{\E}{\mathbb E}
\newcommand{\R}{\mathbb{R}}

\newcommand{\bA}{\mathcal A_{\rm bi}}
\newcommand{\F}{\mathcal F}
\newcommand{\PP}{\mathcal P}

\usepackage{enumitem}

\let\oldthebibliography\thebibliography
\let\endoldthebibliography\endthebibliography

\renewcommand{\thebibliography}[1]{%
  \oldthebibliography{#1}%
  \setlength{\itemsep}{0pt}%
}
\renewcommand{\endthebibliography}{%
  \endoldthebibliography
}

\title{Long-time behavior of a nonlocal and non-monotone SPDE-ODE system arising in electrophysiology}
\author[1]{T. Gebäck}
\author[2]{O. Misiats}
\author[1]{I. Motschan Ulander}
\author[1]{I. Pettersson \thanks{Corresponding author, irinap@chalmers.se}}

\affil[1]{Chalmers University of Technology and Gothenburg University, Sweden}
\affil[2]{Virginia Commonwealth University, USA}

\date{}

\begin{document}
\maketitle

\begin{abstract}
This paper concerns a coupled semilinear SPDE-ODE system modelling the electropermeabilization phenomenon, which designates a transient increase in cell membrane permeability induced by short, high-voltage electric pulses. We present a stochastically perturbed electroporation model that couples electrostatic equations for the electric potential in the extra- and intracellular domains and a nonlinear evolution law for the transmembrane potential jump with an ordinary differential equation describing the porosity degree of the membrane. We prove the existence and uniqueness of a variational solution of the resulting coupled 
stochastic PDE-ODE system. Its long-time behavior is governed by the corresponding invariant measure for which we establish the regularity of its support.
The ergodicity of this invariant measure is further established for a truncated nonlinear reaction term, corresponding to the case of a bounded electric potential. The main technical challenge arises from the nonlinear reaction term, which is neither Lipschitz continuous nor locally monotone. 
We also present a numerical example, computing the solution and its time averages for both additive and multiplicative noise, that provides an indication for the existence of an invariant measure.
\end{abstract}



\tableofcontents

\section{Introduction}
Electroporation, or more generally electropermeabilization, designates a temporary increase in membrane permeability due to the applied short high-voltage pulses. In recent years, the interest in electropermeabilization has been constantly increasing since
electroporation has 
proved to be an efficient non-thermal ablation technique in cancer and arrhythmias treatments \cite{Al-Sakere_Andre_Bernat_Connault_Opolon_Davalos_Rubinsky_Mir_2007}, \cite{Tasu_Vionnet_Velasco_Lafitte_Poignard_2023}. 
We adopt a dynamic phenomenological model by  Kavian et al. \cite{kavian2014classical}  formulated for a single cell, which is described in Section \ref{sec:deterministic_model}. It consists of electrostatics equations for the electric potential $u$ in a bulk domain, satisfying the conditions of continuity of flux and a nonlinear dynamic condition for the jump of the potential on the interface (membrane). These equations are coupled with a nonlinear ordinary differential equation for $w$ describing the proportion of pores in the membrane (see problem \eqref{eq:microscopic-prob}). The membrane permeabilization is described
by choosing an appropriate function for the surface conductivity $S_m$, instead
of adding an electroporation current based on the pore creation as in Neu, Krassowska, et
al. \cite{neu1999asymptotic}. 

In the present work, in order to account for various  random effects, such as temperature fluctuations or uncertainty in the applied electric field, we add noise to the cell membrane in the electroporation model presented in \cite{kavian2014classical}. The nonlinear coupled system can be written as an abstract parabolic equation on the interface (cell membrane) with a nonlocal operator containing a composition of Dirichlet-to-Neumann maps acting on the jump $v=[u]$ of the potential $u$ through the interface (see \eqref{eq:stochastic_kavian}). To be precise, we consider a bounded $G \subset \mathbb{R}^3$ consisting of two disjoint domains, intracellular $G_{\rm i}$ and extracellular $G_{\rm e}$ part, separated by the interface $\Gamma$ representing cell membranes. Then, for $T> 0$, the pair $(v,w)$ satisfies the following coupled system on $\Gamma$:
\begin{equation}
\label{eq:stochastic_kavian_intro}
\begin{aligned}
c_m \dd v &= -(\mathcal{A} v + S_m(v,w) + {\sigma_{\rm i} \nabla p \cdot \n})\, \dd t + b(v)\, \dd W_t, \quad
v(0,x)=v_0(x) &\quad& \text{on}\,\, (0,T]\times \Gamma,\\
\dd w &= f(v,w)\, \dd t, \quad w(0,x)=w_0(x)&\quad& \text{on}\,\, (0,T]\times \Gamma,
\end{aligned}
\end{equation}
where $c_m>0$, $\mathcal{A}$ is a nonlocal Dirichlet-to-Neumann map defined by \eqref{def:L}-\eqref{eq:microscopic-prob_stationary}, $p$ represents an external excitation which solves \eqref{eq:p}, $f$ is a Lipschitz function given by \eqref{eq:f}, $b$ satisfies \ref{B2}, and $S_m =(S_0 + S_1 w) v$ is a nonlinear non-Lipschitz reaction term modelling the surface conductivity. 
We  prove the existence and uniqueness of a variational solution of \eqref{eq:stochastic_kavian_intro} (see Theorem \ref{th:existence-(0,T)_kavian}) and its properties, such as Markov and Feller property. We also establish the existence of an invariant measure in the autonomous case, describe the regularity of its suppport, and prove its uniqueness for a truncated (Lipschitz) nonlinearity. Further, we define a semigroup and show that the variational solution is also a mild solution. Finally, we compute numerically the solution $(v,w)$ of \eqref{eq:stochastic_kavian_intro} and its time averages for one-dimensional additive and multiplicative noise. The time averages of individual trajectories converge, and the standard deviation decays with increasing final time at the rate $T^{-1/2}$, which illustrates the existence of an invariant measure. 

For the well-posedness of the problem, we adopt the variational approach by Pardoux introduced in \cite{pardouxequations, pardoux2021stochastic}, and further developed by Krylov and Rozovski \cite{krylov2007stochastic}. 
An important generalization of the variational approach is developed by Liu, Röckner, and Prévôt \cite{prevot2007concise,liu2010spde,liu2011existence,liu2013local,liu2015stochastic}. The latter approach attracted a lot of attention, and was employed for establishing various probabilistic properties, such as ergodicity, existence of random attractors, and large diviation principle, among others. 
A distinctive feature of the model is the combination of a nonlocal Dirichlet-to-Neumann operator, which induces global spatial coupling along the interface, with a drift satisfying substantially weaker monotonicity and coercivity conditions than those assumed in classical well-posedness theories for globally or locally monotone SPDEs and for equations with one-sided Lipschitz coefficients (see, e.g., \cite{da2014stochastic,liu2015stochastic} and Section~5 of \cite{liu2015stochastic}). Despite its seemingly simple form, the nonlinear reaction $S_m(v,w)=(S_0 + S_1 w) v$ in the electroporation model \eqref{eq:stochastic_kavian} does not satisfy the local monotonicity assumptions in Section 5.1.3, \cite{liu2015stochastic},  for arbitrary functions from the corresponding functional space. However, the analysis of the coupled ODE problem enables us to establish the fulfilment of similar conditions along the solutions only by means of introducing an appropriately chosen auxiliary equation.  

For the proof of the existence of an invariant measure, we use the classical approach based on the results
of Krylov and Bogoliubov \cite{kryloff1937theorie} on the tightness of a family of measures. To this end, we establish the Markov and Feller property of the transition semigroup, and derive uniform in time estimates for the solution in $L^2(\Gamma)$ and $H^{1/2}(\Gamma)$-norm. Note that for the original non-truncated reaction term $S_m(v,w)=(S_0 + S_1 w) v$, it is not sufficient to prove a uniform in time estimate for $\sup_{t\ge 0} \E \Big[\|v\|_{L^2(\Gamma)}^2 + \|w\|_{L^2(\Gamma)}^2\Big]$, as in the case of Lipschitz nonlinearities in Theorem 6.1.2 \cite{da1996ergodicity}. We prove uniform estimates for the time averages in the stronger $H^{1/2}(\Gamma)$-norm (see Lemma \ref{lm:uniform_time_estimate}), which allows us to use the Krylov-Bogoliubov theorem and establish the existence of an invariant measure by a tightness argument (see Theorem \ref{th:existence_inv_measure}). 
We prove that any invariant measure $\mu$ on $L^2(\Gamma) \times L^2(\Gamma;[0,1])$ is supported on a more regular space $H^{1/2}(\Gamma) \times H^{1/2}(\Gamma; [0,1])$. Moreover, we show that a stationary solution with the law $\mu$ takes values in $H^{1/2}(\Gamma) \times H^{1/2}(\Gamma; [0,1])$ with probability one for each $t\ge 0$.  In general, without uniform in time estimates in strong norm and under generalized monotonicity and coercivity assumptions, one can prove the existence of an invariant measure only in finite-dimensional case, as it is done in Proposition 4.3.5 \cite{liu2015stochastic}.

Classically, the long time behavior is studied for dissipative systems or, in other words, under strict monotonicity condition. According to Doob's theorem (see Theorem 4.2.1 in \cite{da1996ergodicity}), uniqueness of invariant measure is then a consequence of $t_0$-regularity (see Section 4.1, \cite{da1996ergodicity}), which follows from the strong Feller property and irreducibility for some $t>0$. The uniqueness of invariant measure for variational solutions under strict monotonicity assumption can be found in Theorem 4.3.9 \cite{liu2015stochastic}. Without strong dissipativity of the drift, the uniqueness of invariant measure is not guaranteed, and has been established in some cases under weak dissipativity assumption (see e.g. \cite{liu2011weak_dissipative} and references therein). In 
In our case, we show the uniqueness of the invariant measure for a truncated nonlinear reaction term (see \eqref{eq:S_M}) which corresponds to a bounded electric potential. This assumption is reasonable from the application perspective.
For the truncated problem, the uniqueness will follow from the exponential stability of the system, which holds if the Lipschitz constants are sufficiently small (see assumptions of Theorem \ref{th:existence_inv_measure}). In this case, we prove also the existence of a complete stationary solution $(U^\ast)_{t\in \mathbb R}$ (see Theorem \ref{th:inv_measure_truncated}).

The existence of an invariant measure for stochastically forced semilinear parabolic partial differential equations of the Ginzburg-Landau type was shown in \cite{eckmann2001invariant}.
For the long-time behavior of stochastic reaction-diffusion equations, we refer also to 
\cite{misiats2016existence} and \cite{misiats2020invariant} where the existence and uniqueness of an invariant measure was proved for stochastic reaction-diffusion equations in unbounded domains and with weakly dissipative nonlinearities using Krylov-Bogoliubov approach. In \cite{glatt2014existence}, the long-time behavior of strong solutions of the stochastic three-dimensional primitive equations was studied. In addition to the ergodicity, it was proved that the invariant measure is supported on strong solutions.
An alternative coupling technique approach for proving the existence of an invariant measure was suggested by Mueller \cite{mueller1993coupling}, where a one-dimensional stochastic heat equation was studied. The technique can also be used for space-time white noise.
Stochastic reaction-diffusion systems with non-Lipschitz reaction term was considered, for example, in \cite{Cerrai_2003}. The authors establish the existence and uniqueness of a mild solution in the space of continuous functions and existence of an invariant measure. The ergodicity for mild solutions of SPDE systems with multiplicative noise was considered in \cite{cerrai2006asymptotic}.

A closely related nonlinear coupled system arises in the bidomain model of the heart tissue. The bidomain model can be derived by the homogenization methods from the cell-scale equations describing the distribution of the electric potential (see e.g. \cite{franzone2002degenerate}, \cite{amar2013hierarchy}, \cite{jerez2023derivation} for derivation in the deterministic case and \cite{Pettersson_Rybalko_Rybalko_2025} for the case of random geometry). 
Well-posedness and long-time behavior for bidomain systems in stochastic setting for various classes of ionic models including FitzHugh--Nagumo, Aliev--Panfilov, and Rogers--McCulloch were obtained in several works. We refer to \cite{bendahmane2019stochastically} for the existence and uniqueness of a martingale solution. In  \cite{hieber2020bidomain}, the authors proved the existence and uniqueness of a global weak solution, and characterized its long time behavior (i.e. existence of invariant measures). The extension of these results to the case of strong solutions was accomplished in \cite{kapustyan2022strong} in dimension two and \cite{hieber2023global} in dimension three. In \cite{bendahmane2024stochastic}, the authors  proposed a stochastic electromechanical bidomain model and established the existence of weak solutions.

The paper is organized as follows. In Section \ref{sec:preliminaries_and_main_result} we formulate the problem in the deterministic setting, introduce a stochastically perturbed version of it, and formulate the main results in Section \ref{sec:main_results}. Theorem \ref{th:existence-(0,T)_kavian} on the existence and uniqueness of a variational solution is proved in Section \ref{sec:proof_Th_existence_solution}. In Section \ref{sec:properties_of_solution} we establish uniform in time bounds for the solution and prove the Markov and Feller properties of the variational solution, as well as show that this solution is also mild. Section \ref{sec:existence_inv_measure} is devoted to the existence of an invariant measure for the non-Lipschitz reaction term and the regularity of its support. Motivated by applications, we study the model with a truncated nonlinearity, which corresponds to the case when the electric potential is bounded. In Section \ref{sec:inv_measure_truncated_reaction} we prove that in this case an invariant measure exists and is unique, and thus ergodic. Finally, in Section \ref{sec:numerics} we present the numerical computations of the solution and its time averages for additive and multiplicative noise.

\section{Preliminaries and main results}
\label{sec:preliminaries_and_main_result}

\subsection{Deterministic electroporation model on the cell-scale}
\label{sec:deterministic_model}
On the cell level, we adopt the model of the electropermeabilization phenomenon presented in \cite{kavian2014classical}. 
Let $G \subset \mathbb{R}^3$ be a Lipschitz domain that  consist of two disjoint domains, intracellular $G_{\rm i}$ and extracellular $G_{\rm e}$ part, separated by the interface $\Gamma$ representing cell membranes (see Figure \ref{fig:Y}), which assumed to be Lipschitz continuous. The interface $\Gamma$ might have one or more connected components, that is the domain $G$ might contain several cells. We assume that $G_{\rm i}$ does not intersect the boundary of $G$. 
\begin{figure}[htp]
\centering
    \def\svgwidth{0.28\textwidth}
    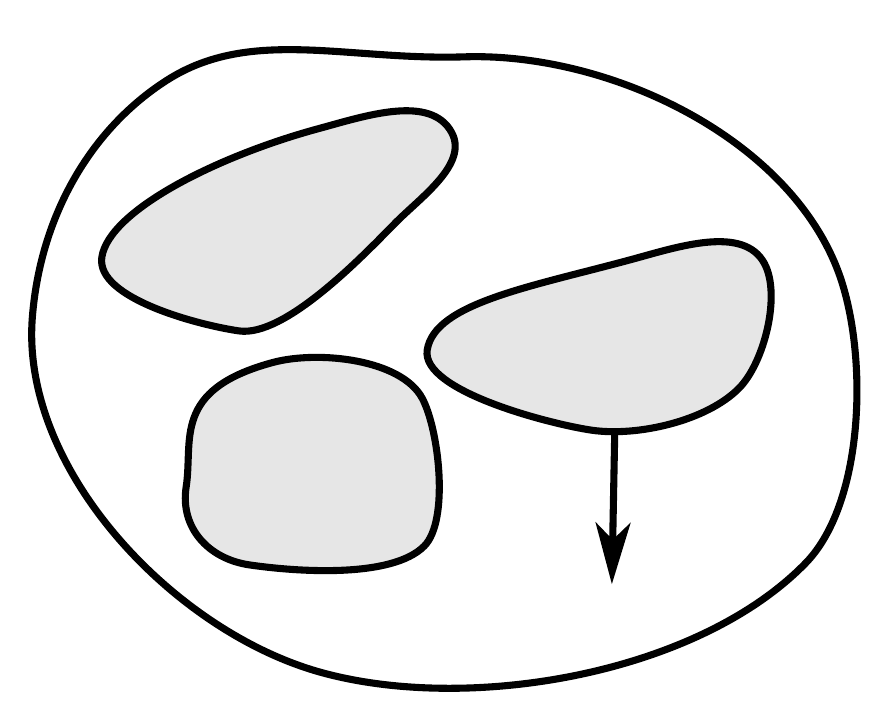
\caption{$G=G_{\rm i} \cup \Gamma \cup G_{\rm e}$.}
\label{fig:Y}
\end{figure}
The conductivity of the medium and the electric potential are denoted by
\begin{align*}
\sigma (x)= \begin{cases}
 \sigma_{\rm i}(x) \,\,\mbox{in}\,\,     G_{\rm i},\\
  \sigma_{\rm e}(  x) \,\,\mbox{in}\,\,     G_{\rm e}
\end{cases}
,
\quad 
u(  x)=\begin{cases}u_{\rm i}(  x) \,\,\mbox{in}\,\,     G_{\rm i}\\
u_{\rm e}(  x) \,\,\mbox{in}\,\,     G_{\rm e}
\end{cases}.
\end{align*}
The electric potential is discontinuous across the membrane, and its jump across the membrane $[u] := u_{\rm i}\big|_{  \Gamma}-u_{\rm e}\big|_{  \Gamma}$ satisfies the following transmission condition on $\Gamma$:
\begin{align*}
       c_m \partial_{  t} [u] +   S_m([u],w) = -    \sigma_{\rm i} \nabla u_{\rm i} \cdot \n= - \sigma_{\rm e} \nabla u_{\rm e} \cdot \n.
\end{align*}
Similar to $[u]$, the jump of the current is denoted by $[\sigma \nabla u \cdot \n]= \sigma_{\rm i} \nabla u_{\rm i} \cdot \n \big|_{\Gamma} - \sigma_{\rm e} \nabla u_{\rm e} \cdot \n \big|_{\Gamma}$, and $\n$ is the unit normal vector exterior to $G_{\rm i}$.
Here $S_m([u],w)=(S_0+S_1 w)[u]$. The factor $S_0+S_1 w$ is interpreted as the surface conductivity of the membrane.
The time-dependent membrane conductivity is modelled by interpolating two values, the lipid conductivity $S_0$ and the value $  S_1+S_0$ above which the permeabilization is not reversible. The interpolation parameter $w$ solves the following initial value problem on $  \Gamma$:\\
\begin{align}
    \label{eq:f}
    \begin{split}
    &\partial_t w
    = \max \left(\frac{\beta([u])-w}{  \tau_{\rm ep}}\, ,\,
    \frac{\beta([u])-w}{  \tau_{\rm res}}\right) =: f([u], w),\quad
    w(0,  x)=w_0(x).
    \end{split}
\end{align}
The positive constants $\tau_{\rm ep}$ and $\tau_{\rm res}$ designate characteristic electroporation and resealing time, respectively, with $\tau_{\rm ep}$ being several orders of magnitude smaller than $\tau_{\rm res}$.
The unknown function $w$ is interpreted as the degree of permeability of the cell membrane.
For a given jump of the membrane potential $[u]$, if $\beta([u]) - w \geq 0$ the electric pulse on the membrane is high enough to increase the membrane permeability, the degree of porosity $w$ starts growing, and the characteristic electroporation time is  $  \tau_{\rm ep}$. If, on the other hand, $\beta([u]) - w < 0$, the membrane is resealing in time $\tau_{\rm res}$. 
We refer to  \cite{kavian2014classical} and \cite{leguebe2014conducting} for specific choices of $\beta$ in the context of modelling the electropermeabilization phenomenon. The assumptions on $\beta$ are summarized in \ref{assumption:f_beta}.

The discontinuous electric potential satisfies the electrostatics equations in the intra- and extracellular medium, and solves the following coupled time-dependent transmission problem:
\begin{align}
\label{eq:microscopic-prob}
\begin{alignedat}{2}
-\di{\sigma \nabla u}  &= 0 &\quad& \mbox{in}\,\, (0,T]\times (G_{\rm i} \cup G_{\rm e}), \\
[\sigma \nabla u  \cdot \n] &= 0&\quad&\mbox{on}\,\, (0,T]\times\Gamma, \\
- \sigma_{\rm i} \nabla u \cdot \n &= c_m \partial_t [u] + S_m([u],w)
+ { \sigma_{\rm i} \nabla p \cdot \n}
&\quad&\mbox{on}\,\,(0,T]\times\Gamma,
\\
\partial_t w
    &= f([u], w)&\quad&\mbox{on}\,\,(0,T]\times\Gamma, \\
[u](0,x) &=v_0(x), \quad w(0,x)=w_0(x)&\quad&\mbox{on}\,\,\Gamma,\\
u   &= 0&\quad&\mbox{on}\,\, (0,T]\times\partial G,
\end{alignedat}
\end{align}
where $p$, for each $t\in [0,T]$, solves the nonhomogeneous Dirichlet problem
\begin{align}
\label{eq:p}
\begin{alignedat}{2}
\di{\sigma \nabla p} &= 0 &\quad& \text{in } G,\\
p &= g &\quad& \text{on } \partial G .
\end{alignedat}
\end{align}
In this way, the voltage $g$ applied of the boundary on the domain is transferred to the interface $\Gamma$, and the term $\sigma_{\rm i} \nabla p \cdot \n$ represents the external excitation. Note that the estimates for the solution will be given in terms of the norm of $g$, the external excitation on the boundary $\partial G$.  

We impose the following assumptions on the coefficients:
\begin{enumerate}[label=(A\arabic*)]
\setcounter{enumi}{0}
\item \label{assumption_sigma}  $\sigma(x)=\begin{cases}
    \sigma_{\rm i}(x) \, \mbox{in} \, G_{\rm i}\\
    \sigma_{\rm e}(x) \, \mbox{in} \, G_{\rm e}
\end{cases}$ with $\sigma_{i,e} \in L^\infty(G)$, $0< \underline{\sigma}\le \sigma \le \overline{\sigma}<\infty$. 

\item \label{assumption:f_beta} $f$ is defined in \eqref{eq:f} where $\beta$ is even, non-decreasing, 
$0 \leq \beta(\xi) \leq 1$, and Lipschitz continuous
\begin{align*}
|\beta(v_1)-\beta(v_2)| \le L_\beta |v_1-v_2|, \quad v _1, v_2 \in \mathbb{R}.
\end{align*}

\item \label{assumption:g} $g\in C([0,T]; H^{1/2}(\partial G))$.

\item \label{assumption:initial_conditions} $v_0, w_0 \in L^2(\Gamma), 0\le w_0 \le 1$.

\end{enumerate}
We note that the Lipschitz continuity of $\beta$ is equivalent to the Lipschitz continuity of the function $f$ defined in \eqref{eq:f}.

We may now rewrite \eqref{eq:microscopic-prob} as an evolution equation on $\Gamma$. To this end, we introduce a Dirichlet-to-Neumann-type operator $\mathcal{A} : D(\mathcal{A}) \subset H^{1/2}(\Gamma) \to H^{-1/2}(\Gamma)$ by setting, for each  $\phi\in H^{1/2}(\Gamma)$, \begin{align}
\label{def:L}
    \langle \mathcal{A} v , \phi \rangle_{H^{-1/2}(\Gamma), H^{1/2}(\Gamma)} = \int_{G_{\rm i} \cup G_{\rm e}} \sigma \nabla z^{(v)} \cdot \nabla z^{(\phi)} \, \dd x,
\end{align}
where 
$z^{(q)}=(z_i^{(q)}, z_e^{(q)}) \in H^1(G_{\rm i}) \times H^1( G_{\rm e})$, for a given jump $q \in H^{1/2}(\Gamma)$ ($q=v, \phi$), solves the following problem:
\begin{align}
\label{eq:microscopic-prob_stationary}
\begin{alignedat}{2}
-\di{\sigma \nabla z^{(q)}}  &=0 &\quad& \mbox{in}\,\, G_{\rm i} \cup G_{\rm e}, \\
[\sigma \nabla z^{(q)}  \cdot \n] &= 0 &\quad&\mbox{on}\,\, \Gamma,\\
[z^{(q)}]&= q
&\quad&\mbox{on}\,\,\Gamma,
\\
z^{(q)}   &= 0&\quad&\mbox{on}\,\, \partial G.
\end{alignedat}
\end{align}
The lemma below presents the key properties of the nonlocal operator $\mathcal{A}$ used throughout the paper.
\begin{lemma}
\label{lm:L}
The operator $\mathcal{A}$ defined by \eqref{def:L} is symmetric, non-negative, and satisfies the following properties: 
\begin{itemize}
\item[(i)] For any $\alpha > 0$, there exists $\Lambda >0$ depending on $\alpha$ such that, for any $v \in H^{1/2}(\Gamma)$,
\[
\langle \mathcal{A} v , v\rangle_{H^{-1/2}(\Gamma), H^{1/2}(\Gamma)} + \alpha \|v\|_{L^2(\Gamma)}^2\ge \Lambda \|v\|^2_{H^{1/2}(\Gamma)}.
\]
\item[(ii)] 
$
|\langle \mathcal{A} v , \varphi\rangle_{H^{-1/2}(\Gamma), H^{1/2}(\Gamma)}|
\leq C \| v \|_{H^{1/2}(\Gamma)} \| \varphi\|_{H^{1/2}(\Gamma)}$, for any $v, \varphi \in H^{1/2}(\Gamma)$.
\item[(iii)] 
$
\|\mathcal{A} v\|_{H^{-1/2}(\Gamma)}
\le C\| v \|_{H^{1/2}(\Gamma)}$, for any $v, \varphi \in H^{1/2}(\Gamma)$.
\end{itemize}
\end{lemma}
\begin{proof}
(i) Denote $\Gamma = \cup_{j=1}^N \Gamma_j$, where $\Gamma_j$ are the $N$ connected components of $\Gamma$. Then the intracellular domain $G_{\rm i}=\cup_{j=1}^N G_{i,j}$ has the same number of connected components. The kernel of $\mathcal{A}$ is ${\rm ker} \mathcal{A} = {\rm span} \{\mathbbm{1}_{\Gamma_1}, \ldots, \mathbbm{1}_{\Gamma_N} \}$, where the characteristic function $\mathbbm{1}_{\Gamma_j}$ is equal to $1$ on $\Gamma_j$ and zero otherwise. 
We start by proving that there exists $\lambda > 0$ such that, for any $v\in H^{1/2}(\Gamma)$,
\begin{align*}
\langle \mathcal{A} v , v \rangle_{H^{-1/2}(\Gamma), H^{1/2}(\Gamma)}
\ge \lambda \inf_{\kappa \in {\rm ker} \mathcal{A}} \|v- \kappa\|_{H^{1/2}(\Gamma)}^2.
\end{align*}
Let $z^{(v)}=(z_i^{(v)},z_e^{(v)})$ be the solution of the transmission
problem \eqref{eq:microscopic-prob_stationary} with jump $q = v$. 
For each connected component $G_{i,j}$, define the average
\[
c_j=
\frac1{|G_{i,j}|}
\int_{G_{i,j}} z_i^{(v)}\,dx,
\]
and set
\[
\kappa=\sum_{j=1}^N c_j\mathbbm{1}_{\Gamma_j}\in {\rm ker} \mathcal{A}.
\]
On $\Gamma_j$ we have
\[
v-\kappa
=
(z_i^{(v)}-c_j)|_{\Gamma_j}
-
z_e^{(v)}|_{\Gamma_j}.
\]
Hence, by the trace theorem,
\[
\|v-\kappa\|_{H^{1/2}(\Gamma)}
\leq
C
\left(
\sum_{j=1}^N
\|z_i^{(v)}-c_j\|_{H^1(G_{i,j})}^2
\right)^{1/2}
+
C\|z_e^{(v)}\|_{H^1(G_{\rm e})}.
\]
By the Poincaré-Wirtinger inequality on each $G_{i,j}$,
\[
\|z_i^{(v)}-c_j\|_{H^1(G_{i,j})}
\leq
C
\|\nabla z_i^{(v)}\|_{L^2(G_{i,j})}.
\]
Moreover, since $z_e^{(v)}=0$ on $\partial G$, the Poincaré inequality on
$G_{\rm e}$ gives
\[
\|z_e^{(v)}\|_{H^1(G_{\rm e})}
\leq
C
\|\nabla z_e^{(v)}\|_{L^2(G_{\rm e})}.
\]
Therefore
\[
\|v-\kappa\|_{H^{1/2}(\Gamma)}
\leq
C
\|\nabla z^{(v)}\|_{L^2(G_{\rm i}\cup G_{\rm e})}.
\]
Since $\kappa\in {\rm ker}\, \mathcal{A}$, it follows that
\[
\inf_{\eta\in {\rm ker} \mathcal{A}}
\|v-\eta\|_{H^{1/2}(\Gamma)}
\leq
C
\|\nabla z^{(v)}\|_{L^2(G_{\rm i}\cup G_{\rm e})}.
\]
Combining this estimate with the ellipticity of $\sigma$, we obtain
\begin{align}
\label{Gaarding}
\langle \mathcal A v,v\rangle_{H^{-1/2}(\Gamma), H^{1/2}(\Gamma)}
\geq
\lambda
\inf_{\kappa\in {\rm ker} \, \mathcal{A}}
\|v-\kappa\|_{H^{1/2}(\Gamma)}^2
\end{align}
for some $\lambda>0$.
The inequality in (i)  holds for any $\alpha, \Lambda > 0$ if $v = 0$ a.e. on $\Gamma$, as $0 \in H^{1/2}(\Gamma)$ belongs to the kernel of $\mathcal{A}$.
Suppose now that (i) does not hold for $v\neq 0$. Then, there exists $\alpha>0$ such that for any $\Lambda$, there exists $v_\Lambda$ satisfying
\[
\langle \mathcal{A} v_\Lambda , v_\Lambda\rangle_{H^{-1/2}(\Gamma), H^{1/2}(\Gamma)} + \alpha \|v_\Lambda\|_{L^2(\Gamma)}^2 < \Lambda \|v_\Lambda\|_{H^{1/2}(\Gamma)}.
\]
Thus, there exists a sequence $\{v_n\}$ such that $\|v_n\|_{H^{1/2}(\Gamma)} = 1$ and
\[
\langle \mathcal{A} v_n , v_n\rangle_{H^{-1/2}(\Gamma), H^{1/2}(\Gamma)} + \alpha \|v_n\|_{L^2(\Gamma)}^2 < \frac{1}{n} \to 0, \quad n \to \infty.
\]
Since $\mathcal{A}$ is non-negative and $\alpha>0$, both terms go to zero:
\begin{align}
\label{eq:coerc_L-1}
\langle \mathcal{A} v_n , v_n\rangle_{H^{-1/2}(\Gamma), H^{1/2}(\Gamma)} \to 0, \quad  \|v_n\|_{L^2(\Gamma)}^2 \to 0, \quad n \to \infty.
\end{align}
Since $\{v_n\}$ bounded in $H^{1/2}(\Gamma)$, by the compact embedding $H^{1/2}(\Gamma) \Subset L^2(\Gamma)$, $v_n$ converges along some subsequence weakly in $H^{1/2}(\Gamma)$ and strongly in $L^2(\Gamma)$ to some $v$. 
By \eqref{Gaarding} and \eqref{eq:coerc_L-1}, there exists a sequence $k_n$ such that $\|v_n- k_n\|_{H^{1/2}(\Gamma)} \to 0$, $n\to \infty$. Since $\|v_n\|_{L^2(\Gamma)} \to 0$, we have $\|k_n\|_{L^2(\Gamma)} \to 0$. The finite number of connected components in $\Gamma$ gives a finite-dimensional ${\rm ker} \mathcal{A}$, and thus we can use the equivalence of norms to get
\begin{align*}
\|k_n\|_{H^{1/2}(\Gamma)} \le C_N \|k_n\|_{L^2(\Gamma)} \to 0, \quad n\to \infty.
\end{align*}
Finally, 
\begin{align*}
\|v_n\|_{H^{1/2}(\Gamma)} \le
\|v_n - k_n\|_{H^{1/2}\Gamma)}
+ \|k_n\|_{H^{1/2}\Gamma)} \to 0, \quad n\to \infty,
\end{align*}
that contradicts the assumption $\|v_n\|_{H^{1/2}(\Gamma)}=1$.\\[5mm]
\medskip
\noindent
(ii) 
By the definition of $\mathcal{A}$ \eqref{def:L} and the positivity of $\sigma$ \ref{assumption_sigma},
\begin{align*}
|\langle \mathcal{A} v , \varphi\rangle_{H^{-1/2}(\Gamma), H^{1/2}(\Gamma)}|
\leq 
C \| \nabla z^{(v)} \|_{L^2(G_{\rm i} \cup G_{\rm e})} \| z^{(\varphi)}\|_{L^2(G_{\rm i} \cup G_{\rm e})}.
\end{align*}
By \eqref{eq:microscopic-prob_stationary}, $z^{(q)}$ for $q= v, \varphi$, minimizes the energy among all admissible functions with a given jump $q$:
\begin{align*}
z^{(q)} = \arg \min_{\zeta \in V^{(q)}} 
\int_{G_{\rm i} \cup G_{\rm e}} \sigma \nabla \zeta \cdot \nabla \zeta \, \dd x,
\end{align*}
where $V^{(q)}=\{\zeta=(\zeta_i, \zeta_e) \in H^1(G_{\rm i} \cap H^1(G_{\rm e})): \,\, [\zeta]_\Gamma = q\}$. Thus, for any $\zeta^{(q)} \in V^{(q)}$, we have
\begin{align*}
\underline{\sigma} \| \nabla z^{(q)} \|_{L^2(G_{\rm i} \cup G_{\rm e})}^2 \le 
\overline{\sigma} \|\nabla \zeta^{(q)} \|_{L^2(G_{\rm i} \cup G_{\rm e})}^2,
\end{align*}
so it is enough to construct one lifting satisfying (ii).
By the trace theorem, for any $q\in H^{1/2}(\Gamma)$, there exists $\tilde{\zeta}^{(q)} \in H^1(G_{\rm i})$ such that 
\[
\|\tilde{\zeta}^{(q)}\|_{H^1(G_{\rm i})} \le C \|q\|_{H^{1/2}(\Gamma)}, \quad 
\]
Taking $\zeta^{(q)}= \tilde{\zeta}^{(q)}$ in $G_{\rm i}$ and zero in $G_{\rm e}$, gives $\zeta^{(q)} \in V^{(q)}$, and estimate (ii) is proved.
\noindent
(iii) is a direct consequence of (ii). Lemma \ref{lm:L} is proved.
\end{proof}

Using the operator $\mathcal{A}$, equations \eqref{eq:microscopic-prob} can be formulated on $\Gamma$:
\begin{align}
\label{eq:micro_deterministic_on_Gamma}
\begin{alignedat}{3}
    c_m \partial_t v + \mathcal{A} v + S_m(v,w)&=-\sigma_{\rm i}\nabla p\cdot \n &\quad &\text{on}\,\, (0,T]\times\Gamma,\\
    \partial_t w &= f(v, w)&\quad & \text{on}\,\, (0,T]\times \Gamma,\\
    v(0,x) &=v_0(x),\quad w(0,x)=w_0(x) &\quad & \text{on}\,\,\Gamma.
\end{alignedat}
\end{align}
The well-posedness of the deterministic problem \eqref{eq:micro_deterministic_on_Gamma} has been proved in \cite{geback2025derivation}. Namely, the following theorem holds:
\begin{theorem}
    Let the assumptions \ref{assumption_sigma}--\ref{assumption:initial_conditions} hold.  Then there exists a unique weak solution $([u], w)$ of problem \eqref{eq:micro_deterministic_on_Gamma} such that 
    \begin{align*}
    &[u] \in C([0,T]; L^2(\Gamma))\cap L^2(0,T; H^{1/2}(\Gamma)), \quad \partial_t [u] \in L^2(0,T; H^{-1/2}(\Gamma)),\\
    &u \in L^2(0,T; H^1(G_{\rm i}\cup G_{\rm e}),\\
    &w \in C([0,T];H^{1/2}(\Gamma)), \quad \partial_t w \in C([0,T]; L^2(\Gamma)), \quad 0\le w \le 1.
    \end{align*}
\end{theorem}

\subsection{Stochastically perturbed electropermeabilization model on the cell scale}

We can now formulate a stochastic version of the electropermeabilization model for a single cell. Let $H= L^2(\Gamma) \times  L^2(\Gamma)$, $V=H^{1/2}(\Gamma) \times  L^2(\Gamma)$ and $Z=H^{1/2}(\Gamma) \times  H^{1/2}(\Gamma)$. Then $Z\subset V\subset H= H^\ast \subset V^\ast$ with continuous and dense embeddings, and we denote $ \langle \cdot, \cdot \rangle_{V^\ast, V}$ the dual pairing between $V$ and $V^\ast$. {Due to the specific structure of the model, the existence of the solutions and their long time behavior will be studied on the following functional spaces:
\begin{align}
\label{def:H_infty}
H_\infty = L^2(\Gamma) \times L^2(\Gamma; [0,1]), \quad V_\infty = H^{1/2}(\Gamma) \times L^2(\Gamma; [0,1]), \quad Z_\infty = H^{1/2}(\Gamma) \times H^{1/2}(\Gamma; [0,1]),
\end{align}
where 
\begin{align*}
L^2(\Gamma; [0,1]) &= \{w\in L^2(\Gamma): \,\, 0\le w \le 1 \,\, \mbox{a.e. on } \, \Gamma\},\\
H^{1/2}(\Gamma; [0,1]) &= \{w\in H^{1/2}(\Gamma): \,\, 0\le w \le 1 \,\, \mbox{a.e. on } \, \Gamma\}.
\end{align*}
Since $L^2(\Gamma; [0,1])$ ($H^{1/2}(\Gamma; [0,1])$) is a strongly closed (and also weakly closed) subset of $L^2(\Gamma)$ ($H^{1/2}(\Gamma)$), the state spaces $H_\infty, V_\infty$, and $Z_\infty$ are Polish spaces with the inherited $H, V$, and $Z$ metrics, respectively.   
}

Let $(\Omega, \F, \F_t, \PP)$ be a complete filtered probability space, $E$ is a given Hilbert space, and $(W_t)_{t\ge 0}$ be a standard $E$-valued $Q$-Wiener process with respect to $(\F_t)_{t\ge 0}$, where $Q\in L(E)$ is nonnegative, symmetric, with finite trace. 
Any such  process has for each $t\in [0,T]$ the following Karhunen–Loève representation (Proposition 2.1.10, \cite{liu2010spde}):
    \begin{align}
    \label{Karhunen-Loeve_kavian}
    W_t= \sum_{k\in \mathbb N} \sqrt{\gamma_k} \beta_k(t) e_k(x),
    \end{align}
    where $\beta_k$ are independent real-valued Brownian motions; $(\gamma_k, e_k)$ are eigenpairs of $Q$ forming an orthonormal basis in $E$ and an orthogonal basis in $E_0:=Q^{1/2}(E)$. The series \eqref{Karhunen-Loeve_kavian} converges in $L^2(\Omega; E)$ uniformly in $t\in[0,T]$, because $Q$ is of a trace class, and thus $\sum_k \gamma_k <\infty$. 

    We incorporate a stochastic perturbation into the deterministic model \eqref{eq:micro_deterministic_on_Gamma} to account for membrane fluctuations and microscopic thermal effects, and consider the following stochastic evolution equation on $\Gamma$:
\begin{equation}
\label{eq:stochastic_kavian}
\begin{aligned}
c_m \dd v &= -(\mathcal{A} v + S_m(v,w) + {\sigma_{\rm i} \nabla p \cdot \n})\, \dd t + b(v)\, \dd W_t&\quad& \text{on}\,\, (0,T]\times \Gamma,\\
\dd w &= f(v,w)\, \dd t&\quad& \text{on}\,\, (0,T]\times \Gamma,\\
v(0,x)&=v_0(x), \quad w(0,x)=w_0(x)&\quad& \text{on}\,\, \Gamma 
\end{aligned}
\end{equation}
where $\mathcal{A}$ is defined by \eqref{def:L}-\eqref{eq:microscopic-prob_stationary} and $p$ solves \eqref{eq:p}. 
Introducing the vector $U=(v,w)$, \eqref{eq:stochastic_kavian} transforms into the following SPDE in the matrix form on $[0,T]\times \Gamma$:
\begin{align}
\label{eq:orig-prob-vector_kavian}
\dd U = (\mathbb{A}(U)+P)\, \dd t + B(U)\, \dd W_t, \quad U(t=0)=U_0.
\end{align}
The operator $\mathbb{A}(\cdot)$ is nonlinear, acting from $V=H^{1/2}(\Gamma) \times  L^2(\Gamma)$ to its dual $V^\ast = H^{-1/2}(\Gamma) \times  L^2(\Gamma)$, and is given by
\begin{align}
\label{def:A_kavian}
\begin{alignedat}{1}
&\mathbb{A}(U) = A \, U + F(U), \quad 
A\,U = 
\begin{pmatrix}
\displaystyle
    -\frac{1}{c_m}(\mathcal{A} + S_0 \mathrm{Id})\, v &0\\ 0&0
\end{pmatrix}, 
\quad 
U_0 = \begin{pmatrix}
    v_0\\w_0
\end{pmatrix},
\\[2mm]
\quad
&F(U)= \begin{pmatrix}
\displaystyle
    -\frac{S_1}{c_m} \, v\,w  \\ f(v,w)
\end{pmatrix}, \quad
P = \begin{pmatrix}
\displaystyle
    - \frac{1}{c_m}\sigma_{\rm i} \nabla p \cdot \n \\ 0
\end{pmatrix}, \quad
B(U) \dd W_t = \begin{pmatrix}
\displaystyle
    \frac{1}{c_m}b(v)\dd W_t \\ 0
\end{pmatrix}.
\end{alignedat}
\end{align}

In what follows we assume that the following hypotheses on the noise term hold.
\begin{enumerate}[label=(B\arabic*)]
\item \label{B1}
$b(\cdot)$ defined on $L^2(\Gamma)$ takes values in the space of Hilbert-Schmidt operators $L_2^0 : = L_{\mathrm HS}(E_0, L^2(\Gamma))$, with the norm denoted $\|\cdot\|_{L_2^0}$, and satisfies the Lipschitz continuity condition:
\begin{align}
\label{eq:Lip-noise_kavian}
    \|b(\xi_1)-b(\xi_2)\|_{L_2^0}^2\le L_{\rm b}\, \|\xi_1-\xi_2\|_{L^2(\Gamma)}^2.
\end{align}
As for the initial condition $U_0$, depending on the situation, we assume that either
\item \label{B2}
$U_0=(v_0, w_0) \in L^2(\Omega, \F_0, \PP; H_\infty)$, or
\item \label{B3}
$U_0=(v_0, w_0) \in L^2(\Omega, \F_0, \PP; V_\infty)$.
\end{enumerate}
Following \cite{pardoux2021stochastic} and \cite{liu2015stochastic}, we adopt the following definition of a variational solution.
\begin{definition}
\label{def:variational}
Given a $\mathcal F_0$-measurable $H$-valued random variable $U_0\in L^2(\Omega, \mathcal F_0, \mathcal P; H)$, we say that a continuous, $\mathcal F_t$-adapted, $H$-valued stochastic process $U(t)$ is a variational solution of \eqref{eq:orig-prob-vector_kavian} if $U\in L^2((0,T)\times \Omega;V)$ and $\PP$-a.s., for every $t\in [0,T]$,
\begin{align}
\label{eq:var_sol_kavian}
U(t)=U_0 + \int_0^t (\mathbb{A}(U(s)) + P(s))\, \dd s + \int_0^t B(U(s))\, \dd W_s.
\end{align}
\end{definition}
Since $U\in L^2(0,T; V)$ $\PP$-a.s., we can understand \eqref{eq:var_sol_kavian} as an equation in $V^\ast$ or in the weak sense
\begin{align*}
(U(t), \varphi)_H=(U_0, \varphi)_H + \int_0^t \langle \mathbb{A}(U(s)), \varphi\rangle_{V^\ast, V}\, \dd s +
\int_0^t \langle P, \varphi\rangle_{V^\ast, V}\, \dd s
+\int_0^t (\varphi, B(U(s))\, \dd W_s)_H.
\end{align*}
By a slight abuse of notation, we use $U$ both for a $\dd t \otimes \PP$ -equivalence class and for its progressively measurable representative $U$, which exists due to Remark 4.2.2 in \cite{liu2015stochastic}. The section below summarizes the main results of the paper.

\subsection{Main results}
\label{sec:main_results}
First, we establish the well-posedness of \eqref{eq:orig-prob-vector_kavian}. 
\begin{theorem}
\label{th:existence-(0,T)_kavian}
    Suppose that \ref{assumption_sigma}-\ref{assumption:g}, \ref{B1}, \ref{B2} be satisfied. Then, for any $T>0$, there exists a unique variational solution $(U(t))_{t\in[0,T]}$ of \eqref{eq:orig-prob-vector_kavian} such that
    \begin{itemize}
    \item[{[i]}]
    The  trajectories of $U$ belong to $L^2(0,T;V_\infty)\cap C([0,T];H_\infty)$ $\PP$-a.s.;
    \item[{[ii]}]
    $ \displaystyle
        \E \Big[\sup_{t\in [0,T]} \|U(t)\|_H^2\Big] + \E \int_0^T \|U(t)\|_V^2\, \dd t 
        \le C(T)\Big(\E \|U_{0}\|_{H}^2 + \sup_{t\in [0,T]}\|g\|_{H^{1/2}(\partial G)}^2 + 1 \Big)$.
    \end{itemize}
\end{theorem}
Next, in Section \ref{sec:Markov_and_Feller}, we prove the Markov and Feller property of the variational solution of \eqref{eq:orig-prob-vector_kavian} (see Lemma \ref{lm:Feller}). It is important to note that for the non-truncated nonlinearity $S_m(v,w)$, we establish the Feller property directly, and the Lipschitz dependence on the initial data does not hold, in contrast to \cite{da2014stochastic}. The existence of an invariant measure in autonomous case, when the boundary excitation $g$, and thus $p$ solving \eqref{eq:p}, does not depend on $t$, is proved in Section \ref{sec:existence_inv_measure} under stronger regularity assumption on the initial data $w_0$ \ref{B3}, which is necessary for obtaining a uniform in time estimate for the time average of $\|w\|_{H^{1/2}(\Gamma)}^2$ (see Lemma \ref{lm:uniform_time_estimate}).  
\begin{theorem}[Existence of invariant measure]
\label{th:existence_inv_measure}
Suppose that \ref{assumption_sigma}--\ref{assumption:g} , \ref{B1}, \ref{B3} hold, and assume that the boundary condition $g$, and thus $p$ solving \eqref{eq:p}, is time-independent. 
Then there exists an invariant measure $\mu \in \mathcal{M}_1(H_\infty)$ for \eqref{eq:orig-prob-vector_kavian}.

Furthermore, if there exists a $\F_0$-measurable random variable $\xi: \Omega \to H_\infty$ with the law $\mathcal{L}(\xi)=\mu$, then the support of the invariant measure $\mu$ lies in $Z_\infty$, that is
\begin{align}
\label{eq:support}
\int_{H_\infty} \|y\|_{Z}^2 \, \mu(\dd y) < \infty.
\end{align}
In addition, if $(U^\ast(t))_{t\ge 0}$ is a solution of \eqref{eq:orig-prob-vector_kavian} such that $U^\ast(0)=\xi$ $\PP$-a.s., then $(U^\ast(t))_{t\ge 0}$ is a stationary solution taking values in $Z_\infty$ with probability one for each $t\ge 0$.
\end{theorem}
\begin{remark}
In general, the existence of a random variable with a given law on a fixed probability space is not guaranteed. However, if $(\Omega, \F_0, \PP)$ is atomless, there exists a uniform random variable $\bar{\xi}: \Omega \to [0,1]$, and thus, by the kernel representation lemma (see Lemma~4.22 in \cite{kallenberg1997foundations}) there exists a random variable $\xi: \Omega \to H_\infty$ with the law $\mu$. Specifically, let $\phi:H_\infty \to [0,1]$ is the Borel isomorphism. Then the random variable $\xi$ can be defined in the following way:
\begin{align*}
    \xi(\omega)= \phi^{-1}(\sup \{\tau \in [0,1]: \,\, \mu(\phi^{-1}([0,\tau])) \le \bar{\xi}(\omega)\}), \quad \omega \in \Omega.
\end{align*}

\end{remark}
\medskip
Typically, the uniqueness of invariant measure is proved using the strict monotonicity (dissipativity) of the operator, which is not satisfied because of the quadratic nonlinearity $S_1 v\, w$. In applications, the electrical potential is bounded, and thus it is reasonable to assume $\sup_t\|v\|_{L^\infty(\Gamma)} \le C$, which, together with an appropriate smallness assumption (see Theorem \ref{th:inv_measure_truncated}), provides the dissipation necessary to prove the uniqueness result. 
Motivated by this observation, we replace $S_m(v,w)$ in \eqref{eq:orig-prob-vector_kavian} by a bounded perturbation of a linear function by setting
\begin{align}
\label{eq:S_M}
S_M(v, w)= S_0 \, v + S_1 w \,T_{[-M,M]}(v), \quad  T_{[-M,M]}(v) = \max\big(-M, \min(M, v)\big),
\end{align}
for some $M>0$. Note that the truncation is $1$-Lipschitz:
\begin{align*}
\|T_{[-M,M]}(v) - T_{[-M,M]}(\varphi)\|_{L^2(\Gamma)}
\le \|v-\varphi\|_{L^2(\Gamma)}.
\end{align*}
The truncated nonlinearity $S_M(v,w)$ is globally Lipschitz with a Lipschitz constant depending on $M$. For any $0\le s\le t\le T$ we study the following stochastic evolution equation on $\Gamma$:
\begin{equation}
\label{eq:stochastic_kavian_truncation}
\begin{aligned}
c_m \dd v &= -(\mathcal{A} v + S_M(v,w) + {\sigma_{\rm i} \nabla p \cdot \n})\, \dd t + b(v)\, \dd W_t&\quad& \text{on}\,\, (s,T]\times \Gamma,\\
\dd w &= f(v,w)\, \dd t, \quad w(s,x)=w_0(x)&\quad& \text{on}\,\, (s,T]\times \Gamma,\\
v(s,x)&=v_0(x), \quad w(s,x)=w_0(x)&\quad& \text{on}\,\, \Gamma, 
\end{aligned}
\end{equation}
where $\mathcal{A}$ as before is defined by \eqref{def:L}, \eqref{eq:microscopic-prob_stationary} and $p$ solves \eqref{eq:p}. 
Denoting $U=(v,w)$, \eqref{eq:stochastic_kavian_truncation} transforms into 
\begin{align}
\label{eq:orig-prob-vector_kavian_truncation}
\dd U = (\mathbb{A}_M(U)+P)\, \dd t + B(U)\, \dd W_t, \quad U(t=s)=U_0.
\end{align}
Here $\mathbb{A}_M(U)=A \, U + F_M(U)$ with $A U, P, B(U)$ defined in \eqref{def:A_kavian} and
\begin{align*}
F_M(U)= \begin{pmatrix}
\displaystyle
    -\frac{S_1}{c_m}\,T_{[-M,M]}(v)\,w \\ f(v,w)
\end{pmatrix}.
\end{align*}
Similar to the original problem \eqref{eq:orig-prob-vector_kavian}, there exists a unique variational solution. For \eqref{eq:orig-prob-vector_kavian_truncation}, we prove the existence and uniqueness of invariant measure under the smallness assumption on the Lipschitz constants $L_{\rm b}, L_\beta$ (see \ref{B1} and \ref{assumption:f_beta}), and $M$ in \eqref{eq:S_M}.
\begin{theorem}
\label{th:inv_measure_truncated}
Suppose the assumptions \ref{assumption_sigma}--\ref{assumption:g}, \ref{B1}, \ref{B2} hold, and the Lipschitz constants $L_b, L_\beta$ in \ref{B1} and \ref{assumption:f_beta}, and the constant $M$ in \eqref{eq:S_M} are sufficiently small. In addition, assume that the boundary condition $g$, and thus $p$ solving \eqref{eq:p}, is independent of $t$. 
Then there exists a unique invariant measure $\mu \in \mathcal{M}_1(H_\infty)$ for \eqref{eq:orig-prob-vector_kavian_truncation}, and it has a finite second moment in $Z$: 
\begin{align*}
\int_{H_\infty} \|y\|_{Z}^2 \, \mu(\dd y) < \infty.
\end{align*}
Moreover, there exists a complete strictly stationary variational solution
$(U^\ast(t))_{t\in\mathbb R}$ of \eqref{eq:orig-prob-vector_kavian_truncation}, and $\mathcal L(U^\ast(t))=\mu$, $t\in\mathbb R$.
More precisely, for every finite interval $[r,T]\subset\mathbb R$, the
restriction $U^\ast|_{[r,T]}$ is a variational solution of \eqref{eq:orig-prob-vector_kavian_truncation} on $[r,T]$ with initial condition $U^\ast(r)$ at time $r$, and 
\[
\sup_{t\in\mathbb R}\mathbb E\|U^\ast(t)\|_Z^2
<\infty.
\]
\end{theorem}

\begin{remark}
    The invariant measure $\mu$ given by Theorem~\ref{th:inv_measure_truncated} is ergodic, that is for any $\varphi\in L^2(H_\infty, \mu)$,
\begin{align*}
    \lim_{T\to +\infty} \frac{1}{T} \int_0^T P_t \varphi\, \dd t
    = \int_{H_\infty} \varphi(y)\, \mu(\dd y) \quad \mbox{in}\,\,  L^2(H_\infty, \mu).
\end{align*}
\end{remark}
\begin{remark}
Note that in contrast to Theorem~\ref{th:existence_inv_measure}, we do not assume that for a given invariant measure $\mu$ there exists a random variable  $\xi$ with the law $\mathcal{L}(\xi)=\mu$. Instead, in the proof of Theorem~\ref{th:inv_measure_truncated}, $\xi=\eta_0$ is constructed explicitly (see \eqref{def_inv_meas}). 
\end{remark}

\section{\texorpdfstring
  {Existence and uniqueness of a variational solution of \eqref{eq:orig-prob-vector_kavian}}
  {Existence and uniqueness of a variational solution}}
\label{sec:proof_Th_existence_solution}
Let us start by proving that the second equation in \eqref{eq:stochastic_kavian}
admits a uniform a priori estimate.
\begin{lemma}
\label{lm:bound-w_m}
Let $v\in L^2(\Omega; C([0,T];L^2(\Gamma)))$ be a continuous $\mathcal{F}_t$-adapted process, and $\beta$ as in \ref{assumption:f_beta}. Assume that $w_0 \in L^2(\Omega \times \Gamma)$ is $\F_0$-adapted and such that $0\le w_0\le 1$.
Then there exists a unique continuous $\mathcal{F}_t$-adapted process $w \in L^2(\Omega; C^1([0,T];L^2(\Gamma))) $ solving 
\begin{align*}
\begin{alignedat}{1}
\partial_t w &= \max\left(\frac{\beta(v) - w}{\tau_{\rm ep}}, \frac{\beta(v) - w}{\tau_{\rm res}}\right),\\
w|_{t=0} &= w_0,
\end{alignedat}
\end{align*}
and satisfying $0\le w \le 1$,  $\PP$-a.s..
\end{lemma}
\begin{proof}
Set $a(t,x)=\beta(v(t,x))$, $\PP$-a.s.. Due to \ref{B1}, $a \in C([0,T]; L^2(\Gamma))$ and $0\le a \le 1$ a.e. on $(0,T)\times \Gamma$. Since 
\[
\tilde{f}(t,w)=\max \left(\frac{a(t)-w}{ \tau_{\rm ep}}\, ,\,
\frac{a(t)-w}{\tau_{\rm res}}\right):[0,T] \times L^2(\Gamma) \to L^2(\Gamma)
\] 
is  continuous in $t$ and Lipschitz continuous in $w$, uniformly in time, 
\begin{align*}
\|\tilde{f}(t,w_1)-\tilde{f}(t,w_2)\|_{L^2(\Gamma)}\le C \|w_1-w_2\|_{L^2(\Gamma)},
\end{align*}
the existence and uniqueness of a solution $w \in C^1([0,T]; L^2(\Gamma))$, $\PP$-a.s., is given by the Cauchy-Lipschitz theorem (Theorem 3.8.1 and the remark at the end of Chapter 3, \cite{ciarlet2025linear}). 
Moreover, since $w$ can be obtained as a limit of the Picard iterations
\begin{align*}
    w^{n+1}(t) = w^n(t) + \int_0^t \tilde{f}(s, w^n(s))\, \dd s,
\end{align*}
$w_0$ is $\F_0$-measurable, and $v$ is $\F_t$-adapted, $w$ is also $v$ is $\F_t$-adapted.
Let us prove that the uniform estimate $0\le w_0\le 1$ is preserved. Denote $w^-(t,x)=\max\{-w, 0\}$ the negative part of $w$ and consider
\begin{align*}
    m(t):=\frac{1}{2}\|w^-\|_{L^2(\Gamma)}^2 = \frac{1}{2}\int_\Gamma |\max(-w, 0)|^2\, \dd S.
\end{align*}
Since $w_0\ge 0$, $m(0)=0$. 
Let us compute the Fréchet derivative of $\Phi(w)=\frac{1}{2}\|w^-\|_{L^2(\Gamma)}^2$. To this end, note that for $r\in \mathbb{R}$
\begin{align*}
\phi(r) = \frac{1}{2}|r^-|^2
= \begin{cases}
    \frac{1}{2}r^2, \,\, r<0\\ 0, \,\, r\ge 0
\end{cases}
\Rightarrow
\phi'(r) = \begin{cases}
    r, \,\, r<0\\ 0, \,\, r\ge 0
\end{cases}
= \min(r,0)= - r^-.
\end{align*}
By the chain rule for integral functionals,
\begin{align*}
D\Phi(w)h = \int_\Gamma \phi'(w) h\, \dd S
= -\int_\Gamma w^- h\, \dd S.
\end{align*}
Since $w \in C^1((0,T); L^2(\Gamma))$ $\PP$-a.s., the chain rule gives
\begin{align*}
m'(t)= D\Phi(w) \partial_t w = - \int_\Gamma w^- \partial_t w \, \dd S
= \int_{\{x: \, w < 0\}} w \tilde{f}(t, w)\, \dd S.
\end{align*}
For $w<0$, $a-w \geq 0$, and we have $m'(t)\le 0$. Since $m(0)=0$ and $m(t) \ge 0$ for all $t$, the non-positive derivative implies that $m(t)=0$ for all $t$. Indeed, since $m(t)$ is absolutely continuous,
\begin{align*}
m(t)=m(0)+\int_0^t m'(s)\, \dd {s} \le 0,
\end{align*}
at the same time as $m(t)\ge 0$ by definition. Thus $m(t)=0$ and, consequently, $w\ge 0$.\\
To prove the upper bound, we define
\begin{align*}
    M(t):=\frac{1}{2}\|(w-1)^+\|_{L^2(\Gamma)}^2 = \frac{1}{2}\int_\Gamma |\max(w-1, 0)|^2\, \dd S.
\end{align*}
Since $w_0\le 1$, $M(0)=0$. By the chain rule,
\begin{align*}
M'(t)
=
\int_\Gamma (w-1)^+\partial_t w\,\dd S
=
\int_{\{x:\,w>1\}} (w-1)\tilde f(t,w)\,\dd S.
\end{align*}
On the set $\{x: \, w>1\}$, we have $a-w \le 0$, and thus $M'(t)\le 0$, which implies
\begin{align*}
    0\le M(t)= M(0) + \int_0^t M'(s)\, \dd s \le 0.
\end{align*}
Hence, $w \le 1$ a.e. on $(0,T)\times \Gamma$.
Finally, $0\le a, w\le 1$ implies $|a-w|\le 1$, thus \begin{align*}
|\partial_t w|\le \max\big(\frac{1}{\tau_{\rm ep}}, \frac{1}{\tau_{\rm res}}\big) \quad \Rightarrow \quad
\sup_{t\in[0,T]} \|\partial_t w\|_{L^2(\Gamma)}^2 \le C.
\end{align*}
This gives $w \in L^2(\Omega; C^1([0,T]; L^2(\Gamma)))$.
\end{proof}

\begin{proof}[Proof of Theorem \ref{th:existence-(0,T)_kavian}]

{
Let us introduce a truncation $T_{[0,1]}: \mathbb{R} \to [0,1]$ defined by
\begin{align*}
    T_{[0,1]}(r) = \max(0 \, , \, \min(1\, , \, r)), \quad r\in \mathbb{R}.
\end{align*}
It is clear that $T_{[0,1]}$ is Lipschitz continuous. We will use the same notation for the associated Nemytskii operator on $L^2(\Gamma)$: $(T_{[0,1]} w)(x) = T_{[0,1]}(w(x))$. \\
We modify the nonlinear term in \eqref{eq:stochastic_kavian} and consider
\begin{equation}
\label{eq:stochastic_kavian_truncation_w}
\begin{aligned}
c_m \dd \bar{v} &= -(\mathcal{A} \bar{v} + S_0\bar{v} + S_1 T_{[0,1]}(\bar{w})\bar{v} + {\sigma_{\rm i} \nabla p \cdot \n})\, \dd t + b(\bar{v})\, \dd W_t
 &\quad& \text{on}\,\, (0,T]\times \Gamma,\\
\dd \bar{w} &= f(\bar{v},\bar{w})\, \dd t &\quad& \text{on}\,\, (0,T]\times \Gamma,\\
\bar{v}(0,x)&=v_0(x), \quad \bar{w}(s,x)=w_0(x)&\quad& \text{on}\,\,  \Gamma.
\end{aligned}
\end{equation}
\begin{lemma}
   Let assumptions \ref{assumption_sigma}--\ref{assumption:g}, \ref{B1}, \ref{B2} hold. A pair $(v,w)$ is a variational solution of \eqref{eq:stochastic_kavian} if and only if it is a solution of \eqref{eq:stochastic_kavian_truncation_w}. 
\end{lemma}
\begin{proof}
Due to Lemma \ref{lm:bound-w_m}, for $0\le w_0\le 1$, the second component $w$ satisfies the same bound. Thus, $(T_{[0,1]} w)(x)=w(x)$ a.e. on $\Gamma$.
\end{proof}
Thus, to prove the well-posedness of \eqref{eq:stochastic_kavian}, it is sufficient to prove the existence and uniqueness of a solution of \eqref{eq:stochastic_kavian_truncation_w}. 
Introducing the vector $U=(v,w)$, \eqref{eq:stochastic_kavian} transforms into the following SPDE in the matrix form:
\begin{align}
\label{eq:orig-prob-vector_kavian_truncation_w}
\dd \bar{U} = (\mathbb{A}_1(\bar{U})+P)\, \dd t + B(\bar{U})\, \dd W_t, \quad \bar{U}(t=0)=U_0.
\end{align}
The operator $\mathbb{A}_1(\cdot)$ acts from $V=H^{1/2}(\Gamma) \times  L^2(\Gamma)$ to $V^\ast = H^{-1/2}(\Gamma) \times  L^2(\Gamma)$, and is given by
\begin{align}
\label{def:A_kavian_truncation_w}
\begin{alignedat}{1}
&\mathbb{A}_1(\bar{U}) = 
\begin{pmatrix}
\displaystyle
    -\frac{1}{c_m}(\mathcal{A} + S_0 \, {\rm Id})(\bar{v}) &0\\ 0&0
\end{pmatrix}
+
\begin{pmatrix}
\displaystyle
    -\frac{S_1}{c_m}\, T_{[0,1]}(\bar{w})\bar{v} \\ f(v,w)
\end{pmatrix}
, 
\quad 
P = \begin{pmatrix}
\displaystyle
    - \frac{1}{c_m}\sigma_{\rm i} \nabla p \cdot \n \\ 0
\end{pmatrix},
\\[2mm]
\quad
&
B(U) \dd W_t = \begin{pmatrix}
\displaystyle
    \frac{1}{c_m}b(v)\dd W_t \\ 0
\end{pmatrix},
\quad
U_0 = \begin{pmatrix}
    v_0\\w_0
\end{pmatrix}.
\end{alignedat}
\end{align}
Let us prove that operators $\mathbb{A}_1, B$ satisfy assumptions (H1)--(H4) in Section 5.1, \cite{liu2015stochastic}. 
}
\begin{lemma}[Properties of the operator $\mathbb{A}_1$]
\label{lm:properties-A_kavian}
Let hypotheses \ref{assumption_sigma}, \ref{assumption:f_beta} hold. The nonlinear operator $\mathbb A_1$ given by \eqref{def:A_kavian_truncation_w} satisfies the following estimates:
\begin{enumerate}
\item[(i)] \textbf{Hemicontinuity}\\ For any $U_1, U_2, U \in V$, the map $s \mapsto \langle \mathbb{A}_1(U_1+s U_2), U\rangle_{V^\ast, V}$ is continuous on $\mathbb{R}$.
\item[(ii)] \textbf{Local monotonicity}\\ For any $U_1, U_2 \in V$, there exists a constant $K\ge 0$: 
    \begin{align*}
        \langle \mathbb{A}_1(U_1)-\mathbb{A}_1(U_2), U_1-U_2\rangle_{V^\ast, V}  \le K(1 + \|U_2\|_V^2) \, \|U_1-U_2\|_H^2.
    \end{align*}
\item[(iii)] \textbf{Generalized coercivity}\\ For any $U=(v,w)\in V$, there exist $C_0,C_1, C_2 >0$ such that
    \begin{align*}
    \langle \mathbb{A}_1(U), U\rangle_{V^\ast, V} &\le -C_0\|U\|_V^2
    +C_1 \|U\|_H^2 + C_2.
    \end{align*}
\item[(iv)] \textbf{Generalized growth condition}\\
For any $U=(v,w)\in V$, there exists a constant $K>0$ such that
\begin{align*}
    \|\mathbb A_1(U)\|_{V^\ast} 
    \le
    K(1+\|w\|_{L^2(\Gamma)})(1+\|v\|_{H^{1/2}(\Gamma)}).
\end{align*}
\end{enumerate} 
\end{lemma}
\begin{proof}[Proof of Lemma \ref{lm:properties-A_kavian}]
$ $\\
(i) The operator $\mathbb{A}_1$ is hemicontinuous since the dependence on $s$ is polynomial in the first equation for $v$ and Lipschitz in the equation for $w$. 

\noindent
(ii) Take arbitrary $U_1=(v_1, w_1), U_2=(v_2, w_2) \in V$ and denote $\delta v :=v_1-v_2$, $\delta w:= w_1-w_2$. By the definition of $\mathbb{A}_1$ \eqref{def:A_kavian}, we have
\begin{align}
\label{est:loc-monotonicity_kavian}
\langle \mathbb{A}_1(U_1) - \mathbb{A}_1(U_2), U_1-U_2\rangle_{V^\ast, V} \nonumber
=& -\frac{1}{c_m}\langle (\mathcal{A}+S_0\, {\rm Id})\,\delta v , \delta v \rangle_{H^{-1/2}(\Gamma), H^{1/2}(\Gamma)} \nonumber \\
&- \frac{S_1}{c_m} \big(T_{[0,1]}(w_1) v_1 - T_{[0,1]}(w_2) v_2), \delta v)_{L^2(\Gamma)} \\
&+ \big(f(v_1,w_1) - f(v_2,w_2), \delta w \big)_{L^2(\Gamma)} =: I_1 + I_2 + I_3. \nonumber
\end{align}
The first term $I_1$ is estimated using \eqref{Gaarding}:
\begin{align}
\label{est:I_1_kavian}
    I_1 \le - \Lambda \|\delta v\|_{H^{1/2}(\Gamma)}^2 .
\end{align}
Adding and subtracting $v_2 T_{[0,1]}(w_1)$, using the embedding $H^{1/2}(\Gamma) \hookrightarrow L^4(\Gamma)$ and Young's inequality with a parameter $\gamma>0$, we estimate $I_2$ in the following way:
\begin{align}
\label{est:I_2_kavian}
\begin{alignedat}{1}
I_2 &= -\frac{S_1}{c_m} \big(v_2 \, (T_{[0,1]}(w_1) v_1 - T_{[0,1]}(w_2)), \delta v)_{L^2(\Gamma)}  + 
\frac{S_1}{c_m} \big( T_{[0,1]}(w_1) \delta v, \delta v)_{L^2(\Gamma)}\\
& \le
\frac{S_1 \gamma}{2c_m} \|\delta v\|_{H^{1/2}(\Gamma)}
+ C_\gamma \|v_2\|_{H^{1/2}(\Gamma)} \|\delta w\|_{L^2(\Gamma)}^2
+ \frac{S_1}{c_m}\|\delta v\|_{L^2(\Gamma)}^2.
\end{alignedat}
\end{align}
To estimate $I_3$, we note that the function
\begin{align*}
    h(r)= \max \big(\frac{r}{\tau_{\rm ep}} , \frac{r}{\tau_{\rm res}}\big), \quad r\in \mathbb{R}
\end{align*}
is monotone, Lipschitz, and satisfies
\begin{align}
    \label{f_monotonicity}
    \min(\frac{1}{\tau_{\rm ep}}, \frac{1}{\tau_{\rm res}}) |a-b|^2
    \le
    (h(a)-h(b))(a-b) \le \max(\frac{1}{\tau_{\rm ep}}, \frac{1}{\tau_{\rm res}}) |a-b|^2, \quad a,b\in \mathbb{R}.
\end{align}
Thus, adding and subtracting $\delta \beta := \beta(v_1)-\beta(v_2)$, using \eqref{f_monotonicity}, and Lipschitz continuity of $\beta$, $I_3$ in \eqref{est:loc-monotonicity_kavian} is estimated as follows:
\begin{align}
    \label{est:I_3_kavian}
    I_3 \le - \frac{1}{4\tau_{\rm res}} \|\delta w\|_{L^2(G)}^2  + C \|\delta v\|_{L^2(G)}^2.
\end{align}
Finally, combining estimates \eqref{est:I_1_kavian}, \eqref{est:I_2_kavian}, and \eqref{est:I_3_kavian}, and choosing $\gamma$ such that $\Lambda - \frac{S_1 \gamma}{2c_m} > 0$ yields the local monotonicity (ii).

\medskip
\noindent
(iii) Take arbitrary $U=(v,w) \in V$. Using the coercivity estimate \eqref{Gaarding} and the monotonicity \eqref{f_monotonicity}, we obtain 
\begin{align*}
    \langle \mathbb A_1(U), U\rangle_{V^\ast, V} &
    = -\frac{1}{c_m} \langle (\mathcal{A}+S_0\, {\rm Id})\, v , v \rangle_{H^{-1/2}(\Gamma), H^{1/2}(\Gamma)} \nonumber 
- \frac{S_1}{c_m} \big(T_{[0,1]}(w) v, v)_{L^2(\Gamma)} \\
&+ \big(f(v,w), w \big)_{L^2(\Gamma)}\\
    &\le -\frac{\Lambda}{c_m} \|v\|_{H^{1/2}(\Gamma)}^2 
    - \frac{1}{4 \tau_{\rm res}}\|w\|_{L^2(\Gamma)}^2 
    + \frac{S_1}{c_m}\|v\|_{L^2(\Gamma)}^2 +C(1+ \|v\|_{L^2(\Gamma)}^2\\
    & \le -C_0 \|U\|_V^2 + C_1 \|U\|_H^2 + C_2.
\end{align*}

\noindent
(iv) To check the growth condition, we take arbitrary $U=(v, w), \Phi=(\varphi, \psi)\in V$ and estimate
\begin{align*}
\|\mathbb A_1(U)\|_{V^\ast} = \sup_{\Phi \in V\setminus\{0\}}\Big\{ \langle \mathbb A_1(U), \Phi\rangle_{V^\ast, V}: \,\, \|\Phi\|_V=1\Big\}.
\end{align*}
By Lemma \ref{lm:L} and \eqref{f_monotonicity},
\begin{align*}
 \langle \mathbb A_1(U), \Phi\rangle_{V^\ast, V}
 &= -\frac{1}{c_m}\langle (\mathcal{A}+S_0{\rm Id}) v, \varphi \rangle_{H^{-1/2}(\Gamma), H^{1/2}(\Gamma)}
 -\frac{S_1}{c_m}(T_{[0,1]}(w)\, v , \varphi)_{L^2(\Gamma)}
 + (f(v,w), \psi)_{L^2(\Gamma)}\\
 &\le C\|v\|_{H^{1/2}(\Gamma)}\|\varphi\|_{H^{1/2}(\Gamma)} 
 + \frac{S_0+S_1}{c_m}\|v\|_{L^2(\Gamma)}\|\varphi\|_{L^2(\Gamma)}
 + C(1+\|w\|_{L^2(\Gamma)})\|\psi\|_{L^2(\Gamma)}.
\end{align*}
Then
\begin{align*}
\langle \mathbb A_1(U), \Phi\rangle_{V^\ast, V}
 & \le C(1+\|U\|_V)(\|\varphi\|_{H^{1/2}(\Gamma)} + \|\psi\|_{L^2(\Gamma)}),
\end{align*}
which together with the normalization $\|\varphi\|_{H^{1/2}(\Gamma)}^2 + \|\psi\|_{L^2(\Gamma)}^2=1$ yields the desired estimate. Lemma \ref{lm:properties-A_kavian} is proved.
\end{proof}
By Theorem 5.1.3 and Lemma 5.1.5 in \cite{liu2015stochastic}, for any $U_0 \in L^2(\Omega, \F, \PP; H)$, there exists a variational solution $(\bar{U}(t))_{t\in [0,T]}$ of \eqref{eq:orig-prob-vector_kavian_truncation_w} satisfying (ii) in Theorem \ref{th:existence-(0,T)_kavian}.
Moreover, if the external excitation $g=g(x)$ so that the system \eqref{eq:orig-prob-vector_kavian_truncation_w} is autonomous, the solution $U(t)$ is a Markov process.

In view of Lemma \ref{lm:bound-w_m}, the solution $\bar{U}$ is also a solution of the non-truncated system \eqref{eq:orig-prob-vector_kavian}. The well-posedness of \eqref{eq:orig-prob-vector_kavian} is proved.

\end{proof}

\section{Properties of the variational solution}
\label{sec:properties_of_solution}
\subsection{Uniform in time estimates}

In this section, we will derive uniform in time estimates for the $\E\|U\|_H^2$ and the time averages of $\E\|U\|_V^2$ under the assumption that the Lipschitz constant in \ref{B1} is sufficiently small. In contrast to \cite{liu2015stochastic}, where the lack of estimates in stronger norms prevented from establishing the existence of invariant measure unless ${\rm dim}(H) < \infty$, we derive such estimates thanks to the special structure of the coupled system. Note that \eqref{eq:better-estimate-w_H^1/2} is derived under the assumption $w_0 \in H^{1/2}(\Gamma)$. Indeed, the regularity of $w$ in is inherited from the initial data and the regularity of $v$.
\begin{lemma}[Improved estimates]
\label{lm:uniform_time_estimate}
Let the hypotheses \ref{assumption_sigma}-\ref{assumption:f_beta}, \ref{B1}-\ref{B2} be satisfied. Assume in addition that the Lipschitz constant in \ref{B1} for some positive $\kappa>0$ satisfies the condition
\begin{align}
\label{eq:small_Lip_noise}
    c_m \Lambda - L_{\rm b} \ge 2\kappa >0,
\end{align}
where $\Lambda$ is the coercivity constant in \eqref{Gaarding}.
Then there exists $\lambda=\lambda(\Lambda, c_m, L_{\rm b}) >0$ and a constant $C>0$ independent of $t$ such that, for any $t\ge 0$,
\begin{align}
\label{eq:better-estimate-v_L^2}
\E \|v(t)\|_{L^2(\Gamma)}^2
\le 
& e^{-\lambda t} \, \E \|v_0\|_{L^2(\Gamma)}^2 
+\frac{C}{\lambda} \, \Big(\|b(0)\|_{L_2^0}^2+ \sup_{t\ge 0} \|g\|_{H^{1/2}(\partial G)}^2\Big),\\
\label{eq:better-estimate-v_H^1/2}
\E\int_0^t \|v(s)\|_{H^{1/2}(\Gamma)}^2\, \dd s
\le 
& \, \E \|v_0\|_{L^2(\Gamma)}^2 
+C \, t \, \Big(\|b(0)\|_{L_2^0}^2+ \sup_{t\ge 0} \|g\|_{H^{1/2}(\partial G)}^2\Big),\\
\label{eq:better-estimate-w_L^2}
\E \|W_t\|_{L^2(\Gamma)}^2
\le 
& \, e^{-\frac{t}{\tau_{\rm res}}} \, \E \|w_0\|_{L^2(\Gamma)}^2 
+C,\\
\label{eq:better-estimate-w_H^1/2}
{\E\int_0^t \|W_s\|_{H^{1/2}(\Gamma)}^2\, \dd s
\le }
& {\, 
\max(1,2\tau_{\rm res}) \E \|w_0\|_{H^{1/2}(\Gamma)}^2
+ C\, t \,\Big(\|b(0)\|_{L_2^0}^2+ \sup_{t\ge 0} \|g\|_{H^{1/2}(\partial G)}^2\Big).}
\end{align}
\end{lemma}
\begin{remark}
As it can be seen from the proof of \eqref{eq:better-estimate-v_L^2}, in the case of additive or bounded $\|b\|_{L_2^0} \le C$ noise, no smallness assumption is needed. 
\end{remark}
\begin{proof}

We consider the two equations in \eqref{eq:stochastic_kavian} for $v$ and $w$ separately.

The product rule combined with the Itô formula for the first component, after taking the expectation, gives 
\begin{align}
\label{eq:E-of-norm_kavian-aux}
\begin{alignedat}{2}
\E\Big[e^{\lambda t} \|v(t)\|_{L^2(\Gamma)}^2\Big]  =& \E \|v_0\|_{L^2(\Gamma)}^2 
- \frac{2}{c_m} \int_0^t e^{\lambda s} \E \Big[ \langle ( \mathcal{A}+S_0{\rm Id}) v, v \rangle_{H^{-1/2}(\Gamma), H^{1/2}(\Gamma)} + (S_1 v w, v)_{L^2(\Gamma)}\Big]\, \dd s\\
& - \frac{2}{c_m} \int_0^t e^{\lambda s} \E \Big[ \langle  \sigma_{\rm i} \nabla p \cdot \n, v \rangle_{H^{-1/2}(\Gamma), H^{1/2}(\Gamma)} \Big] \, \dd s
 + \frac{1}{c_m^2}\int_0^{t} e^{\lambda s} \E \Big[ \| b(v)\|_{L_2^0}^2 \Big] \dd s\\
 & + \lambda \int_0^t e^{\lambda s} \E \|v(s)\|_{L^2(\Gamma)}^2\, \dd s.
\end{alignedat}
\end{align}
For brevity, we omit the argument $s$ under the integrals in time and write, for example, $\int_0^{t} e^{\lambda s} \E \Big[ \| b(v)\|_{L_2^0}^2 \Big] \dd s$ for $\int_0^{t} e^{\lambda s} \E \Big[ \| b(v(s, \cdot))\|_{L_2^0}^2 \Big] \dd s$.   
The constant $\lambda>0$ will be chosen later.
Using the coercivity estimate \eqref{Gaarding}, the boundedness $0\le w\le 1$, the estimate $\sup_t\|\sigma_{\rm i} \nabla p \cdot \n\|_{H^{-1/2}(\Gamma)} \le C\sup_t \|g\|_{H^{1/2}(\partial G)}$, and the Lipschitz property of $b(v)$, we arrive at
\begin{align*}
\begin{alignedat}{1}
\E\Big[ e^{\lambda t} \|v(t)\|_{L^2(\Gamma)}^2\Big]  \le & \E \|v_0\|_{L^2(\Gamma)}^2 
- \frac{2 \Lambda}{c_m} \int_0^t e^{\lambda s}\E  \|v(s)\|_{H^{1/2}(\Gamma)}^2\, \dd s  
\\
& + \frac{2C}{\delta c_m} \int_0^t e^{\lambda s}\sup_{0\le s\le t} \|g\|_{H^{1/2}(\partial G)}^2 \, \dd s
+ \frac{2\delta}{c_m} \int_0^t e^{\lambda s}\E\|v(s)\|_{H^{1/2}(\Gamma)}^2 \, \dd s\\
 &+ \frac{L_{\rm b}}{c_m^2}\int_0^{t} e^{\lambda s}\E \Big[ \|b(0)\|_{L_2^0}^2+ \| v(s)\|_{L^2(\Gamma)}^2 \Big] \dd s
 + \lambda \int_0^t e^{\lambda s} \E \|v(s)\|_{L^2(\Gamma)}^2\, \dd s.
\end{alignedat}
\end{align*}
We choose $\delta>0$ such that $2(\Lambda-\delta)/c_m\ge \Lambda/c_m=: \kappa_0>0$. For sufficiently small $L_{\rm b}>0$, namely under the assumption 
\begin{align*}
2 c_m \Lambda - L_{\rm b} \ge 2 \kappa > 0,
\end{align*}
we can choose $\lambda$ such that 
\begin{align*}
    \frac{2\Lambda}{c_m} - \frac{L_{\rm b}}{c_m^2} - \lambda \ge \kappa >0.
\end{align*}
Rearranging the terms, we have
\begin{align}
\label{eq:better-estimate-aux1}
\begin{alignedat}{1}
\E \Big[e^{\lambda t}\|v(t)\|_{L^2(\Gamma)}^2\Big]  \le & \E \|v_0\|_{L^2(\Gamma)}^2 
- \kappa_0\int_0^t e^{\lambda s}\E  \|v(s)\|_{H^{1/2}(\Gamma)}^2\, \dd s\\  
&- \kappa \int_0^t e^{\lambda s}\E \|v(s)\|_{L^2(\Gamma)}^2\, \dd s +C \, \Big(\|b(0)\|_{L_2^0}^2+ \sup_{t\ge 0} \|g\|_{H^{1/2}(\partial G)}^2\Big)\int_0^t e^{\lambda s}\, \dd s, 
\end{alignedat}
\end{align}
for $\kappa_0, \kappa>0$. 
Note that
\begin{align*}
e^{-\lambda t} \int_0^t e^{\lambda s}\, \dd s
= \frac{1}{\lambda}(1-e^{-\lambda t}) \le \frac{1}{\lambda}, \quad t\ge 0.
\end{align*}
Thus, in particular, \eqref{eq:better-estimate-aux1} implies \eqref{eq:better-estimate-v_L^2}. Repeating the estimates for \eqref{eq:E-of-norm_kavian-aux} with $\lambda=0$, we obtain \eqref{eq:better-estimate-v_H^1/2}. Note that applying the Itô formula for $\|v\|_{L^2(\Gamma)}^2$ without the product rule with $e^{\lambda t}$ does not lead to a uniform in time estimate. Specifically, the last term in the first estimate in \eqref{eq:better-estimate-aux1} without the exponential factor would give $C\, t \, (1+\sup_{t\ge 0} \|g\|_{H^{1/2}(\partial G)}^2)$ which blows up as $t\to \infty$. 
In contrast to the pointwise in time estimate for the $L^2(\Gamma)$-norm, the estimate for the $H^{1/2}(\Gamma)$-norm is integral in time, and grows linearly in $t$.

Next, we turn to $w$. Clearly, since if $0\le w_0 \le 1$, then by Lemma \ref{lm:bound-w_m}, $0\le w \le 1$, and we have a uniform in time estimate for $\E \|w\|_{L^2(\Gamma)}^2$. We can, in addition, specify the dependence on the initial condition, similar to $v$ above. Multiplying \eqref{eq:f} by $w$ and integrating over $\Gamma$ yields, $\PP$-a.s.,
\begin{align}
\label{eq:est-w-aux1}
\frac{1}{2}\frac{\dd}{\dd t} \|w\|_{L^2(\Gamma)}^2
= (w-\beta(v)\, , \, f(v,w))_{L^2(\Gamma)}
+ (\beta(v)\,,  f(v,w))_{L^2(\Gamma)}.
\end{align}
Since $\tau_{\rm ep} < \tau_{\rm res}$,
\begin{align*}
(w-\beta(v)\, , \, f(v,w))_{L^2(\Gamma)}
\le -\frac{1}{\tau_{\rm res}} \|w-\beta(v)\|_{L^2(\Gamma)}^2.
\end{align*}
Using the Young inequality with a parameter in \eqref{eq:est-w-aux1}, we have
\begin{align*}
\frac{1}{2}\frac{\dd}{\dd t} \|w\|_{L^2(\Gamma)}^2
\le -\frac{1}{\tau_{\rm res}}\|w-\beta(v)\|_{L^2(\Gamma)}^2
+ \frac{\delta}{2\tau_{\rm ep}}\|w-\beta(v)\|_{L^2(\Gamma)}^2
+ \frac{1}{2\delta \tau_{\rm ep}}|\Gamma|.
\end{align*}
Choosing $\delta=\tau_{\rm ep}/2\tau_{\rm res}$ such that $\frac{1}{\tau_{\rm res}} - \frac{\delta}{2\tau_{\rm ep}} = \frac{3}{4\tau_{\rm res}}>0$ and using that $0\le \beta \le 1$, we obtain
\begin{align}
\label{eq:est-w-aux3}
\begin{alignedat}{1}
\frac{1}{2}\frac{\dd}{\dd t} \|w\|_{L^2(\Gamma)}^2
& \le
 -\frac{3}{4\tau_{\rm res}}\|w-\beta(v)\|_{L^2(\Gamma)}^2
+ \frac{1}{2\delta \tau_{\rm ep}}|\Gamma|\\
& = -\frac{3}{4\tau_{\rm res}}\|w\|_{L^2(\Gamma)}^2
+ \frac{3}{2\tau_{\rm res}}\int_\Gamma w\,\beta(v)\, \dd S -\frac{3}{4\tau_{\rm res}}\|\beta(v)\|_{L^2(\Gamma)}^2 +
\frac{1}{2\delta \tau_{\rm ep}}|\Gamma|\\
& \le -\frac{1}{2\tau_{\rm res}}\|w\|_{L^2(\Gamma)}^2
+ C.
\end{alignedat}
\end{align}
Thus, multiplying both sides of \eqref{eq:est-w-aux3} by $\exp(t/\tau_{\rm res})$ and integrating in time, yields
\begin{align*}
\|w\|_{L^2(\Gamma)}^2
\le& e^{-\frac{t}{\tau_{\rm res}}}\|w_0\|_{L^2(\Gamma)}^2
+ C,
\end{align*}
and we obtain \eqref{eq:better-estimate-w_L^2}. \\
Finally, let us prove \eqref{eq:better-estimate-w_H^1/2}. This estimate is more delicate since the dependence on the spatial variable $x$ comes in $w$ through $v$.  We first prove the estimate \eqref{eq:better-estimate-w_H^1/2} for smooth $v, w$ by estimating the Sobolev-Slobodeckij seminorm, $|w|_{H^{1/2}(\Gamma)} = \int_\Gamma \int_\Gamma \frac{|w(t,x)-w(t,y)|^2}{|x-y|^3}\dd S_x \dd S_y$, \cite{dinezza2012hitchhiker}. Denote
\begin{align*}
\delta w = w(t, x) - w(t, y), \quad x,y\in \Gamma.
\end{align*}
Then $\delta w$ satisfies the equation
\begin{align}
\label{eq:aux_est_w-1}
\partial_t \delta w = f(v(t,x), w(t,x)) - f(v(t,y), w(t,y)).
\end{align}
Multiplying \eqref{eq:aux_est_w-1} by $\delta w$ and writing $\delta w = \delta \beta - (\delta \beta - \delta w)$ with 
\begin{align*}
\delta \beta = \beta(v(t,x)) - \beta(v(t,y)), 
\end{align*}
we obtain
\begin{align*}
\frac{1}{2}\partial_t |\delta w|^2
& = (f(v(t,x), w(t,x)) - f(v(t,y), w(t,y))) \delta \beta\\ 
& \quad - (f(v(t,x), w(t,x)) - f(v(t,y), w(t,y))) (\delta \beta - \delta w)\\
& \le 
C|\delta \beta - \delta w|\, |\delta \beta|
-\frac{1}{\tau_{\rm res}} |\delta \beta - \delta w|^2 \\
& \le
C |\delta \beta|^2 - \frac{1}{2\tau_{\rm res}} |\delta \beta - \delta w|^2.
\end{align*}
Using the inequality $|a-b|^2\ge \frac{1}{2}|b|^2 - |a|^2$ and rearranging the terms, we have
\begin{align*}
\partial_t |\delta w|^2 + 
\frac{1}{{2}\tau_{\rm res}} |\delta w|^2
& \le
C |\delta \beta|^2.
\end{align*}
Dividing by the $|x-y|^3$ and using the Lipschitz continuity of $\beta$, we obtain the following estimate for the seminorm:
\begin{align*}
\partial_t |w|_{H^{1/2}(\Gamma)}^2 + \frac{1}{2 \tau_{\rm res}} |w|_{H^{1/2}(\Gamma)}^2 \le 
C |v|_{H^{1/2}(\Gamma)}^2.
\end{align*}
Multiplying both sides with $\exp(t/2 \tau_{\rm res})$, integrating over $(0,t)$, and changing order of integration in the second integral, we get, $\PP$-a.s.
\begin{align}
\label{eq:aux_est_w-3}
\begin{alignedat}{2}
\int_0^T |w(t,\cdot)|_{H^{1/2}(\Gamma)}^2\, \dd t 
\le & 
\int_0^T e^{-\frac{t}{2\tau_{\rm res}}} |w_0|_{H^{1/2}(\Gamma)}^2\, \dd t
+
C \int_0^T \int_0^t
e^{-\frac{(t-s)}{2\tau_{\rm res}}}
|v(s,\cdot)|_{H^{1/2}(\Gamma)}^2\, \dd s\, \dd t\\
&\le
2\tau_{\rm res}|w_0|_{H^{1/2}(\Gamma)}^2
+
C \int_0^T |v(s,\cdot)|_{H^{1/2}(\Gamma)}^2 \int_s^T
e^{-\frac{(t-s)}{2\tau_{\rm res}}}
\, \dd t\, \dd s\\
&
\le
2\tau_{\rm res}|w_0|_{H^{1/2}(\Gamma)}^2
+
C \int_0^T |v(s,\cdot)|_{H^{1/2}(\Gamma)}^2\, \dd s.
\end{alignedat}
\end{align}
Estimate for the seminorm \eqref{eq:aux_est_w-3} is obtained for smooth functions $v, w$. Approximating functions $v, w\in L^2(0,T; H^{1/2}(\Gamma))$ pathwise by smooth functions, passing to the limit in \eqref{eq:aux_est_w-3}, and using semicontinuity of the $H^{1/2}(\Gamma)$-norm, estimate \eqref{eq:aux_est_w-3} is justified for $v, w\in L^2(0,T; H^{1/2}(\Gamma))$, $\PP$-a.s.
Finally, combining \eqref{eq:aux_est_w-3} with \eqref{eq:better-estimate-w_L^2}, \eqref{eq:better-estimate-v_H^1/2}, and taking the expectation yields \eqref{eq:better-estimate-w_H^1/2}.
Lemma \ref{lm:uniform_time_estimate} is proved.

\end{proof}


\subsection{Markov and Feller property of variational solutions}
\label{sec:Markov_and_Feller}
Let $s\in[0,T]$ and $\xi \in L^2(\Omega, \F_s, \PP; H_\infty)$. Consider a variational solution of \eqref{eq:orig-prob-vector_kavian} starting at time $s$ with the initial value $\xi:=U(s)$,
\begin{align}
\label{eq:var_sol_s_xi}
U(t)=\xi + \int_s^t (\mathbb{A}(U(r)) + P(r))\, \dd r + \int_s^t B(U(r))\, \dd W_r, \quad t\in[s,T].
\end{align}
The results of Theorem \ref{th:existence-(0,T)_kavian} proved for $s=0$ carry over to this case, and we obtain a unique variational solution which we denote $U(t, s, \xi)$. Then for $0\le r\le s\le T$, the uniqueness implies that, $\PP$-a.s.,
\begin{align*}
    U(t,r,\xi) = U(t, s, U(s, r, \xi)), \quad \mbox{for each }t\in [s,T].
\end{align*}


The goal of this section is to prove that the variational solution defined by \eqref{eq:var_sol_s_xi} satisfies the Markov property and that the corresponding transition semigroup is Feller. The Feller property is a statement about the stability of the dynamics with respect to the initial condition. Because of the quadratic nonlinear term $S_m(v,w)=(S_0+S_1 w) v$, we cannot prove the Lipschitz dependence on the initial data, as it is done in Theorem 9.1 in \cite{da2014stochastic}, which  implies immediately the continuous dependence on the initial data, or Lemma 4.3.11 in \cite{liu2015stochastic} for variational solutions. Note that the truncation $T_{[0,1]}(w)$ used in Theorem \ref{th:existence-(0,T)_kavian} does not yield Lipschitz dependence on the initial condition, but the truncation $T_{[-M, M]}(v)$ used in Section \ref{sec:inv_measure_truncated_reaction} does. For the non-truncated $S_m(v,w)$, we will use a localization trick by introducing a suitable stopping time $\tau_R$ (see \eqref{eq:tau_R_Feller}) and deriving stability estimates up to $\tau_R$.  

For any bounded Borel function $\varphi\in B_b(H_\infty)$ with $H_\infty$ defined in \eqref{def:H_infty}, any $0\le s\le t\le T$, and initial condition $\xi \in H_\infty$, we define the two-parameter transition operators
\begin{align}
\label{def:P_s,t}
        P_{s,t}\varphi(\xi) = \E\big[\varphi(U(t, s, \xi))\big].
\end{align}
In the case of autonomous coefficients in \eqref{eq:stochastic_kavian} it holds $P_{s,t}\varphi(\xi)=P_{t-s}\varphi(\xi)$, and {we can work with the operators
\begin{align}
\label{P_t}
    P_t \varphi(\xi) = P_{0,t}\varphi(\xi)= \E\big[\varphi(U(t, 0, \xi))\big].
\end{align}
This is true for the constant in time excitation $g$ on $\partial \Omega$, but clearly not for an arbitrary time-dependent $g(t,x)$. 
Similarly to Theorem 9.14 in \cite{da2014stochastic}, one can show that the solution $U(t,s,\xi)$ is Markov, that is for arbitrary $\varphi:H_\infty\to \mathbb{R}$ measurable and bounded function, and $0 \le s\le r\le t\le T$ we have
\begin{align*}
    \E\big[\varphi(U(t, s, \xi)) \, | \, \F_r \big]
    = P_{r,t} \varphi (U(r, s, \xi)), \quad \PP-\text{a.s.}
\end{align*}
As a corollary of the Markov property, we obtain the semigroup property for the transition operators: $P_{s,t}=P_{s,r}P_{r,t}$:
\begin{align*}
P_{s,t}\varphi (\xi)
= \E\big[\varphi(U(t, s, \xi))\big]
= \E \Big[\E\big[\varphi(U(t, s, \xi)) \, |\, \F_r\big]\Big]
= \E\big[P_{r,t} \varphi(U(r, s, \xi))\big]
= P_{s,r} \big(P_{r,t} \varphi\big)(\xi).
\end{align*}
Let us prove that the two-parameter transition semigroup satisfies the Feller property in the strong topology of $H_\infty$ inherited from $H=L^2(\Gamma) \times L^2(\Gamma)$.
\begin{lemma}[Feller property]
\label{lm:Feller}
Assume that the hypotheses \ref{assumption_sigma}-\ref{assumption:g}, \ref{B1}-\ref{B2} hold. In addition, assume that condition \eqref{eq:small_Lip_noise} is satisfied.  
Then the two-parameter transition semigroup $P_{s,t}$ defined in \eqref{def:P_s,t} satisfies the Feller property in the strong topology of $H$, that is for every $\varphi \in C_b(H_\infty)$, $P_{s,t}\varphi: H_\infty\to \mathbb{R}$ is continuous. 
\end{lemma}
\begin{proof}[Proof of Lemma \ref{lm:Feller}]
Fix $0\le t_0\le t\le T$. To prove the continuity of $P_{t_0,t}$, it suffices to prove that $\xi_n \to \xi$ in $H$ implies $U(t, t_0, \xi_n) \to U(t,t_0,\xi)$ in probability in $H$, that is for every $0\le t_0 \le t \le T$ and every $\ve >0$,
\begin{align}
\label{eq:Feller_aux1}
\PP (\|U(t, t_0, \xi_n) - U(t, t_0, \xi)\|_H > \ve) \to 0, \quad n\to \infty.
\end{align}
Indeed, \eqref{eq:Feller_aux1} implies that for arbitrary $\varphi \in C_b(H_\infty)$, $\varphi(U(t, t_0, \xi_n)) \to \varphi(U(t,t_0,\xi))$ in probability. Then, since $\varphi(U(t, t_0, \xi_n))\le \|\varphi\|_{L^\infty(H)}$ is bounded, the sequence $\varphi(U(t, t_0, \xi_n))$ is uniformly integrable and thus we have convergence in $L^1(\Omega)$:
\begin{align*}
\E\big[|\varphi(U(t, t_0, \xi_n)) - \varphi(U(t, t_0, \xi))|\big] \to 0, \quad n \to \infty,
\end{align*}
which yields $P_{t_0,t} \varphi(\xi_n) \to P_{t_0,t}\varphi(\xi)$ as $n \to \infty$. The rest of the section is devoted to the proof of \eqref{eq:Feller_aux1}.  

Let $U_1(t,t_0,\xi_1) =(v_1, w_1)$ and $U_2(t,t_0,\xi_2)=(v_2, w_2)$ be two variational solutions of \eqref{eq:stochastic_kavian} with initial data at $t=t_0$, equal to $\xi_1=(v_{0,1}, w_{0,1})$ and $\xi_2=(v_{0,2}, w_{0,2})$, respectively, driven by the same Wiener process. Denote 
\begin{align*}
\delta v := v_1-v_2, \quad 
\delta w := w_1-w_2.
\end{align*}
Then, for $0\le t_0\le t\le T$, we have the following equations for $\delta v, \delta w$:
\begin{align*}
\begin{aligned}
c_m \dd(\delta v) &= -\big((\mathcal{A}+S_0{\rm Id}) \delta v + S_1 v_1 w_1- S_1 v_2 w_2\big)\, \dd t + \big(b(v_1)-b(v_2)\big)\, \dd W_t &\quad& \text{on}\,\, (0,T]\times \Gamma,\\
\dd (\delta w) &= f(v_1,w_1)-f(v_2, w_2)\, \dd t&\quad& \text{on}\,\, (0,T]\times \Gamma, \\
\quad
v_i(t_0,x)&=v_{0,i}(x), \quad w_i(t_0,x)=w_{0,i}(x)&\quad& \text{on}\,\, \Gamma, \quad i=1,2.
\end{aligned}
\end{align*}
By the Itô formula, for $t\in [0,T]$,
\begin{align}
\label{eq:stochastic_kavian_difference_Ito}
\begin{aligned}
\|\delta v\|_{L^2(\Gamma)}^2 =& \| v_{0,1}-v_{0,2}\|_{L^2(\Gamma)}^2 -\frac{2}{c_m} \int_{t_0}^t \langle (\mathcal{A}+S_0{\rm Id}) \delta v, \delta v\rangle_{H^{-1/2}(\Gamma), H^{1/2}(\Gamma)}\, \dd s\\ 
&-\frac{2S_1}{c_m}\int_{t_0}^t \big(w_1 \delta v\, , \, \delta v\big)_{L^2(\Gamma)}\, \dd s 
- \frac{2S_1}{c_m}\int_{t_0}^t \big(v_2\, \delta w\,,\, \delta v\big)_{L^2(\Gamma)}\, \dd s\\
&+ \frac{1}{c_m^2}\int_{t_0}^t \|b(v_1)-b(v_2)\|_{L_2^0}^2\, \dd s
+ M_t,\\
\|\delta w\|_{L^2(\Gamma)}^2 =& 
\| w_{0,1}-w_{0,2}\|_{L^2(\Gamma)}^2
+ 2 \int_{t_0}^t \big(f(v_1,w_1)-f(v_2, w_2) \,,\, \delta w\big)_{L^2(\Gamma)}
\, \dd s,
\end{aligned}
\end{align}
where $M_t$ is defined by
\begin{align*}
    M_t=\frac{2}{c_m}\int_{t_0}^t \big(\delta v\, , \, (b(v_1)-b(v_2))\, \dd W_s\big)_{L^2(\Gamma)}.
\end{align*}
$M_t$ is a martingale since $v_1, v_2$ are progressively measurable and 
\begin{align*}
&\int_{t_0}^t \big(\delta v\, , \, (b(v_1)-b(v_2))\sqrt{\gamma_k}e_k)\big)_{L^2(\Gamma)}^2\, \dd s \le 
\int_{t_0}^t \|\delta v\|_{L^2(\Gamma)}^2
\|b(v_1)-b(v_2)\|_{L_2^0}^2\|Q^{1/2} e_k\|_{E_0}^2\, \dd s\\
&
\le (t-t_0)\sup_{s\in [t_0, t]}\|\delta v\|_{L^2(\Gamma)}^4\le C(T). 
\end{align*}
By Lemma \ref{lm:L}, boundedness of $w_i$, $i=1,2$, the Lipschitz property of $b$ \ref{B1}, and the Lipschitz continuity of $f$, we have
\begin{align*}
\begin{aligned}
\|\delta v\|_{L^2(\Gamma)}^2 \le & \| v_{0,1}-v_{0,2}\|_{L^2(\Gamma)}^2 -\frac{2\Lambda}{c_m} \int_{t_0}^t \|\delta v\|_{H^{1/2}(\Gamma)}^2\, \dd s\\ 
&-\frac{2S_0}{c_m}\int_{t_0}^t \|\delta v\|_{L^2(\Gamma)}^2\, \dd s 
- \frac{2S_1}{c_m}\int_{t_0}^t \big(v_2\, \delta w\,,\, \delta v\big)_{L^2(\Gamma)}\, \dd s\\
&+ \frac{L_{\rm b}}{c_m^2}\int_{t_0}^t \|\delta v\|_{L^2(\Gamma)}^2\, \dd s
+ M_t,\\
\|\delta w\|_{L^2(\Gamma)}^2 \le& 
\| w_{0,1}-w_{0,2}\|_{L^2(\Gamma)}^2
+ 2C \int_{t_0}^t \big(\|\delta v\|_{L^2(\Gamma)}^2 + \|\delta w\|_{L^2(\Gamma)}^2 \big)
\, \dd s,
\end{aligned}
\end{align*}
The problematic term in \eqref{eq:stochastic_kavian_difference_Ito} is the one containing a product of three functions $\int_{t_0}^t \big(v_2\, \delta w\,,\, \delta v\big)_{L^2(\Gamma)}\, \dd s$. The key issue here a growing factor $\|v_2\|_{H^1(\Gamma)}^2$ in front of $\|\delta w\|_{L^2(\Gamma)}^2$, which makes it impossible to use the Grönwall inequality directly. Indeed, by the Hölder inequality $$\int_\Gamma |a\, b\, c|\, dx \le \|a\|_{L^2(\Gamma)} \|b\|_{L^4(\Gamma)} \|c\|_{L^4(\Gamma)},$$ together with the continuous embedding $H^{1/2}(\Gamma) \hookrightarrow L^4(\Gamma)$ for a two-dimensional surface $\Gamma$ in $\mathbb{R}^3$, the Young inequality with a parameter yields
\begin{align*}
\Big|\int_{t_0}^t \big(v_2\, \delta w\,,\, \delta v\big)_{L^2(\Gamma)}\, \dd s\Big|
\le
\frac{\gamma}{2}\int_{t_0}^t\|\delta v\|_{H^{1/2}(\Gamma)}^2\, \dd s
+ \frac{1}{2\gamma} \int_{t_0}^t\|v_2\|_{H^{1/2}(\Gamma)}^2
\|\delta w\|_{L^2(\Gamma)}^2\, \dd s.
\end{align*}
Choosing $\gamma>0$ such that $-2\Lambda + S_1 \gamma <0$, we have, for a generic constant $C$,
\begin{align}
\label{eq:stochastic_kavian_difference_Ito-3}
\begin{aligned}
\|\delta v\|_{L^2(\Gamma)}^2 \le & \| v_{0,1}-v_{0,2}\|_{L^2(\Gamma)}^2 
+C\int_{t_0}^t \|\delta v\|_{L^2(\Gamma)}^2\, \dd s \\
&+ C\int_{t_0}^t\|v_2\|_{H^{1/2}(\Gamma)}^2
\|\delta w\|_{L^2(\Gamma)}^2\, \dd s
+ M_t,\\
\|\delta w\|_{L^2(\Gamma)}^2 \le& 
\| w_{0,1}-w_{0,2}\|_{L^2(\Gamma)}^2
+ 2C \int_{t_0}^t \big(\|\delta v\|_{L^2(\Gamma)}^2 + \|\delta w\|_{L^2(\Gamma)}^2 \big)
\, \dd s.
\end{aligned}
\end{align}
The key feature of the first inequality in \eqref{eq:stochastic_kavian_difference_Ito-3} is that $\|v_2\|_{H^{1/2}(\Gamma)}^2$ is multiplied by $\|\delta w\|_{L^2(\Gamma)}^2$ and not $\|\delta v\|_{L^2(\Gamma)}^2 + \|\delta w\|_{L^2(\Gamma)}^2$. Combining the two inequalities in \eqref{eq:stochastic_kavian_difference_Ito-3} yields
\begin{align*}
\int_{t_0}^t\|v_2\|_{H^{1/2}(\Gamma)}^2
\|\delta w\|_{L^2(\Gamma)}^2\, \dd s
\le&
\| w_{0,1}-w_{0,2}\|_{L^2(\Gamma)}^2 \int_{t_0}^t\|v_2(s)\|_{H^{1/2}(\Gamma)}^2
\, \dd s\\
&+
2C\int_{t_0}^t\|v_2(s)\|_{H^{1/2}(\Gamma)}^2 \int_{t_0}^s
\big(\|\delta v(r)\|_{L^2(\Gamma)}^2
+ \|\delta w(r)\|_{L^2(\Gamma)}^2\big)\, \dd r\, \dd s\\
& \le
\| w_{0,1}-w_{0,2}\|_{L^2(\Gamma)}^2 \int_{t_0}^t\|v_2(s)\|_{H^{1/2}(\Gamma)}^2
\, \dd s\\
&+
2C\int_{t_0}^t\|v_2(s)\|_{H^{1/2}(\Gamma)}^2\, \dd s \int_{t_0}^t
\big(\|\delta v(s)\|_{L^2(\Gamma)}^2
+ \|\delta w(s)\|_{L^2(\Gamma)}^2\big)\, \dd s.
\end{align*}
Denoting
\begin{align*}
\rho(t):= \int_{t_0}^t\|v_2(s)\|_{H^{1/2}(\Gamma)}^2\, \dd s,
\end{align*}
adding up the two inequalities in \eqref{eq:stochastic_kavian_difference_Ito-3}, we obtain, $\PP$-a.s.
\begin{align}
\label{eq:stochastic_kavian_difference_Ito-4}
\begin{aligned}
\|\delta v\|_{L^2(\Gamma)}^2 
+ \|\delta w\|_{L^2(\Gamma)}^2
\le & 
 C (1 + \rho(t)) \big(\| v_{0,1}-v_{0,2}\|_{L^2(\Gamma)}^2 
+ \| w_{0,1}-w_{0,2}\|_{L^2(\Gamma)}^2\big)\\
&+C(1 + \rho(t))\int_{t_0}^t \big( \|\delta v\|_{L^2(\Gamma)}^2+\|\delta w\|_{L^2(\Gamma)}^2\big)\, \dd s + M_t, \,\, t\in [t_0,T].
\end{aligned}
\end{align}
Since $\rho(t)$ is random, we employ a localization argument. Let us introduce a stopping time
\begin{align}
    \label{eq:tau_R_Feller}
    \tau_R := \inf\{t\ge t_0:\,\,
    \int_{t_0}^t\|v_2(s)\|_{H^{1/2}(\Gamma)}^2\, \dd s >R\} \land T, \quad R>0,
\end{align}
and denote for brevity
\begin{align*}
Y(t):= \|U(t,t_0,\xi_1)-U(t,t_0,\xi_2)\|_H^2 = \|\delta v(t)\|_{L^2(\Gamma)}^2 
+ \|\delta w(t)\|_{L^2(\Gamma)}^2.
\end{align*}
Then \eqref{eq:stochastic_kavian_difference_Ito-4} yields
\begin{align*}
\begin{aligned}
Y(\tau_R \land t)
\le 
 C (1 + R) Y({t_0})+C(1 + R)\int_{t_0}^{\tau_R\land t} Y(s)\, \dd s + M_{\tau_R \land t}, \,\, t\in [t_0,T].
\end{aligned}
\end{align*}
Taking the expectation on both sides of the last inequality and using the fact that $M_{\tau_R \land t}$ is a martingale, gives
\begin{align*}
\begin{aligned}
\E \big[Y(\tau_R \land t)\big]
\le 
 C (1 + R) Y({t_0})+C(1 + R)\int_{t_0}^t \E \big[\mathbbm{1}_{[{t_0}, \tau_R]}Y(s)\big]\, \dd s, \,\, t\in [{t_0},T].
\end{aligned}
\end{align*}
Taking into account
\begin{align*}
    \mathbbm{1}_{[{t_0}, \tau_R]}Y(s)
    = \mathbbm{1}_{[{t_0}, \tau_R]}Y(\tau_R \land s)
    \le Y(\tau_R \land s),
\end{align*}
and using the Grönwall inequality yields
\begin{align*}
\begin{aligned}
\E \big[Y(\tau_R \land t)\big]
\le 
 C(T, R) \, Y({t_0}), \,\, t\in [{t_0},T],
\end{aligned}
\end{align*}
or equivalently
\begin{align}
\label{eq:est_delta_U_tau_R}
\E \big[\|U(\tau_R \land t,{t_0},\xi_1)-U(\tau_R \land t,{t_0},\xi_2)\|_H^2\big]
\le 
 C(T, R) \, \|\xi_1-\xi_2\|_H^2, \,\, t\in [{t_0},T].
\end{align}
Now we turn to the localization argument. For an arbitrary $\ve >0$ and a fixed $t\in [{t_0},T]$, we have
\begin{align}
\label{eq:localization_Feller}
\begin{alignedat}{1}
&\PP \Big(\|U(t,{t_0},\xi_1)-U(t,{t_0},\xi_2)\|_H >\ve\Big)= \PP \Big(\|U(t,{t_0},\xi_1)-U(t,{t_0},\xi_2)\|_H^2 >\ve^2\Big)\\
&=
\PP \Big(\|U(t,{t_0},\xi_1)-U(t,{t_0},\xi_2)\|_H^2 >\ve^2 \,, \, \tau_R<t \Big)
+ \PP \Big(\|U(t,{t_0},\xi_1)-U(t,{t_0},\xi_2)\|_H^2 >\ve^2 \,, \, \tau_R\ge t \Big)\\
&\le \PP \Big(\tau_R<t \Big)
+ \PP \Big(\|U(t,{t_0},\xi_1)-U(t,{t_0},\xi_2)\|_H^2 >\ve^2 \,, \, \tau_R\ge t \Big).
\end{alignedat}
\end{align}
By the definition of the stopping time \eqref{eq:tau_R_Feller}, the Markov inequality, and the second estimate in \eqref{eq:better-estimate-v_H^1/2}, we have
\begin{align*}
\begin{alignedat}{1}
\PP\big(\tau_R <T\big)=&
\PP\big(\int_{t_0}^T \|v_2(s)\|_{H^{1/2}(\Gamma)}^2\, \dd s \, > \, R\big)\\
&\le \frac{1}{R} \E \Big[\int_{t_0}^T \|v_2(s)\|_{H^{1/2}(\Gamma)}^2\, \dd s\Big] \le \frac{C(T)}{R} \to 0, \quad R\to \infty.
\end{alignedat}
\end{align*}
By the Markov inequality and \eqref{eq:est_delta_U_tau_R},
\begin{align*}
\begin{alignedat}{1}
&\PP \Big(\|U(t,{t_0},\xi_1)-U(t,{t_0},\xi_2)\|_H^2 >\ve^2 \,, \, \tau_R\ge t \Big)\le
\frac{1}{\ve^2} \E\Big[\mathbbm{1}_{\{\tau_R \ge  t\}}\|U(t,{t_0},\xi_1)-U(t,{t_0},\xi_2)\|_H^2\Big]\\
&
\le \frac{C(T,R)}{\ve^2} \|\xi_1-\xi_2\|_H^2.
\end{alignedat}
\end{align*}
Choosing $\xi_1=\xi_n$ and $\xi_2=\xi$ such that $\|\xi_n - \xi\|_H \to 0$ as $n\to \infty$, combining \eqref{eq:localization_Feller}--\eqref{eq:est_delta_U_tau_R}, and passing to the limit first as $n\to \infty$ and then as $R\to \infty$, we obtain
\begin{align*}
\begin{alignedat}{1}
&\limsup_{n\to \infty}\PP \Big(\|U(t,{t_0},\xi_n)-U(t,{t_0},\xi)\|_H >\ve\Big)\le
\lim_{R\to \infty}\frac{C(T)}{R}=0.
\end{alignedat}
\end{align*}
Since the probability is nonnegative, we finally obtain
\begin{align*}
&\lim_{n\to \infty}\PP \Big(\|U(t,t_0,\xi_n)-U(t,{t_0},\xi)\|_H >\ve\Big)=0.
\end{align*}
Lemma \ref{lm:Feller} is proved.
\end{proof}
\subsection{Mild solution}

By definition, $\mathbb{A}(U)=A U + F(U)$, where $A$ given by \eqref{def:A_kavian} does not depend on $t, \omega$.
To prove the existence of a mild solution, we need a higher regularity of the domain and the data.
Assume that $\Gamma$ and $\partial G$ are $C^{2}$, and $g\in C([0,T]; H^{3/2}(\Gamma))$, then $p$ solving \eqref{eq:p} belongs to $C([0,T]; H^2(G_{\rm i}\cup G_{\rm e}))$, and consequently $\sigma_{\rm i} \nabla p \cdot \n \in C([0,T]; L^2(\Gamma))$.
Then the variational solution $U$ given by Theorem \ref{th:existence-(0,T)_kavian} is also an analytically weak solution (see Definition G.0.3 in \cite{liu2015stochastic}). Recall that an $H$-valued predictable process $U(t)$, $t\in [0,T]$, is called an analytically weak solution of \eqref{eq:stochastic_kavian} if 
\begin{align}
\label{eq:analytically_weak_sol_kavian}
\begin{alignedat}{2}
(U(t), \varphi)_H&=(U_0, \varphi)_H + \int_0^t (U(s), A^\ast \varphi)_H\, \dd s + 
\int_0^t (F(U(s)), \varphi)_H\, \dd s \\
&+\int_0^t (P(s), \varphi)_H\, \dd s
+\int_0^t (\varphi, B(U(s))\, \dd W_s)_H,
\end{alignedat}
\end{align}
for any $\varphi \in D(A^\ast)$, $t\in [0,T]$. Here $A^\ast$ is the adjoint of $A$ on $L^2(\Gamma)$, and the integrals in \eqref{eq:analytically_weak_sol_kavian} are well-defined due to Theorem \ref{th:existence-(0,T)_kavian}.

In order to prove that the weak solution is also mild, we need to introduce a semigroup. Let us consider the Dirichlet to Neumann operator  $\mathcal{A}:D(\mathcal{A}) \subset L^2(\Gamma) \to L^2(\Gamma)$, an unbounded symmetric operator with dense domain $D(\mathcal{A})=\{v\in H^{1/2}(\Gamma):\,\, \mathcal{A}v \in L^2(\Gamma)\}$. It is monotone $(\mathcal{A} v, v) \ge 0$, $v\in D(\mathcal{A})$. Let us show that $\mathcal{A}$ is $m$-accretive. To this end, we prove that for every $\lambda>0$, $({\rm Id} + \lambda \mathcal{A})$ is bijective from $D(\mathcal{A})$ onto $L^2(\Gamma)$ with $\|({\rm Id} + \lambda \mathcal{A})^{-1}\|_{L(L^2(\Gamma))} \le 1$. Indeed, $({\rm Id} + \lambda \mathcal{A})$ is injective since $-\mathcal{A}$ is non-negative (see Lemma \ref{lm:L}). To prove that $({\rm Id} + \lambda \mathcal{A})$ is surjective, we fix $f\in L^2(\Gamma)$ and find $v \in D(A)$ such that $({\rm Id} + \lambda \mathcal{A}) v = f$. The weak formulation reads 
\begin{align*}
\lambda \int_{G_{\rm i} \cup G_{\rm e}} \sigma \nabla u^{(v)} \cdot \nabla \phi\, \dd x + \int_\Gamma v \, [\phi]\, \dd S = \int_\Gamma f\, \phi\, \dd S, \quad \phi\in H^{1/2}(\Gamma),
\end{align*}
where $u^{(v)}$ solves \eqref{eq:micro_deterministic_on_Gamma}. Clearly, for every $f\in L^2(\Gamma)$, there exists a solution $v\in H^{1/2}(\Gamma)$. Further, $\mathcal{A} v = \frac{1}{\lambda}(f-v) \in L^2(\Gamma)$, that implies $v\in D(\mathcal{A})$. Since for $v\in D(\mathcal{A})$, $\|v\|_{L^2(\Gamma)} \le \|f\|_{L^2(\Gamma)}$, we have also $\|({\rm Id} + \lambda \mathcal{A})^{-1}\|_{L(L^2(\Gamma))} \le 1$. Thus, by the Hille-Yosida theorem (Theorem 7.8 in \cite{brezis2011functional}) $-\mathcal{A}$ generates a $C_0$-semigroup of contractions on $L^2(\Gamma)$ which we denote $S(t)=e^{-t\mathcal{A}}$. 

{In addition, since the resolvent $({\rm Id} + \lambda \mathcal{A})^{-1}:L^2(\Gamma) \to H^{1/2}(\Gamma)$ is bounded and the embedding $H^{1/2}(\Gamma) \hookrightarrow L^2(\Gamma)$ is compact, the resolvent is compact on $L^2(\Gamma)$. This implies the compactness of the semigroup $S(t)$ (see Corollary 3.5).} 

Consider now the linear operator $A:D(A)\subset H \to H$ given by \eqref{def:A_kavian}, with $D(A) = D(\mathcal{A}) \times L^2(\Gamma)$. Then $A$ generates a semigroup on $H$:
\begin{align*}
\mathbb{S}(t) = \begin{pmatrix}
    e^{-\frac{S_0}{c_m} t} S(t) & 0\\ 0& 0
\end{pmatrix}.
\end{align*}
Let us recall the definition of a mild solution of \eqref{eq:stochastic_kavian}:

\begin{definition}
An $H$-valued predictable process $U(t)$, $t\in[0,T]$, is called a mild solution of \eqref{eq:orig-prob-vector_kavian} if it takes values in $B(U)$, $\int_0^T \|U\|_H\, \dd s <\infty$, $\PP$-a.s., and for all $t\in[0,T]$ it holds
\begin{align*}
U(t)=\mathbb{S}(t)U_0 + \int_0^t \mathbb{S}(t-s) \big(F(U(s))+ P(s)\big)\, \dd s+
\int_0^t \mathbb{S}(t-s) B(U(s))\, \dd W_s, \quad \PP-\text{a.s.}.
\end{align*}
\end{definition}

Due to the Lipschitz continuity assumption \ref{B1} and the estimates obtained in Theorem \ref{th:existence-(0,T)_kavian}, the variational solution in Theorem \ref{th:existence-(0,T)_kavian} is also mild (see Theorem 6.7 in \cite{da2014stochastic} or Proposition G.0.5 in \cite{liu2015stochastic}).

\section{Existence of invariant measure and regularity of its support}
\label{sec:existence_inv_measure}

In this section, we consider a special case of a constant in time boundary data $g=g(x)$ and prove the existence of an invariant measure and its properties for \eqref{eq:orig-prob-vector_kavian}. In this case, system \eqref{eq:orig-prob-vector_kavian} is autonomous, and the transition operators are $P_{s,t}\varphi(\xi)=P_{0,t-s}\varphi(\xi)$. Throughout this section we will use the notation $P_t$ for the transition semigroup generated by the solution $U(t, 0, \xi)$ of \eqref{eq:orig-prob-vector_kavian} on $H_\infty$.

\paragraph{Existence of invariant measure}
We start by recalling necessary definitions. Given a Markovian transition operator $P_t$ \eqref{P_t} it can be written in terms of the corresponding transition function $P_t(x, \cdot)$:
\begin{align*}
P_t \varphi(x) = \int_{H_\infty} \varphi(x)\, P_t(x, \dd y), \quad t\ge 0, \,\, x\in H_\infty.
\end{align*}
We introduce a dual semigroup $P_t^\ast$ acting on probability measures $\mathcal{M}_1(H_\infty)$  by setting
\begin{align*}
P_t^\ast \nu (A) = \int_{H_\infty} P_t(x, A)\, \nu(\dd x), \quad \nu \in \mathcal{M}_1(H_\infty), \, \, A\in \mathcal{B}(H_\infty). 
\end{align*}
\begin{definition}[Invariant measure]
Let $P_t$, $t\ge 0$, be a Markov semigroup on $H_\infty$. A probability measure $\mu$ on $H_\infty$ is called invariant if for any $\varphi\in C_b(H_\infty)$ it holds
\begin{align}
\label{def:inv_meas_autonom}
\int_{H_\infty} P_t \varphi(x) \, \mu(\dd x)
= \int_{H_\infty} \varphi(x)\, \mu(\dd x), \quad 
t\ge 0.
\end{align}
\end{definition}
In terms of the dual semigroup, the condition \eqref{def:inv_meas_autonom} is equivalent to $P_t^\ast \mu = \mu$, for any $t\ge 0$. Next, we define the time average of the probabilities for $U(t, 0, \xi)$ starting in $\xi$  being in $A$ at time $t$:
\begin{align*}
R_T(x, A) = \frac{1}{T} \int_0^T P_t(x,A)\, \dd t, \quad x\in {H_\infty}, \,\, A\in \mathcal{B}({H_\infty}).
\end{align*}
Note that for each $x\in {H_\infty}$ and $T\ge 0$, $R_T(x, \cdot)$ is a probability measure on ${H_\infty}$ since $P_t$ is stochastically continuous. 
For $\nu\in \mathcal{M}_1({H_\infty})$ and $A \in \mathcal{B}({H_\infty})$ we define the dual $R_T^\ast$ acting on $\mathcal{M}_1(H)$ by
\begin{align}
\label{def:R_T^ast}
R_T^\ast \nu(A) = \frac{1}{T} \int_0^T P_t^\ast \nu (A)\, \dd t.
\end{align}
\begin{definition}[Tight sequence of measures]
    A family of measures $(\mu_n)_{n\ge 1}$ on ${H_\infty}$ is tight if for any $\ve >0$, there exists a compact $K_\ve \Subset {H_\infty}$ such that
    $\mu_n(K_\ve) \ge 1-\ve$.
    Equivalently, $\sup_n \mu_n (K_\ve^c) <\ve$.
\end{definition}
We will prove the existence of an invariant measure by using the Krylov-Bogoliubov theorem (see Theorem 3.1.1 in \cite{da1996ergodicity}) which we formulate below for the reader's convenience. 
\begin{theorem}[Krylov-Bogoliubov theorem]
Let $P_t$, $t\ge 0$, be a Markovian Feller semigroup which is stochastically continuous. If for some $\xi \in {H_\infty}$, some $\nu\in \mathcal{M}_1({H_\infty})$ and some sequence $T_n \to +\infty$, the sequence $R_{T_n}^\ast \nu$ is tight, then there exists an invariant measure for $P_t$, $t\ge 0$.
\end{theorem}
Now we turn to the proof of Theorem \ref{th:existence_inv_measure}.
\begin{proof}[Proof of Theorem \ref{th:existence_inv_measure}]
Let us define 
\begin{align*}
K_R = \{(v, w)\in H^{1/2}(\Gamma) \times H^{1/2}(\Gamma;[0,1]): \quad \|v\|_{H^{1/2}(\Gamma)}^2 + \|w\|_{H^{1/2}(\Gamma)}^2 \le R\}.
\end{align*}
The set $K_R\subset H_\infty$ is compact since $V\hookrightarrow H$ is compact and $K_R$ is closed in $H_\infty$ under $L^2$-limits. 
We choose $\nu=\delta_\xi$ with $\xi\in L^2(\Gamma) \times H^{1/2}(\Gamma;[0,1])$ and we will estimate $R_T^\ast \delta_\xi (K_R^c)$. By definition \eqref{def:R_T^ast},
we have
\begin{align*}
R_T^\ast \delta_\xi (K_R^c)
=& R_T(\xi, K_R^c) 
= \frac{1}{T} \int_0^T P_t(\xi, K_R^c)\, \dd t
= \frac{1}{T} \int_0^T \PP \big(\|v\|_{H^{1/2}(\Gamma)}^2 + \|w\|_{H^{1/2}(\Gamma)}^2 > R\big)\, \dd t.
\end{align*}
By the Markov inequality,
\begin{align*}
R_T^\ast \delta_\xi (K_R^c)
\le &
\frac{1}{R} \frac{1}{T} \int_0^T \big(\E \big[\|v\|_{H^{1/2}(\Gamma)}^2\big] + \E \big[\|w\|_{H^{1/2}(\Gamma)}^2\big]\big)\, \dd t.
\end{align*}
By the estimates in Lemma \ref{lm:uniform_time_estimate},
\begin{align*}
R_T^\ast \delta_\xi (K_R^c)
\le &
\frac{C}{R}\Big(1+\frac{1}{T}\Big),
\end{align*}
where the constant $C$ depends on $\|v_0\|_{L^2(\Gamma)}$, $\|w_0\|_{H^{1/2}(\Gamma)}$, $\|g\|_{H^{1/2}(\partial G)}$, $\|b(0)\|_{L_2^0}$, and is independent of $T$ and $R$.
For $T\ge 1$,
\begin{align*}
\sup_{T\ge 1} R_T^\ast \delta_\xi (K_R^c)
\le &
\frac{C}{R},
\end{align*}
which can be made arbitrary small as $R\to +\infty$. Thus, $R_T^\ast \delta_\xi$ is tight in $H$, and up to a subsequence, $R_T^\ast \delta_\xi$ converges to some $\mu \in \mathcal{M}_1({H_\infty})$. The existence of invariant measure is thus given by the Krylov-Bogoliubov theorem.

\paragraph{Regularity of the support}
We are now in the position to prove \eqref{eq:support} and show first that any invariant measure $\mu\in \mathcal{M}_1({H_\infty})$ is supported on ${V_\infty}$, and then conclude that the support is in fact $Z_\infty$. 
The idea of the proof is inspired by the works \cite{glatt2014existence} and \cite{kapustyan2022strong}, in which the authors derived similar results for the invariant measures emerging in stochastic primitive equations and bidomain equations. For any $R>0$ and $U \in H_\infty$ denote \begin{align*}
f_R(U): =\begin{cases}
\min (|U\|_{V}^2, R), \quad U\in {V_\infty}, \\ 
R, \quad {U \in H_\infty\setminus V_\infty}. 
\end{cases}
\end{align*}
Note that $f_R$ is Borel measurable on {$H_\infty$}. By definition of invariant measure,
\begin{align*}
\int_{{H_\infty}} f_R(U_0) \mu(\dd U_0) =
\int_{H_\infty} P_t f_R(U_0) \mu(\dd U_0).
\end{align*}
By taking the time average on both sides, for any $T>0$, we get
\begin{align*}
\int_{{H_\infty}} f_R(U_0) \mu(\dd U_0) =
\frac{1}{T} \int_{0}^T \int_{{H_\infty}}  P_t f_R(U_0) \mu(\dd U_0) \dd t.
\end{align*}
For any $\rho \geq 1$, denote \begin{align*}
B_\rho^H = \{U_0 \in {H_\infty}: \,\, \|U_0\|_H \le \rho\}.
\end{align*}
Then,  estimates \eqref{eq:better-estimate-v_H^1/2}, \eqref{eq:better-estimate-w_L^2} in Lemma \ref{lm:uniform_time_estimate} yield
\begin{align}\label{inv2}
\begin{alignedat}{1}
\int_{{H_\infty}} f_R(U_0) \mu(\dd U_0)
&=
\frac{1}{T} \int_{0}^T \int_{B_\rho^H}  P_t f_R(U_0) \mu(\dd U_0) \dd t
+ 
\frac{1}{T} \int_{0}^T \int_{{H_\infty}\setminus B_\rho^H}  P_t f_R(U_0) \mu(\dd U_0) \dd t
\\
&\le 
\int_{B_\rho^H} \frac{1}{T} \int_{0}^{T} 
\E \|U(t,0, U_0)\|_V^2 \, \dd t \, \mu(\dd U_0)
+ R\, \mu({H_\infty}\setminus B_\rho^H)
\\
&
\leq C\int_{B_\rho^H} (1 + \frac{1}{T}\|U_0\|_{H}^2) \,\mu(\dd U_0) +R\, \mu({H_\infty}\setminus B_\rho^H)\\ 
&\leq C \left(1 + \frac{\rho^2}{T}\right)\mu(B_\rho^H) +R\, \mu({H_\infty}\setminus B_\rho^H).
\end{alignedat}
\end{align}
The constant $C$ above depends on $\|b(0)\|_{L_2^0}, \|g\|_{H^{1/2}(\partial G)}$.
First, we choose a sufficiently large $\rho=\rho(R)$ such that
$R\mu(H \setminus B_\rho^H) \leq 1$.
Then for such $\rho$ we choose $T$ large enough so that
$\rho^2/T \le 1$.
Altogether, \eqref{inv2} yields an estimate with $C$ independent of $R$:
\begin{equation}\label{bndR}
\int_{H_\infty} f_R(U_0) \mu(\dd U_0) \leq C.
\end{equation}
Since 
\begin{align*}
f_R(U_0) \to \|U_0\|_{V,\ast}^2=\begin{cases}
\|U_0\|_{V}^2, \quad U_0 \in { V_\infty},\\
\infty, \quad U_0 \in {H_\infty \setminus V_\infty},
\end{cases}
\end{align*}
applying the Monotone Convergence Theorem to \eqref{bndR} we get
\[
\int_{{H_\infty}}\|U_0\|^2_{V,\ast} \mu(\dd U_0) < \infty.
\]
Thus, $\mu({H_\infty \setminus V_\infty})=0$, and
\begin{align}
\label{eq:supp_mu_in_V}
\int_{{H_\infty}}\|U_0\|^2_{V} \mu(\dd U_0) < \infty.
\end{align}
In particular, $\mu({V_\infty})=1$, that is $\mu$ is supported on ${V_\infty}$.

The last step of the proof is to show that any invariant measure $\mu$ is in fact supported on $Z_\infty = H^{1/2}(\Gamma) \times H^{1/2}(\Gamma; [0,1])$. To this end, we will exploit the fact that two solutions $w_1, w_2$ of the equation $\partial_t w = f(v,w)$ driven by the same $v$ satisfy the estimate
\begin{align*}
    \|w_1(t)-w_2(t)\|_{L^2(\Gamma)}\le e^{-\kappa (t-s)} \|w_1(s)-w_2(s)\|_{L^2(\Gamma)}.
\end{align*}
In other words, $w$ forgets its initial condition exponentially fast. 

{Take any invariant measure $\mu$ on $H_\infty$. Let $U^\ast(t)$ be the solution of \eqref{eq:orig-prob-vector_kavian} with the initial condition $U^\ast(0)=\xi$, where $\xi$ is a  $H_\infty$-valued, $\F_0$-measurable random variable with the law $\mathcal{L}(\xi)=\mu$. Since $\mu$ is invariant, $\mathcal{L}(U^\ast(t))=P_t^\ast \mu =\mu$ for all $t\ge 0$. Further, if we denote
\begin{align*}
\pi_w: H_\infty \to L^2(\Gamma; [0,1]), \quad \pi_w(v,w)= w, \quad
\mu^w:= \mu \circ \pi_w^{-1}, 
\end{align*}
then $\mathcal{L}(w)= \mu_w$, and for any Borel set $B \subset L^2(\Gamma; [0,1])$, $\mu^w(B)= \mu(L^2(\Gamma)\times B)$, and 
\begin{align}
\label{eq:mu_w_and_mu}
\int_{L^2(\Gamma; [0,1])} \varphi(w)\, \mu_w(\dd w) =
\int_{H_\infty} \varphi(\pi_w(U))\, \mu(\dd U), \quad \varphi\in C_b(L^2(\Gamma; [0,1])).
\end{align}
For $t \ge 0$, let us consider the auxiliary problem
\begin{align*}
    \partial_t w &= f(v^\ast, w) \quad \mbox{on}\,\, (0, \infty) \times \Gamma,\\
    w(0)&=0\quad \mbox{on}\,\, \Gamma.
\end{align*}
Since the initial data is smooth and $v^\ast$ is $H^{1/2}(\Gamma)$-valued, as in the proof of Lemma \ref{lm:uniform_time_estimate}, for any $n\ge 0$, we obtain
\begin{align}
\label{aux_inv_supp_1}
\|w(n)\|_{H^{1/2}(\Gamma)}^2 \le C \int_0^n e^{-\frac{(n-r)}{2\tau_{\rm res}}}(1+ \|v^\ast(r)\|_{H^{1/2}(\Gamma)}^2)\, \dd r.
\end{align}
By stationarity,
\begin{align*}
\E\|v^\ast(t)\|_{H^{1/2}(\Gamma)}^2 =
\E\|v^\ast(0)\|_{H^{1/2}(\Gamma)}^2, \quad t\ge 0.
\end{align*}
Thus, taking the expectation in \eqref{aux_inv_supp_1} yields
\begin{align}
\label{aux_inv_supp_2}
\sup_{n\ge 0} \E \|w(n)\|_{H^{1/2}(\Gamma)}^2 \le C (1+ \E \|v^\ast(0)\|_{H^{1/2}(\Gamma)}^2)\le C.
\end{align}
Next, as in the proof of Lemma \ref{lm:uniform_time_estimate}, we derive
\begin{align}
\label{aux_inv_supp_3}
\|w^\ast(n)-w(n)\|_{L^2(\Gamma)}^2 
\le e^{-\frac{n}{2\tau_{\rm res}}}\|w^\ast(0)-w(0)\|_{L^2(\Gamma)}^2 
\le e^{-\frac{n}{2\tau_{\rm res}}}|\Gamma| \to 0, \quad n\to +\infty,
\end{align}
which implies the strong convergence $w^\ast(n)-w(n) \to 0$ in $L^2(\Omega; L^2(\Gamma))$. 
Let us show that $\mu_n:=\mathcal{L}(w(n))$ converges weakly to $\mathcal{L}(w^\ast)=\mu_w$ in $L^2(\Gamma; [0,1])$, and from there conclude that $\mu_w$ is supported on $H^{1/2}(\Gamma; [0,1])$. 
Define
\begin{align*}
    F(y)= \begin{cases}
        \|y\|_{H^{1/2}(\Gamma)}^2, \quad y\in H^{1/2}(\Gamma; [0,1])\\
        +\infty, \quad y\in L^2(\Gamma; [0,1]) \setminus H^{1/2}(\Gamma; [0,1]).
    \end{cases}
\end{align*}
By \eqref{aux_inv_supp_2},
\begin{align*}
    \sup_n \int_{L^2(\Gamma; [0,1])} F(y)\, \mu_n(\dd y) < \infty.
\end{align*}
%
Let us show that $\mu_n$ converges weakly to $\mu_w$ in $L^2(\Gamma)$. For any $\varphi \in C_b(L^2(\Gamma; [0,1]))$, by \eqref{aux_inv_supp_3}, $\varphi(w(n)) - \varphi(w^\ast(n)) \to 0$ in probability in $L^2(\Gamma)$, and 
thus, by the dominated convergence theorem,
\begin{align*}
\E\big[\varphi(w(n)) - \varphi(w^\ast(n))\big] \to 0, \quad n \to \infty,
\end{align*}
or equivalently
\begin{align*}
\int_{L^2(\Gamma;[0,1])}\varphi(y)\mu_n(\dd y) \to \int_{L^2(\Gamma;[0,1])}\varphi(y)\mu_w(\dd y), \quad \varphi \in C_b(L^2(\Gamma; [0,1])).
\end{align*}
By the Portmanteau lower semicontinuity,
\begin{align*}
    \int_{L^2(\Gamma; [0,1])} F(y)\, \mu_w(\dd y) \le
    \liminf_{n\to \infty} \int_{L^2(\Gamma; [0,1])} F(y)\, \mu_n(\dd y)
    = \liminf_{n\to \infty} \E \|w(n)\|_{H^{1/2}(\Gamma)}^2 < \infty.
\end{align*}
Hence,
\begin{align}
\label{eq:supp_mu_w}
    \int_{L^2(\Gamma; [0,1])} \|y\|_{H^{1/2}(\Gamma)}^2\, \mu_w(\dd y) < \infty.
\end{align}
Combining \eqref{eq:mu_w_and_mu}, \eqref{eq:supp_mu_w}, and \eqref{eq:supp_mu_in_V} yields 
\begin{align*}
    \int_{H_\infty} \|y\|_{Z}^2\, \mu(\dd y) < \infty.
\end{align*}
Since $\mu= \mathcal{L}(U^\ast(t))$ and $\mu(Z_\infty)=1$, we conclude that the stationary solution $U^\ast \in Z_\infty$, $\PP$-a.s..
}
The proof of Theorem \ref{th:existence_inv_measure} is complete.
\end{proof}

\section{Long-time behavior in the case of truncated nonlinearity}
\label{sec:inv_measure_truncated_reaction}

While the existence of invariant measure for \eqref{eq:orig-prob-vector_kavian} can be proved by combining a localization argument with uniform in time estimates in Lemma \ref{lm:uniform_time_estimate}, the uniqueness for the non-truncated (non-Lipschitz) $S_m(v,w)$ is not clear. Classically, the uniqueness of invariant measure is proved for strictly monotone operators. The goal of this section is to prove Theorem \ref{th:inv_measure_truncated}, that is to show that for a truncated nonlinearity \eqref{eq:S_M}, the invariant measure exists and is unique. The exponential stabilization of solutions will lead to a constructive proof using a pullback limit. We prove also the existence of a stationary solution defined on the whole real line.
\begin{proof}[Proof of Theorem \ref{th:inv_measure_truncated}]
Recall that we set
\begin{align*}
S_M(v, w)= S_0 \, v + S_1 w \,T_{[-M,M]}(v), \quad  T_{[-M,M]}(v) = \max\big(-M, \min(M, v)\big),
\end{align*} 
\noindent for some $M>0$. Theorem \ref{th:existence-(0,T)_kavian} and Lemma \ref{lm:uniform_time_estimate} remain valid for \eqref{eq:orig-prob-vector_kavian_truncation} and an arbitrary initial time $s\ge 0$. 
As in Section \ref{sec:Markov_and_Feller}, we denote by $P_t$, $t\ge 0$, the transition operators associated with the solution $U(t,s,\xi)$ of \eqref{eq:orig-prob-vector_kavian_truncation}.
The key property needed to establish the uniqueness of invariant measure is the following lemma, which guarantees the exponential stability.
Following \cite{liu2015stochastic}, Sec. 4.3, we extend the Wiener process $W$ for negative $t$ by
\begin{align*}
\overline{W}_t =
\begin{cases}
W_t, \text{ for } t \geq 0,\\
V_{-t}, \text{ if } t < 0,
\end{cases}
\end{align*}
where $V_t,t > 0$ is another Wiener process, independent of $W_t$. For $s\le t$, the increments
$(\overline W_t-\overline W_s)$
are independent of the past. We take $(\overline{\mathcal F}_t)_{t\in\mathbb R}$
to be the usual augmentation of the natural two-sided filtration
\[
\overline{\mathcal F}_t
=
\sigma\bigl(\overline W(b)-\overline W(a): -\infty<a\le b\le t\bigr),
\qquad t\in\mathbb R.
\]
Then $\overline W$ is a two-sided Wiener process with respect to
$(\overline{\mathcal F}_t)_{t\in\mathbb R}$.
For simplicity of notation, we still denote the process $\overline{W_t}$ by $W_t$, and $\overline \F_t$ by $\F_t$. We now choose
{$s_1\le s_2 \le 0$} and denote by $U(t,s_1, \xi)$ the solution of \eqref{eq:orig-prob-vector_kavian_truncation} {with the extended Wiener process and}  initial condition $U(s_1, s_1, \xi) = \xi$. Let $U_1(t) := U(t,s_1,\xi)$ and $U_2 := U(t,s_2, \xi)$. 
\begin{lemma}\label{Le:exp_cont}
Suppose the Lipschitz constants $L_b, L_\beta$, and the constant $M$ in \eqref{eq:S_M} are sufficiently small, and the assumptions \ref{assumption_sigma}--\ref{assumption:initial_conditions}, \ref{B1}-\ref{B2} hold. Then 
\begin{itemize}
\item[(i)]
For any $s \leq t$ and for $U_1(t):=U_1(t,s,\xi_1)=(v_1,w_1)$, $U_2(t):=U_2(t,s,\xi_2)=(v_2,w_2)$, $\xi_1, \xi_2 \in H$ defined by \eqref{eq:var_sol_s_xi} we have
\[
\E \|U_1(t,s,\xi_1) - U_2(t,s,\xi_2)\|_{H}^2  \leq e^{-\kappa(t-s)} \E \|\xi_1 - \xi_2\|_{H}^2.
\]
\item[(ii)] For any $-\infty \le s_1 \le s_2\le t$ and $\xi\in H$, there exists $\kappa > 0$ such that 
\[
\E \|U_1(t, s_1,\xi) - U_2(t,s_2,\xi)\|_{H}^2  \leq e^{-\kappa(t-s_2)} \E \|U_1(s_2, s_1,\xi) - \xi\|_{H}^2.
\]
\end{itemize}
\end{lemma}

\begin{proof}
(i)
Denote $\delta U:=U_1-U_2$, $\delta v:= v_1 - v_2, \delta w = w_1-w_2$.
By Ito's formula,
\begin{align}\label{exp_contr}
\begin{alignedat}{1}
    \|\delta U(t)\|_H^2  &= \|\delta U(s)\|_H^2 - \frac{2}{c_m} \int_s^t \langle (\mathcal{A} + S_0\, {\rm Id}) \delta v, \delta v \rangle_{H^{-1/2}(\Gamma), H^{1/2}(\Gamma)}\, \dd r \\
   & -  \frac{2 S_1}{c_m} \int_s^t \left(w_1 (T_{[-M,M]}(v_1) - T_{[-M,M]}(v_2), \delta v)\right)_{L^2(\Gamma)} \, \dd r\\
   &-  \frac{2 S_1}{c_m} \int_s^t \left( T_{[-M,M]}(v_2) \delta w, \delta v\right)_{L^2(\Gamma)} \, \dd r  \\
  & + \int_s^t (f(v_1,w_1) - f(v_2,w_2), \delta w)_{L^2(\Gamma)}\, \dd r + M_t + \frac{1}{c_m^2}  \int_s^t \|b(v_1) - b(v_2)\|_{L_2^0}^2 \, \dd r,
  \end{alignedat}
\end{align}
where we denoted
\begin{align*}
    M_t = 2 \int_s^t \Big(\delta U\, , \, (B(U_1)-B(U_2))\dd W_r\Big)_H.
\end{align*}
Since $T_{[-M,M]}(v)$ is monotonously increasing in $v$ and $w_1 \ge 0$, we have
\[
-\left(w_1 (T_{[-M,M]}(v_1) - T_{[-M,M]}(v_2)), \delta v\right)_{L^2(\Gamma)}  \leq 0. 
\]
Next, using \eqref{f_monotonicity}, as in the proof of Lemma \ref{lm:L},  we have
\begin{align*}
(f(v_1,w_1) - f(v_2,w_2), \delta w)_{L^2(\Gamma)} 
& \leq \left(-\frac{1}{2\tau_{\rm res}} + \frac{\gamma_1}{2\tau_{\rm ep}}\right) \|\delta w\|_{L^2(\Gamma)}^2 + L^2 \left(\frac{1}{\tau_{\rm res}} - \frac{\gamma_1}{\tau_{\rm ep}} +  \frac{1}{\tau_{\rm ep} \gamma_1} \right) \|\delta v\|_{L^2(\Gamma)}^2.
\end{align*}
Finally, using again Young's inequality with $\gamma_2>0$ we have
\[
 |\left( T_{[-M,M]}(v_2) \delta w, \delta v\right)_{L^2(\Gamma)}| \le \gamma_2 \|\delta v\|^2 + \frac
 {M^2}{\gamma_2} \|\delta w\|^2.
\]
Altogether, \eqref{exp_contr} now reads as 
\begin{align*}
& \E \|\delta U(t)\|_H^2  \leq  \E \|\delta U(s)\|_H^2 \\
& + \left(-\frac{2\Lambda -2S_1 \gamma_2 }{c_m} + \frac{L_b}{c_m^2}
+ L^2 \left(\frac{1}{\tau_{\rm res}} - \frac{\gamma_1}{\tau_{\rm ep}} +  \frac{1}{\tau_{\rm ep} \gamma_1} \right)  \right) \E \bigg[ \int_s^t \|\delta v\|^2 \, \dd r  \bigg]  \\
& + \left(-\frac{1}{2\tau_{\rm res}} + \frac{\gamma_1}{2\tau_{\rm ep}} + \frac
 {2 S_1M^2}{\gamma_2}\right) \E \bigg[ \int_s^t \|\delta w\|^2\, \dd r \bigg].
\end{align*}
We first choose $\gamma_1>0$ such that 
\[
-\frac{1}{2\tau_{\rm res}} + \frac{\gamma_1}{2\tau_{\rm ep}} \leq -\frac{1}{4\tau_{\rm res}}.
\]
Next, let $\gamma_2, L$ and $L_{\rm b}$ be such that 
\[
-\frac{2\Lambda + 2 S_0 -2S_1 \gamma_2}{c_m} + \frac{L_b}{c_m^2} 
+ L^2 \left(\frac{1}{\tau_{\rm res}} - \frac{\gamma_1}{\tau_{\rm ep}} +  \frac{1}{\tau_{\rm ep} \gamma_1} \right) \leq -\frac{\Lambda +  S_0}{c_m}. 
\]
Finally, choose $M>0$ such that
\[
-\frac{1}{4 \tau_{\rm res}} + \frac{2 S_1 M^2}{\gamma_2} \leq -\frac{1}{8 \tau_{\rm res}}.
\]
This way, choosing 
\[
\kappa:= \min\left\{\frac{1}{8 \tau_{\rm res}}, \frac{\Lambda +  S_0}{c_m}\right\},
\]
we have
\begin{align*}
& \E \|\delta U(t)\|_H^2  \leq  \E \|\delta U(s)\|_H^2  - \kappa \int_{s}^t  \E \|\delta U(r)\|_H^2 \, \dd r,
\end{align*}
and the conclusion follows from Grönwall's inequality.

\noindent
(ii)
By the evolution property of the solution,
\[
U(t,s_1,\xi)
=
U\bigl(t,s_2,U(s_2,s_1,\xi)\bigr).
\]
Hence,
\[
U(t,s_1,\xi)-U(t,s_2,\xi)
=
U\bigl(t,s_2,U(s_2,s_1,\xi)\bigr)-U(t,s_2,\xi).
\]
Applying the exponential contraction estimate (i) with $s=s_2$, $\xi_1=U(s_2,s_1,\xi)$, and $\xi_2=\xi$, we obtain
\begin{align*}
\E \|U(t,s_1,\xi)-U(t,s_2,\xi)\|_H^2
&=
\E \Bigl\|
U\bigl(t,s_2,U(s_2,s_1,\xi)\bigr)-U(t,s_2,\xi)
\Bigr\|_H^2 \\
&\le
e^{-\kappa (t-s_2)}
\E \|U(s_2,s_1,\xi)-\xi\|_H^2.
\end{align*}
Lemma \ref{Le:exp_cont} is proved.
\end{proof}

We proceed now with the proof of Theorem \ref{th:inv_measure_truncated}.
For $t \ge s_2$, by Lemma \ref{Le:exp_cont} and (\ref{eq:better-estimate-v_L^2}), there exist $\kappa, \nu >0$ such that
\begin{align*}
& \E \|U(t,s_1,\xi) - U(t,s_2,\xi)\|_H^2 \leq e^{{-\kappa}(t-s_2)} \E \|U(s_2,s_1,\xi) - \xi\|_H^2 \leq e^{{-\nu}(t-s_2)} C (1 + \|\xi\|_H^2).
\end{align*}
Letting $s_1 \to - \infty$, we conclude that there exists $\eta_t(\xi) \in L^2(\Omega, H_\infty)$ such that 
\begin{equation}\label{eq:inv_conv}
U(t,s,\xi) \to \eta_t (\xi) \,\, \mbox{in}\,\, L^2(\Omega, H_\infty) \text{ as } s \to -\infty.
\end{equation}
Since $H_\infty$ is closed in $H$ and $U(t,s,\xi)$ takes values in $H_\infty$ $\PP$-a.s., $\eta_t$ takes values in $H_\infty\,$ $\PP$-a.s.. 
Let us show that $\eta_t$ is independent of $\xi$. Indeed,
\begin{align*}
0\le \mathbb{E}\,\|\eta_t(\xi_1)-\eta_t(\xi_2)\|_H
&\le
\mathbb{E}\,\|\eta_t(\xi_1)-U(t,s,\xi_1)\|_H +
\mathbb{E}\,\|U(t,s,\xi_1)-U(t,s,\xi_2)\|_H\\
&+
\mathbb{E}\,\|U(t,s,\xi_2)-\eta_t(\xi_2)\|_H
\;\rightarrow 0, \quad s\to -\infty.
\end{align*}
We now define a probability measure on $H_\infty$ as
$\mu := P \circ \eta_0^{-1}$,
i.e.
\begin{equation}\label{def_inv_meas}
\mu(A):=
\mathbb{P}
\bigl(
\{\omega\in\Omega:\eta_0(\omega)\in A\}
\bigr).
\end{equation}
For any bounded Lipschitz function $\varphi:H_\infty \to \R$, using \eqref{eq:inv_conv}, we have
\begin{equation}\label{weak_meas}
\left|
P_{s,t}\varphi(\xi)
-
\mathbb{E}\!\left[\varphi(\eta_t)\right]
\right|
\le
\mathbb{E}
\left|
\varphi(U(t,s,\xi))
-
\varphi(\eta_t)
\right|
\le
L_\varphi\,
\mathbb{E}
\|U(t,s,\xi)-\eta_t\|_H
 \to 0, \quad s\to -\infty.
\end{equation}
Setting $t = 0$ in \eqref{weak_meas}, we conclude that
\begin{equation}\label{weak_meas1}
P_{s,0}(\xi, \cdot) \to \mu \,\, \mbox{ weakly in}\,\, \mathcal{M}_1(H_\infty), \ \ \mbox{as}\,\, s \to -\infty.
\end{equation}
Further, for a bounded Lipschitz
$\varphi:H_\infty\to\mathbb{R}$, it follows from \eqref{weak_meas1} that
\[
\lim_{s\to-\infty}
\int_{H_\infty}
P_{0,t}\varphi(y)\,
P_{s,0}(\xi,\dd y) = \int_{H_\infty} P_{0,t}\varphi(y)\,\mu(\dd y).
\]
On the other hand, using the semigroup property of $P_{s,t}$, namely,
\[
P_{s,0} P_{0,t} = P_{s,t} = P_{0,t-s} = P_{s-t,0}
\]
by \eqref{weak_meas1}, we have
\[
\lim_{s\to-\infty}
\int_{H_\infty}
P_{0,t}\varphi(y)\,
P_{s,0}(\xi,\dd y) = \lim_{s\to-\infty}
\int_{H_\infty} \varphi(y)
P_{s-t,0}(\xi, \dd y) \, =
\int_H
\varphi(y)\,\mu(\dd y),
\]
which implies that 
\[
\int_{H_\infty}
P_{0,t}\varphi(y)\,\mu(\dd y)
=
\int_{H_\infty}
\varphi(y)\,\mu(\dd y) \ \Longleftrightarrow \langle \varphi, P_{0,t}^{*}\mu\rangle
=
\langle P_{0,t}\varphi,\mu\rangle,
\]
implying that $\mu$ is an invariant measure 
$
P_{0,t}^{*}\mu
=
\mu$.

To show its uniqueness, take any invariant measure $\nu$. Since $\nu$ is invariant, for any bounded and Lipschitz $\varphi:{H_\infty} \to \R$,
\[
\int_{H_\infty} \varphi(y)\,\nu(\dd y)
=
\int_{H_\infty} P_{0,t}\varphi(y)\,\nu(\dd y).
\]
Since $\nu({H_\infty})=1$,
using Proposition 4.3.2 in \cite{liu2015stochastic}, we have
\begin{multline*}
\left|
P_{0,t}\varphi(\xi)
-
\int_{H_\infty} \varphi(y)\,\nu(\dd y)
\right|
=
\left|
\int_{H_\infty}
(P_{0,t}\varphi(\xi)-P_{0,t} \varphi(y))
\,\nu(\dd y)
\right| \le
L_{\varphi}\,
e^{-\frac{\kappa}{2}t}
\int_{H_\infty}
\|\xi-y\|_H
\,\nu(\dd y).
\end{multline*}
Since
\[
\int_{H_\infty}
\|\xi-y\|_H
\,\nu(\dd y) < \infty,
\]
we conclude that 
\[
\left|
P_{0,t}\varphi(\xi)
-
\int_{H_\infty} \varphi(y)\,\nu(\dd y)
\right| \to 0 \ \ \text{ as } t \to + \infty,
\]
implying that the invariant measure $\mu$ defined in \eqref{def_inv_meas} is unique. 
Ergodicity of the invariant measure follows from Theorem 11.11 \cite{da1996ergodicity}. The regularity of the support is proved in the same way as in Theorem \ref{th:existence_inv_measure}.

Let $W=(W_t)_{t\in\mathbb{R}}$ be a two-sided Wiener process on a filtered probability space
$(\Omega,\F_t,\mathbb P)$. We proceed with extending the notion of a variational solution from Definition \ref{def:variational} for $t \in \R$.
\begin{definition}
    An $\F_t$-adapted, $H$-valued stochastic process $U(t)$ is called a global variational solution of \eqref{eq:orig-prob-vector_kavian_truncation} if for any $-\infty<s<t<\infty$ we have $U \in L^2((s,t) \times \Omega; V)$ and 
    \begin{align*}
\label{eq:global_sol}
U(t)=U(s) + \int_s^t (\mathbb{A}(U(r)) + P(r))\, \dd r + \int_s^t B(U(r))\, \dd W_r \quad \PP-\mbox{a.s.}.
\end{align*}
\end{definition}

\begin{lemma}[Existence and uniqueness of a stationary solution]
\label{lm:stationary_sol}
\noindent
\begin{itemize}
\item[{[i]}] The process $U^\ast(t):=\eta_t \in H_\infty$ $\PP$-a.s. constructed as the pullback limit of solutions $U(t,s,\xi)$ in \eqref{eq:inv_conv} is a global variational solution of \eqref{eq:orig-prob-vector_kavian_truncation}.
\item[{[ii]}] $U^\ast$ is a stationary process, i.e. for every $n\in\mathbb N$,
every $t_1,\ldots,t_n\in\mathbb{R}$, and every $h\in\mathbb{R}$,
\[
\mathcal L(U^\ast({t_1+h}),\ldots,U^\ast({t_n+h}))
=
\mathcal L(U^\ast({t_1}),\ldots,U^\ast({t_n})).
\]
\item[{[iii]}] $U^\ast$ satisfies global in time bound
\begin{align}
\label{est_stationary_sol}
    \sup_{t\in \mathbb R} \E \|U^\ast(t)\|_{Z}^2 < \infty,
\end{align}
and is a unique stationary solution of \eqref{eq:orig-prob-vector_kavian_truncation}, that is pathwise unique in the class of stationary global variational solutions.
\end{itemize}
\end{lemma}
\begin{proof}
{[i]} For fixed $t\in\mathbb{R}$ and $s<t$, $\eta_t$ is $\F_t$-measurable as the limit in $L^2(\Omega;H)$ of 
$U(t,s,\xi)$. In particular, $\eta_r$ is an admissible initial condition at every time
$r\in\mathbb{R}$. By the flow property,
\[
U(t,s,\xi)=U(t,r,U(r,s,\xi)), \quad s<r.
\]
Letting $s\to-\infty$, we have
\[
U(t,s,\xi)\to \eta_t, \quad U(r,s,\xi)\to \eta_r, \quad
\qquad \text{in }L^2(\Omega;H).
\]
By continuous dependence of solutions on the initial data (see Lemma \ref{Le:exp_cont}),
\[
U(t,r,U(r,s,\xi))\to U(t,r,\eta_r)
\qquad \text{in }L^2(\Omega;H).
\]
Hence
\begin{align}
\label{eta_t}
\eta_t=U(t,r,\eta_r)
\qquad \text{in }L^2(\Omega;H),
\end{align}
and therefore $\eta_t=U(t,r,\eta_r)$
$\PP$-a.s.. Note that the equality above is proved for each $t\in \mathbb{R}$ $\PP$-a.s., that is for each $t$ there exists a null set $N_t$ depending on $t$, such that $\eta_t(\omega)=U(t,r,\eta_r)(\omega)$ holds for $\omega \notin N_t$. Since $t \mapsto U(t,r,\eta_r)$ has continuous $H$-paths, and thus the process $\widetilde{\eta}_t = U(t,r,\eta_r)$ for all $t\ge r$ is a continuous modification of $\eta_t$ on $[r, +\infty)$, for which we keep the same notation $\eta_t$.
In order to obtain an equality outside one common null set, we can use the rational time argument. Indeed, taking $t$ first in the countable set $\mathbb{Q}\cap[r,T]$, we obtain a
single null set outside which \eqref{eta_t} holds.
Since both processes have now $H$-continuous paths on $[r,T]$, the equality extends to every $t\in[r,T]$. Thus, after modification on a null set, \eqref{eta_t} holds for all $t\in[r,T]$.
Consequently, the restriction $\eta|_{[r,T]}$ is the variational solution of \eqref{eq:orig-prob-vector_kavian_truncation} on $[r,T]$ with initial data given by the value of the process at the left endpoint.



\medskip
\noindent
{[ii]} Let
$n\in\mathbb N$, $t_1\le \cdots\le t_n$, and $h\in\mathbb{R}$. Let $\Phi\in \operatorname{Lip}_b(H_\infty^n)$. For $s<t_1+h$,
\[
\bigl(
U(t_1+h,s,\xi),\ldots,U(t_n+h,s,\xi)
\bigr)
\to
(\eta_{t_1+h},\ldots,\eta_{t_n+h}) \,\,\mbox{in}\, \,L^2(\Omega; H^n), \quad s\to -\infty.
\]
Therefore
\[
\mathbb E\Phi(\eta_{t_1+h},\ldots,\eta_{t_n+h})
=
\lim_{s\to-\infty}
\mathbb E
\Phi
\bigl(
U(t_1+h,s,\xi),\ldots,U(t_n+h,s,\xi)
\bigr).
\]
By time-homogeneity and stationary increments of the two-sided $W_t$,
\[
\mathcal L
\bigl(
U(t_1+h,s,\xi),\ldots,U(t_n+h,s,\xi)
\bigr)
=
\mathcal L
\bigl(
U(t_1,s-h,\xi),\ldots,U(t_n,s-h,\xi)
\bigr).
\]
Hence
\[
\begin{aligned}
&\mathbb E\Phi(\eta_{t_1+h},\ldots,\eta_{t_n+h})
=
\lim_{s\to-\infty}
\mathbb E
\Phi
\bigl(
U(t_1,s-h,\xi),\ldots,U(t_n,s-h,\xi)
\bigr).
\end{aligned}
\]
Passing to the limit as $s\to-\infty$ yields
\[
\mathbb E\Phi(\eta_{t_1+h},\ldots,\eta_{t_n+h})
=
\mathbb E\Phi(\eta_{t_1},\ldots,\eta_{t_n}).
\]
Since bounded Lipschitz test functions determine probability measures on the
Polish space $H_\infty^n$, we conclude that
\[
\mathcal L(\eta_{t_1+h},\ldots,\eta_{t_n+h})
=
\mathcal L(\eta_{t_1},\ldots,\eta_{t_n}).
\]
Thus $\eta=(\eta_t)_{t\in \mathbb R}$ is stationary. In particular, $\mathcal{L}(\eta_t)=\mu$ for all $t\in \mathbb{R}$.

\medskip
\noindent
{[iii]} Finally, to derive \eqref{est_stationary_sol}, we note that the invariant measure $\mu$ has a finite second moment \eqref{eq:support}, and therefore
\begin{align*}
    \sup_t \E \|\eta_t\|_{Z}^2 
    = \int_{H_\infty} \|y\|_Z^2 \, \mu(\dd y) < \infty. 
\end{align*}
The uniqueness of a stationary solution follows now from \eqref{est_stationary_sol} and Lemma \ref{Le:exp_cont}. Indeed, suppose there exist two stationary solutions $U_1^\ast$ and $U_2^\ast$. Then
\begin{align}
\label{eq:aux5}
\E\|U_1^\ast(t)-U_2^\ast(t)\|_H^2 \le e^{-\kappa(t-s)} \E \|U_1^\ast(s)-U_2^\ast(s)\|_H^2, \quad s\le t.
\end{align}
Using the stationarity
\[
\E \|U_1^\ast(s)-U_2^\ast(s)\|_H^2 \leq 2(\E \|U_1^\ast(s)\|^2 +\|U_2^\ast(s)\|_H^2) = 2 (\E \|U_1^\ast(0)\|^2 +\|U_2^\ast(0)\|_H^2),
\]
and passing to the limit in \eqref{eq:aux5} as $s\to -\infty$, we obtain $U_1^\ast(t)=U_2^\ast(t)$ for all $t\in \mathbb{R}$, $\PP$-a.s.. 
Lemma~\ref{lm:stationary_sol} and thus, Theorem \ref{th:inv_measure_truncated} is proved.
\end{proof}

\end{proof}

\section{Numerical example}\label{sec:numerics}

In order to illustrate how the solution $(v,w)$ of \eqref{eq:stochastic_kavian} behaves, we compute it numerically on $G=(0,1)^2$ being the union of two non-overlapping domains $G_{\rm i}$ and $G_{\rm e}$, where $G_{\rm i}= B_{\frac{1}{4}}(\frac{1}{2}, \frac{1}{2})$ is the circular cell. Periodic boundary conditions are imposed on the external boundary $\partial G$, and zero initial conditions $v_0=w_0=0$ are assumed. When problem \eqref{eq:micro_deterministic_on_Gamma} is stated in a domain with multiple cells representing a tissue, a dimension analysis can be performed to obtain a dimensionless form of equations (see \cite{geback2025derivation}). We assume that this dimension analysis is already done, and will use the dimensionless parameters given in Table~\ref{tab:dimensionless_coeffs}. The function $\beta$ in \eqref{eq:f} is given by
\begin{align*}
    \beta(\xi) = \frac{1+\tanh(k_{\text{ep}}(|\xi | - V_{\text{th}}))}{2}, \quad k_{\text{ep}} = 40, \quad V_{\text{th}} = 2.5.
\end{align*}
\begin{table}[htp]
\centering
\begin{tabular}{l|l|l}
    Parameter & Symbol & Value \\[1mm]
    \hline
    Intracellular conductivity & $\sigma_{\rm i}$  & {$0.239$} \\[1mm]
    Extracellular conductivity & $\sigma_{\rm e}$ & {2.632}\\[1mm]
    Membrane conductivity before EP &$S_0$ & $1$ \\[1mm]
    Max increase in membrane conductivity & $S_1$ & $10001$ \\[1mm]
    EP  time & $\tau_{\text{ep}}$ & 1 \\[1mm]
    Resealing time & $\tau_{\text{res}}$ & 10 \\[1mm]
    Membrane capacitance & ${c}_m$ & 1 \\
    \hline
\end{tabular}
\caption{Dimensionless coefficients in \eqref{eq:stochastic_kavian}.}
\label{tab:dimensionless_coeffs}
\end{table}

We include the external excitation on $\Gamma$ in the form $\sigma_{\rm i} \nabla p\cdot \n=\sigma_{\rm i} g\, e_1\cdot \n$, that is the applied electric field is constant and acts in the direction of $x_1$. The problem to solve becomes
\begin{align}
\label{eq:micro_example}
\begin{alignedat}{2}
-{\rm div}(\sigma (\nabla u+g\, e_1))  &= 0 &\quad& \mbox{in}\,\, (0,T]\times G_{\rm i} \cup G_{\rm e}, \\
\sigma_{\rm i} (\nabla u_{\rm i} + g\, e_1)  \cdot \n &= \sigma_{\rm e} (\nabla u_{\rm e} + g\, e_1)  \cdot \n =: I_m &\quad&\mbox{on}\,\, (0,T]\times\Gamma, \\
c_m \dd [u]&= (-I_m - S_m([u],w))\dd t + b([u]) \dd W_t
&\quad&\mbox{on}\,\,(0,T]\times\Gamma,
\\
\dd w&= f([u], w)\, \dd t&\quad&\mbox{on}\,\,(0,T]\times\Gamma, \\
[u](0,x) &=0, \quad w(0,x)=0&\quad&\mbox{on}\,\,\Gamma,\\
u_{\rm e} & \,\,\text{is}\,\, (0,1)^2\text{-periodic}. & &
\end{alignedat}
\end{align}

To derive a weak formulation to be used in numerical computations, we follow the ideas in \cite{kuchta2021solving}. We multiply \eqref{eq:micro_example} by test functions $\phi_i \in H^1(G_{\rm i})$, $\phi_e \in H^1(G_{\rm e})$ and integrate by parts to yield the problem: find $u_{\rm i}$, $u_{\rm e}$, and $I_m$ such that
\begin{equation}
    \begin{aligned}
    &\int_{G_{\rm i}} \sigma_{\rm i} \nabla u_{\rm i} \cdot \nabla \phi_i \, \dd x - \int_\Gamma I_m \,  \phi_i \, \dd S = -\int_{\Gamma} \sigma_{\rm i} \, g e_1 \cdot \n\, \phi_i \, \dd x, \\
    &\int_{G_{\rm e}} \sigma_{\rm e} \nabla u_{\rm e} \cdot \nabla \phi_e \, \dd x + \int_\Gamma I_m \,  \phi_e \, \dd S = \int_{\Gamma} \sigma_{\rm e} \, g e_1 \cdot \n \, \phi_e \, \dd x, \\
    & I_m = \sigma_{\rm i} (\nabla u_{\rm i} + g\, e_1)  \cdot \n, \\
    & \int_{\Gamma}(
    u_{\rm i}-u_{\rm e}) j_m dS = \int_{\Gamma} v j_m dS,
    \end{aligned}
    \label{eq:weak_form_numerics}
\end{equation}
where $v$ satisfies the dynamic condition in \eqref{eq:micro_example} and $j_m \in H^{1/2}(\Gamma)$. Note that the mapping of the jump $v$ to the boundary flux $I_m$ corresponds to the operator $\mathcal{A}$ in \eqref{def:L}.
Given a jump $v = u_{\rm i}\big|_{\Gamma}-u_{\rm e}\big|_{\Gamma}$, the solution of \eqref{eq:weak_form_numerics} gives $I_m[v]$, which we can use to solve the SDE in the third equation in \eqref{eq:micro_example}
\begin{align*}
    c_m \dd v = - \big(I_m[v] + S_m(v,w)\big)\dd t + b(v)\dd W_t,
\end{align*}
using a semi-implicit Euler method. 
To derive an implicit solver for $I_m$, we consider an equidistant time discretization $0 = t_0 < t_1 < \ldots < t_n = T, \, \, n \in \mathbb{N}$ with step size $\Delta t=t_{k+1}- t_k>0$, and
use time-splitting on the third equation in \eqref{eq:micro_example}, which is split into  
\begin{align*}
    \displaystyle c_m \frac{\tilde v_k(x_s) - v(t_k, x_s)}{\Delta t} &= -I_m[\tilde v_k], \\
    \displaystyle
    c_m \frac{v(t_{k+1}, x_s) - \tilde v_k(x_s)}{\Delta t} &= -S_m(v(t_{k+1}, x_s),w(t_k,x_s)) + b(v(t_k,x_s)) \sqrt{\frac{1}{\Delta t}} \zeta,
\end{align*}
for each mesh point $x_s \in \Gamma$, with $\zeta \sim \mathcal{N}(0,1)$.
We then insert $\tilde v_k = v(t_k) - \frac{\Delta t}{c_m} I_m$
for $v$ in \eqref{eq:weak_form_numerics},
and write the equations in matrix form as
\begin{align}
     a(U,\Phi) = l_k(\Phi),
     \label{eq:linear_system}
\end{align}
 where for $U = (u_{\rm i}, u_{\rm e}, I_m )$ and $\Phi= (\phi_i, \phi_e, j_m)$, we define
\begin{align*}
    a(U,\Phi) = \begin{pmatrix}
        \int_{G_{\rm i}} \sigma_{\rm i} \nabla u_{\rm i} \cdot \nabla \phi_i \, \dd x -\int_\Gamma I_m \phi_i \, \dd S \\[2mm]
        \int_{G_{\rm e}} \sigma_{\rm e} \nabla u_{\rm e} \cdot \nabla \phi_e \, \dd x + \int_\Gamma I_m \phi_e \, \dd S \\[2mm]
        \int_\Gamma u_{\rm i} j_m \, \dd S - \int_\Gamma u_{\rm e} j_m \, \dd S + \int_\Gamma \frac{\Delta t}{c_m} I_m j_m \, \dd S
\end{pmatrix},
\quad
l_k(\Phi) = \begin{pmatrix}
                - \int_{\Gamma}  \sigma_{\rm i}g\, e_1 \cdot \n \, \phi_i \, \dd x \\[2mm]
                \int_{\Gamma}  \sigma_{\rm e} g\, e_1 \cdot \n \, \phi_e \, \dd x \\[2mm]
                \int_\Gamma v(t_k) j_m \, \dd S
    \end{pmatrix}.
\end{align*}

Discretizing this system by a choice of test functions $\Phi$ yields an implicit solver for $I_m$, which we use to obtain $\tilde v_k$. Applying a semi-implicit Euler--Maruyama method, the second time-splitting equation, an SDE, then yields
\begin{align*}
    v(t_{k+1}, x_s) = \frac{1}{1+ \frac{\Delta t}{c_m}(S_0 + S_1 w(t_k,x_s))} \Big( v(t_{k},x_s) - \frac{\Delta t}{c_m} I_m[\tilde v_k](x_s) + \frac{1}{c_m} b(v(t_k,x_s)) \sqrt{\Delta t}\, \zeta \Big)
\end{align*}
for $\Delta t = t_{k+1}- t_k$, each mesh point $x_s \in \Gamma$, and $\zeta \sim \mathcal{N}(0,1)$.
Similarly, we solve the ODE for $w$ using a semi-implicit Euler method for time step size $\Delta t = t_{k+1}- t_k$: 
\begin{align*}
    w(t_{k+1},x_s) = \frac{w(t_k,x_s) + \Delta t \, \tau_{\text{max}} \beta(v(t_k,x_s))}{1+ \Delta t \,\tau_{\text{max}}},
\end{align*}
with
\begin{align*} \tau_{\text{max}} =
\begin{cases}
\tau_{\text{ep}} \,\, \mbox{if} \,\, \beta(v(t_k,x_s)) \geq w(t_k,x_s),\\ 
\tau_{\text{res}} \,\, \mbox{if} \,\, \beta(v(t_k,x_s)) < w(t_k,x_s),
\end{cases}
\end{align*}
for each mesh point $x_s \in \Gamma$. 

The simulations were performed using the FEM package FEniCSx 0.9 \cite{BarattaEtal2023,AlnaesEtal2014} to discretize and solve \eqref{eq:linear_system} using CG1 elements. We use gmsh (version 4.14.1) to generate the mesh. 
Below, we present numerical solution $(v,w)$ for additive and multiplicative noise. 

In Figure~\ref{fig:contour_nonoise} one can see the results of the simulations of the electric potential $u(t,x)$ in the deterministic case, and in Figure~\ref{fig:contour_additive} the solution for additive noise.
\begin{figure}[ht]
  \centering
  \begin{subfigure}[t]{0.33\textwidth}
    \centering
    \includegraphics[width=\linewidth]{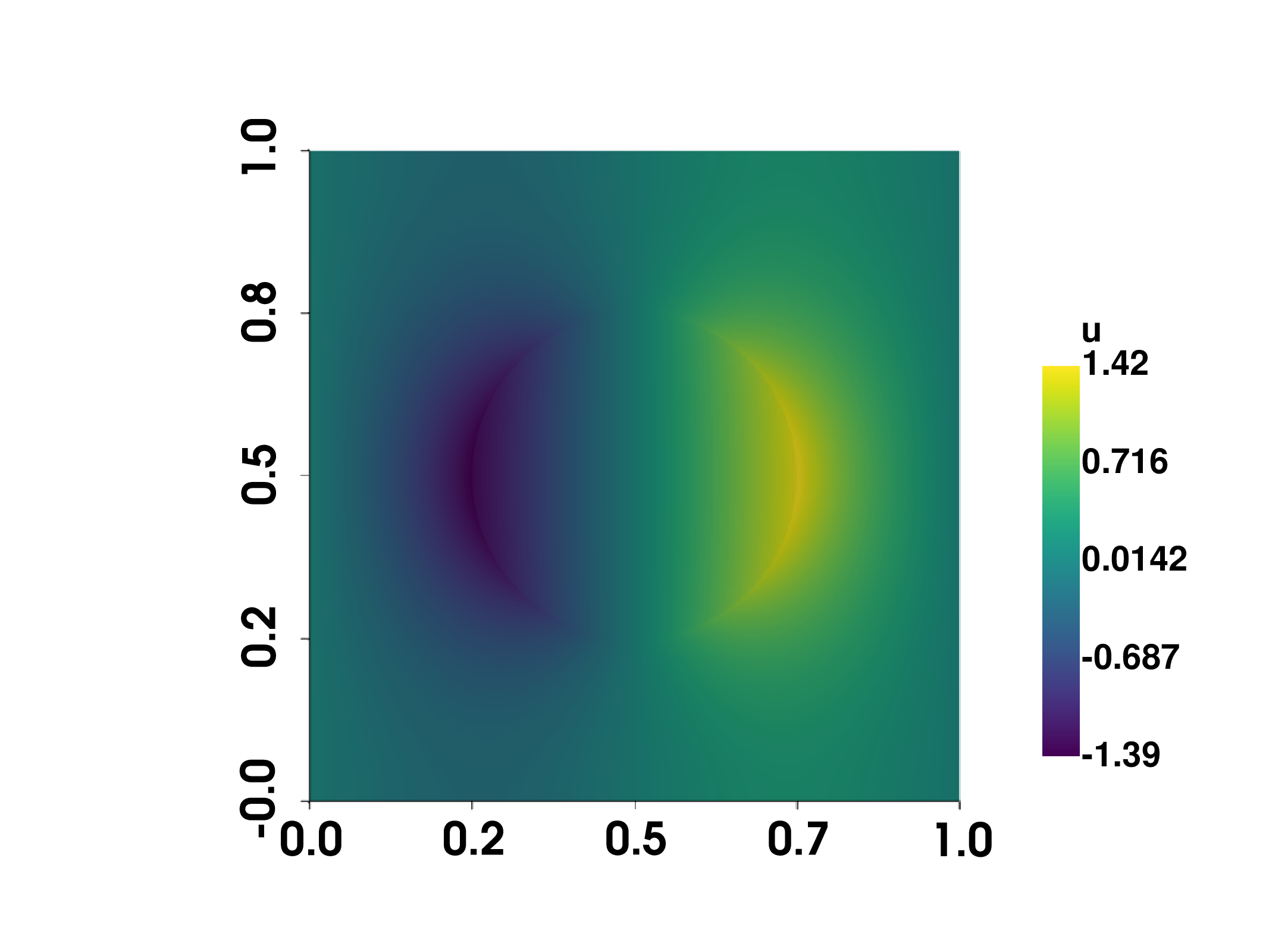}
    \caption{$t_0=0 \, \mu$s}
  \end{subfigure}\hfill
  \begin{subfigure}[t]{0.33\textwidth}
    \centering
    \includegraphics[width=\linewidth]{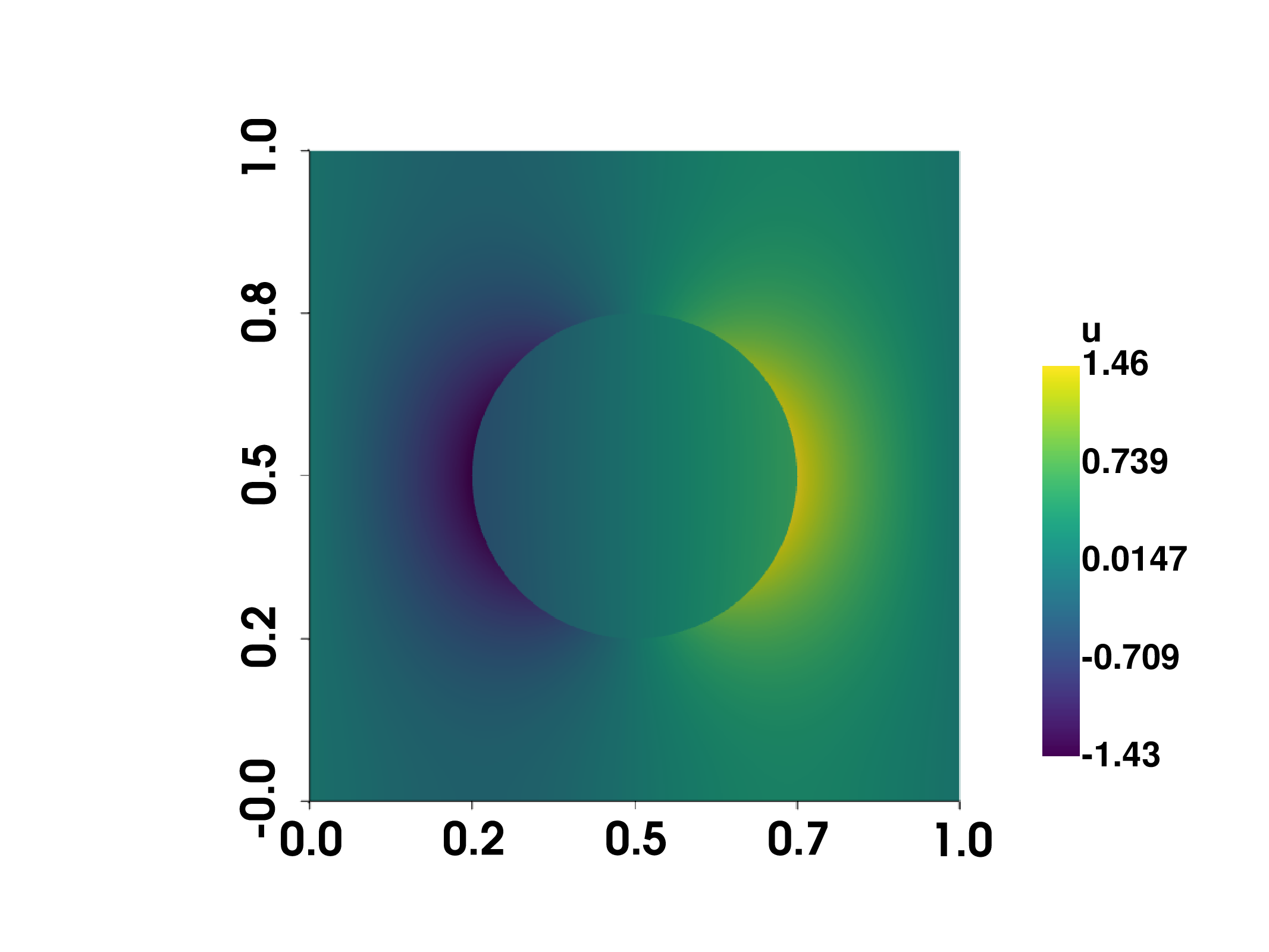}
    \caption{$t_1=0.25 \, \mu$s}
  \end{subfigure}\hfill
    \begin{subfigure}[t]{0.33\textwidth}
    \centering
    \includegraphics[width=\linewidth]{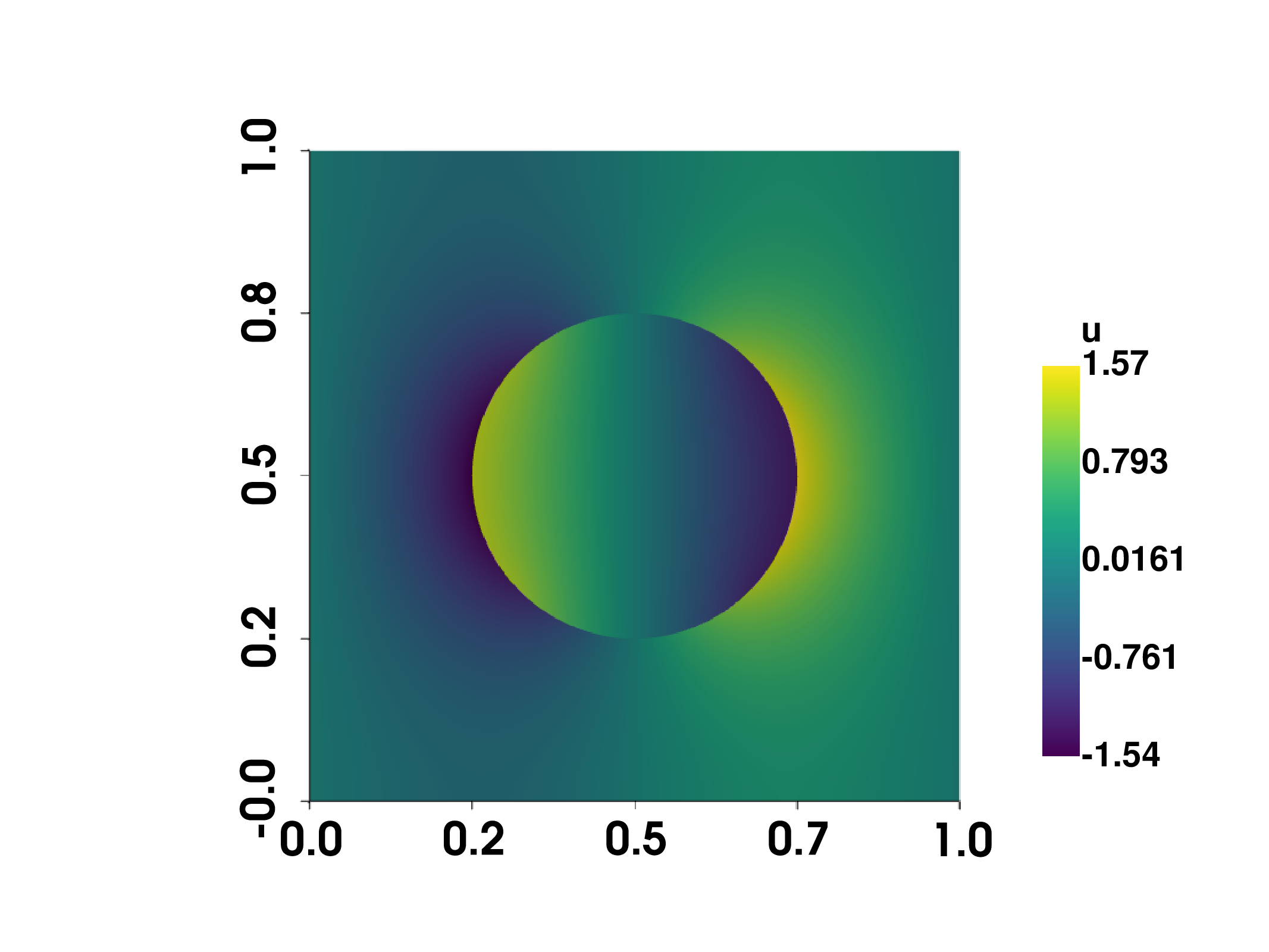}
    \caption{$t_2=1.875 \, \mu$s}
  \end{subfigure}\hfill
  \caption{Electric potential $u$ [V] in the periodicity cell $Y$ over time for the deterministic case.}
  \label{fig:contour_nonoise}
\end{figure}

\begin{figure}[ht]
  \centering
  \begin{subfigure}[t]{0.33\textwidth}
    \centering
    \includegraphics[width=\linewidth]{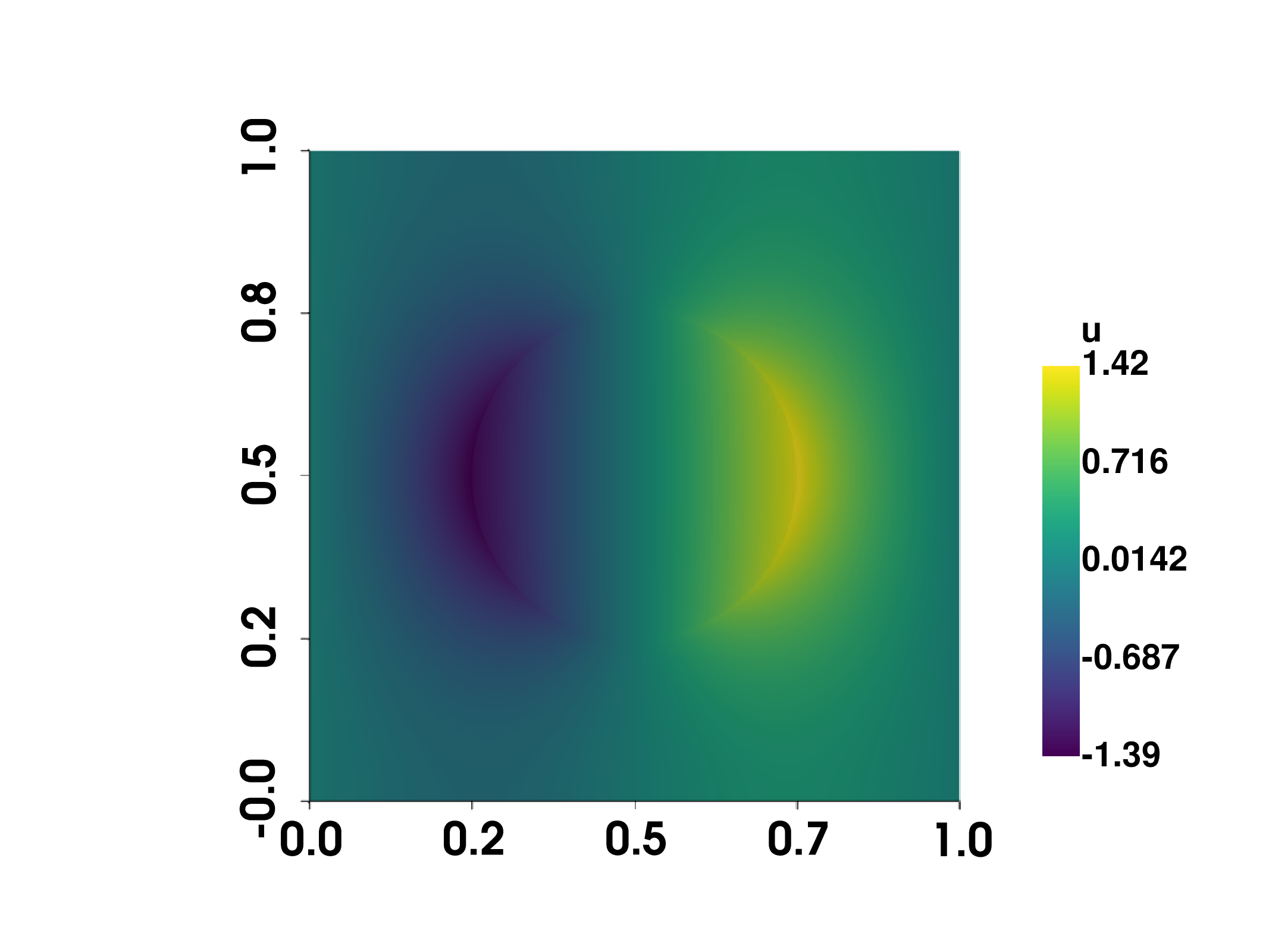}
    \caption{$t_0=0 \, \mu$s}
  \end{subfigure}\hfill
  \begin{subfigure}[t]{0.33\textwidth}
    \centering
    \includegraphics[width=\linewidth]{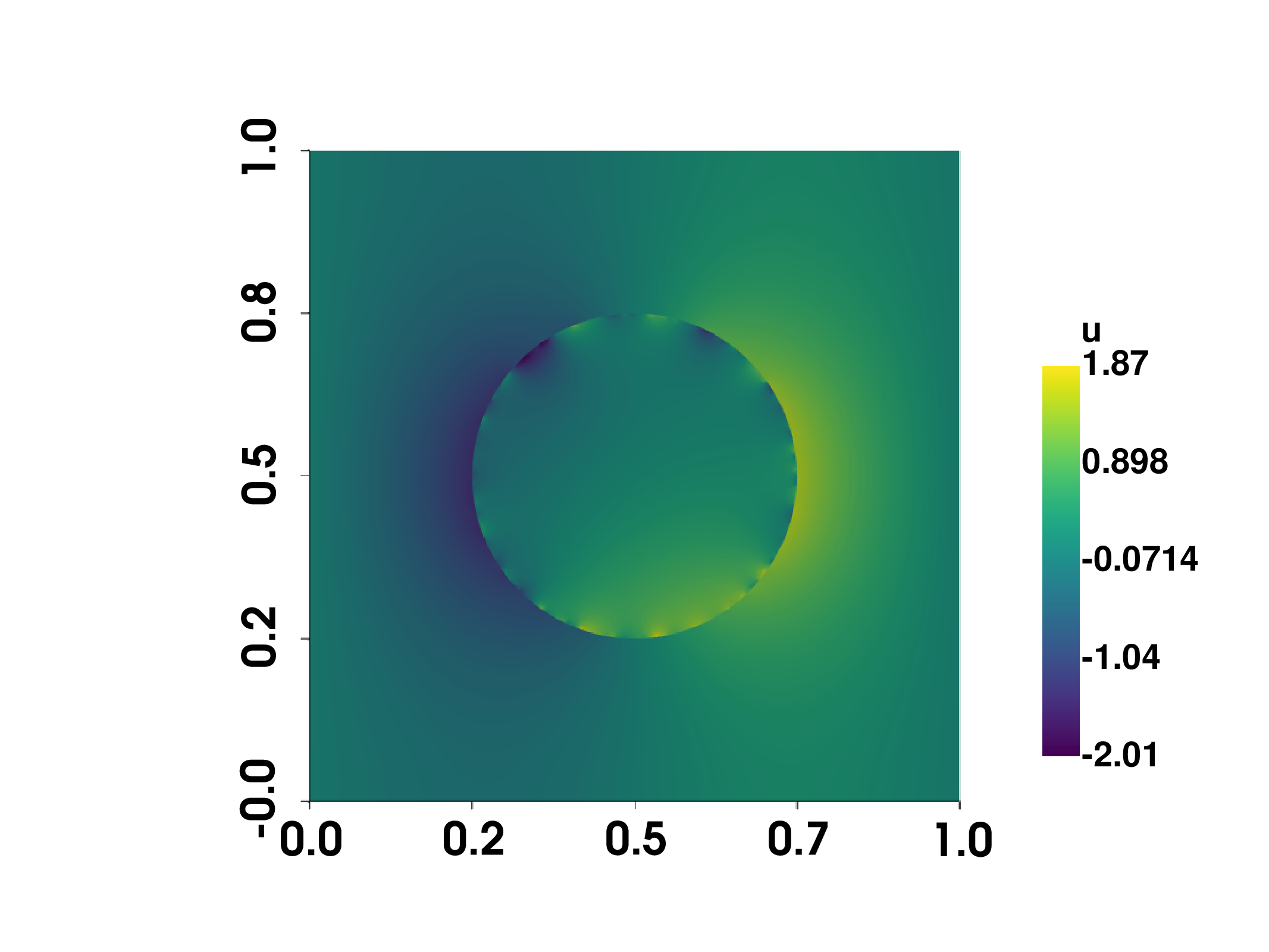}
    \caption{$t_1=0.25 \, \mu$s}
  \end{subfigure}\hfill
    \begin{subfigure}[t]{0.33\textwidth}
    \centering
    \includegraphics[width=\linewidth]{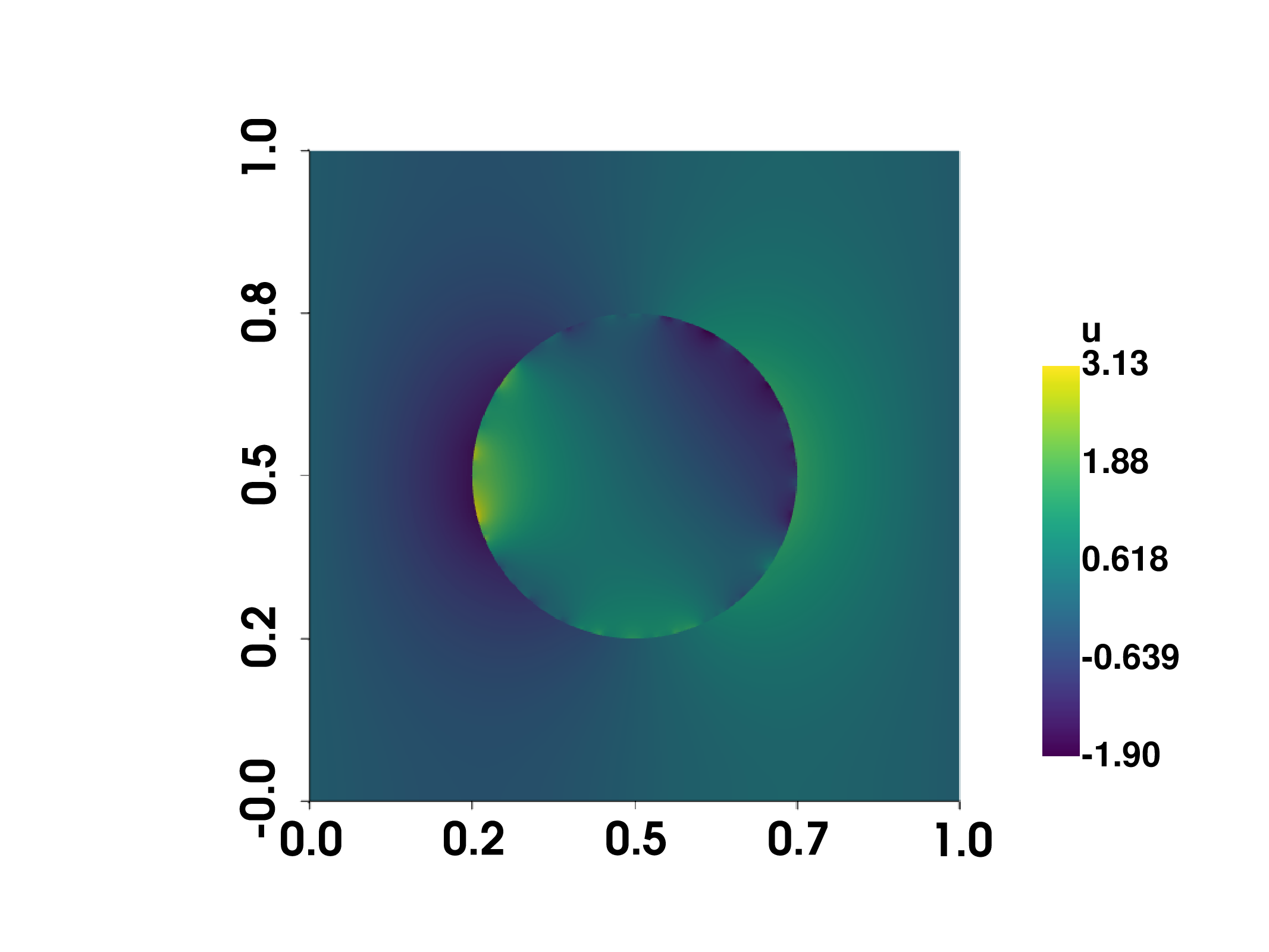}
    \caption{$t_2=1.875 \, \mu$s}
  \end{subfigure}\hfill
  \caption{Electric potential $u$ [V] in the periodicity cell $Y$ over time for additive noise $b(v) = 5$.}
  \label{fig:contour_additive}
\end{figure}















\subsection{Additive noise}
\begin{figure}[!htb]
  \centering
  \begin{subfigure}[t]{0.45\textwidth}
    \centering
    \includegraphics[width=\linewidth]{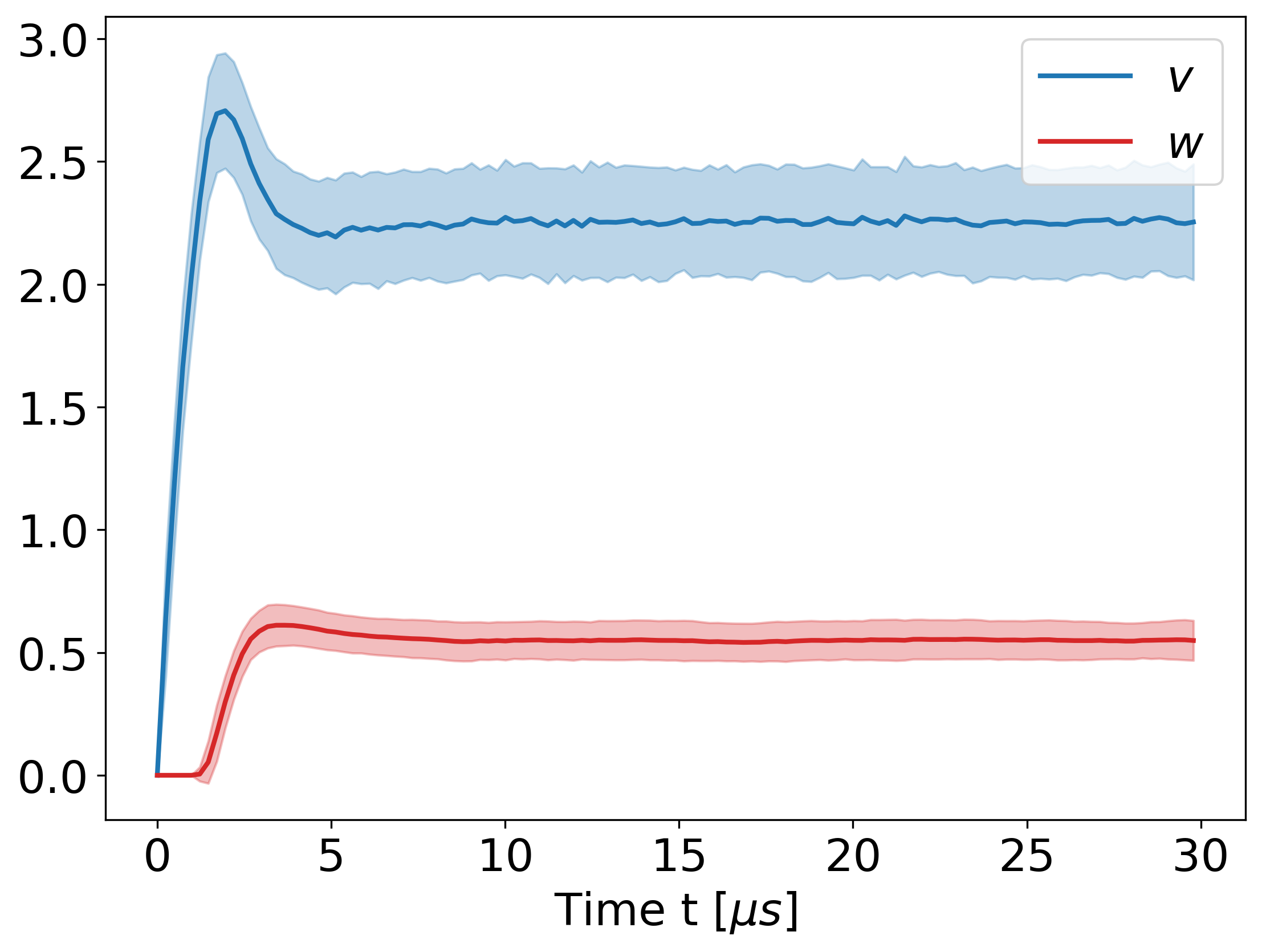}
    \caption{$v$ and $w$ at the pole $\theta=\pi$.}
    \label{fig:v_w_vs_t}
  \end{subfigure}\hfill
  \begin{subfigure}[t]{0.45\textwidth}
    \centering
    \includegraphics[width=\linewidth]{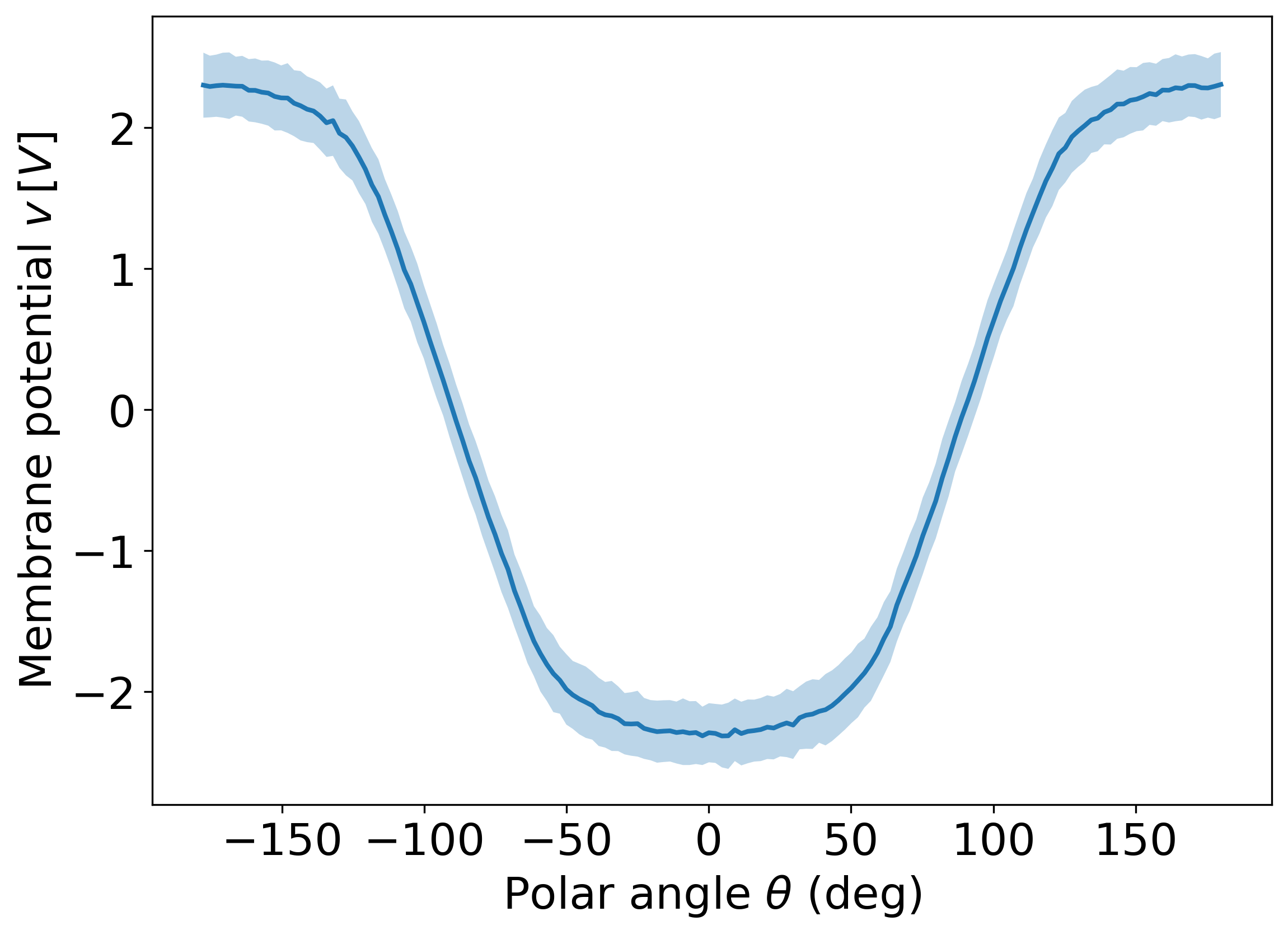}
    \caption{$v$ at $T=500 \, \mu$s.}
    \label{fig:v_vs_theta}
  \end{subfigure}\hfill
  \caption{Membrane potential $v$ and degree of porosity $w$ in the case of additive noise.}
  \label{fig:v_and_w_additive}
\end{figure}
We choose $E=\mathbb{R}$, $Q={\rm Id}$, and $W_t = \alpha \zeta(t)$ a real-valued standard Brownian motion. The noise amplitude is $\alpha=0.5$. The stochastic perturbation might be interpreted as a global temporal fluctuation of the membrane potential, spatially uniform along the cell boundary, or as random fluctuations of the applied voltage $g$. In Figure~\ref{fig:v_and_w_additive} we present the results of the simulations of the solution $(v,w)$ as a function of time at the pole of the cell corresponding to the polar angle $\theta=\pi$ for $t\in [0, 30] \, \mu$s. The shaded region, based on $500$ Monte-Carlo simulations, shows one standard deviation of the solution $v$. One can see that the potential grows rapidly and then, after reaching its maximum, decays due to the simultaneous growth of the porosity degree $w$, as it is seen in Figure \ref{fig:v_w_vs_t}. The electric potential at the final time $T_{\rm final}=300 \, \mu$s, as a function of the polar angle $\theta$, is shown in Figure \ref{fig:v_vs_theta}. One can see that it attains its maximum for $\theta=\pi$ and is close to zero at $\theta=\pm \pi/2$.
\begin{figure}[!htb]
  \centering
  \begin{subfigure}[t]{0.45\textwidth}
    \centering
    \includegraphics[width=\linewidth]{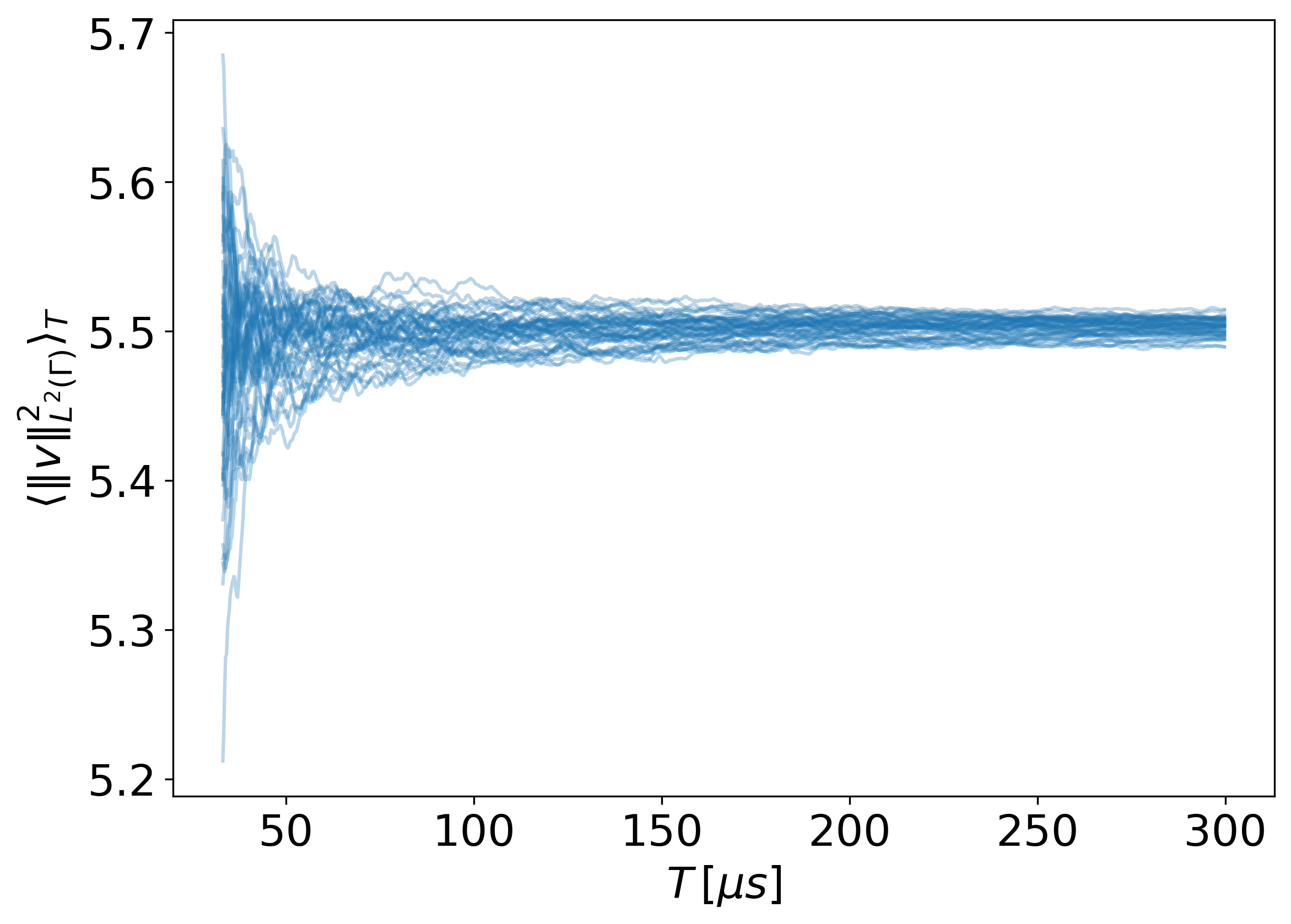}
    \caption{$\langle \|v\|_{L^2(\Gamma)}^2\rangle_{T}$ over $[1.1 T_{\rm burn-in}, T_{\rm final}]$.}
  \end{subfigure}\hfill
 \begin{subfigure}[t]{0.45\textwidth}
    \centering
    \includegraphics[width=\linewidth]{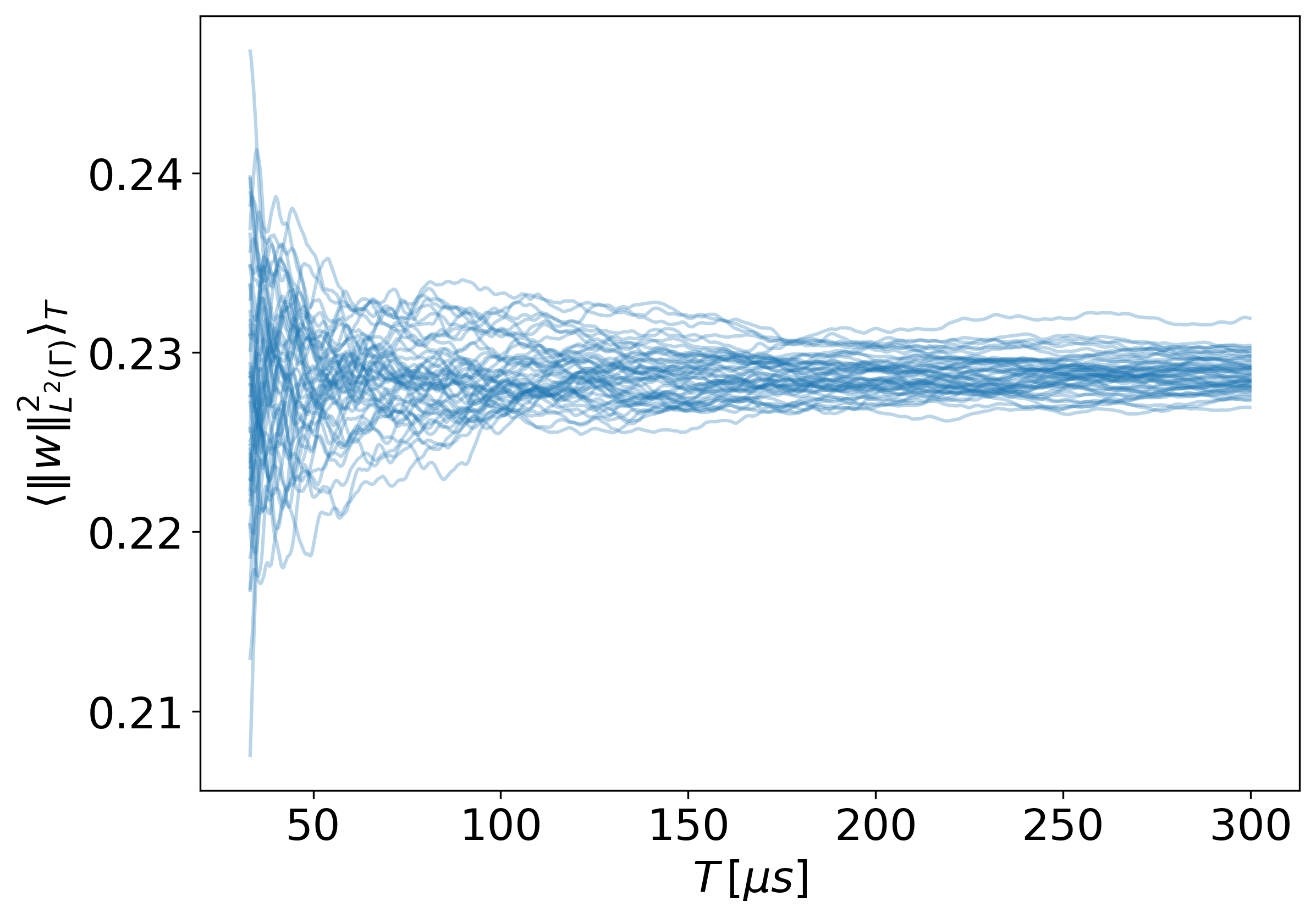}
    \caption{$\langle \|w\|_{L^2(\Gamma)}^2\rangle_{T}$ over $[1.1 T_{\rm burn-in}, T_{\rm final}]$.}
  \end{subfigure}
  \caption{Time averages in the case of additive noise. }
  \label{fig:time-average}
\end{figure}
\begin{figure}[!htb]
  \centering
  \begin{subfigure}[t]{0.45\textwidth}
    \centering
  \includegraphics[width=\linewidth]{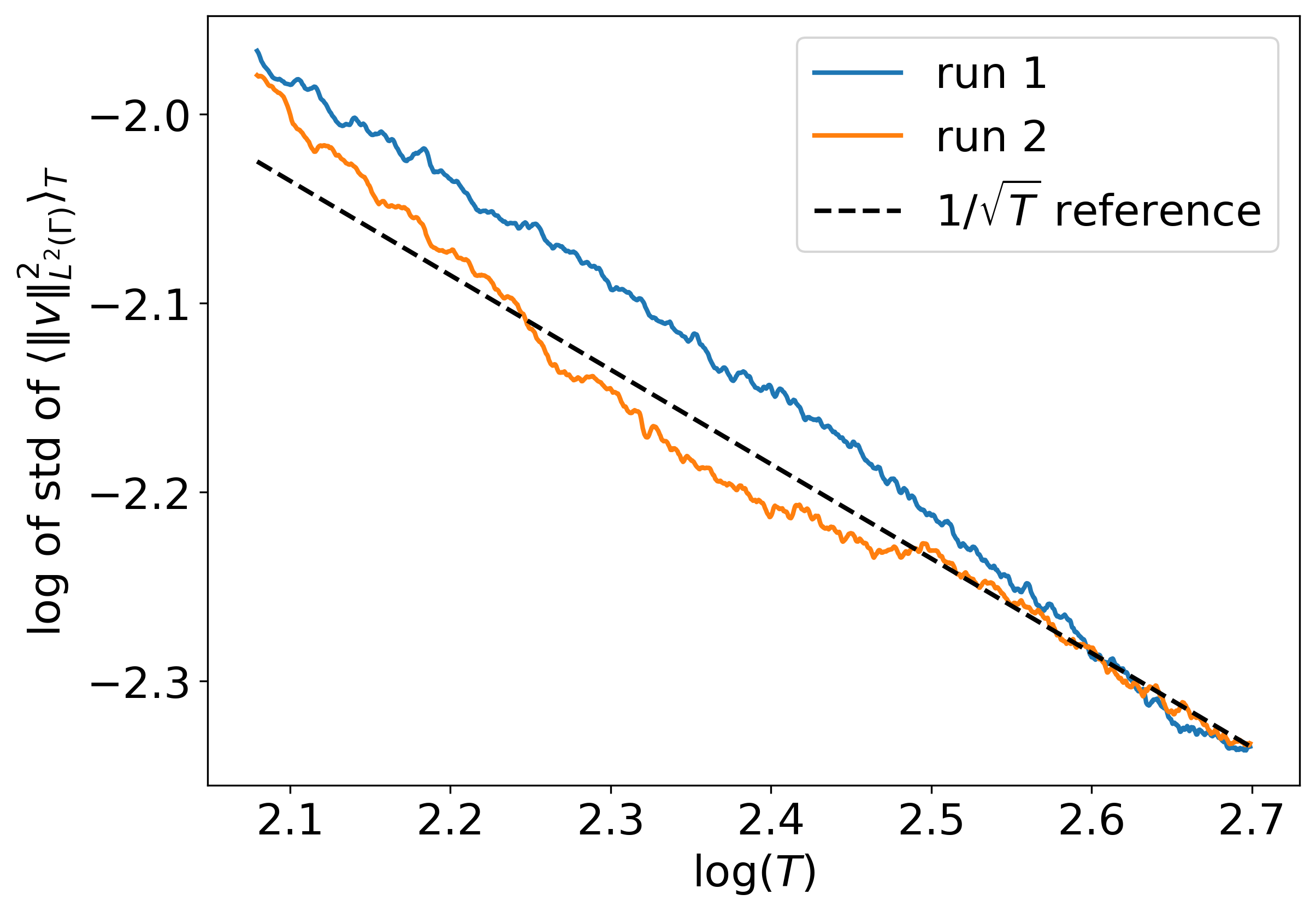}
  \end{subfigure}
  \hspace{1cm}
    \begin{subfigure}[t]{0.45\textwidth}
    \centering
    \includegraphics[width=\linewidth]{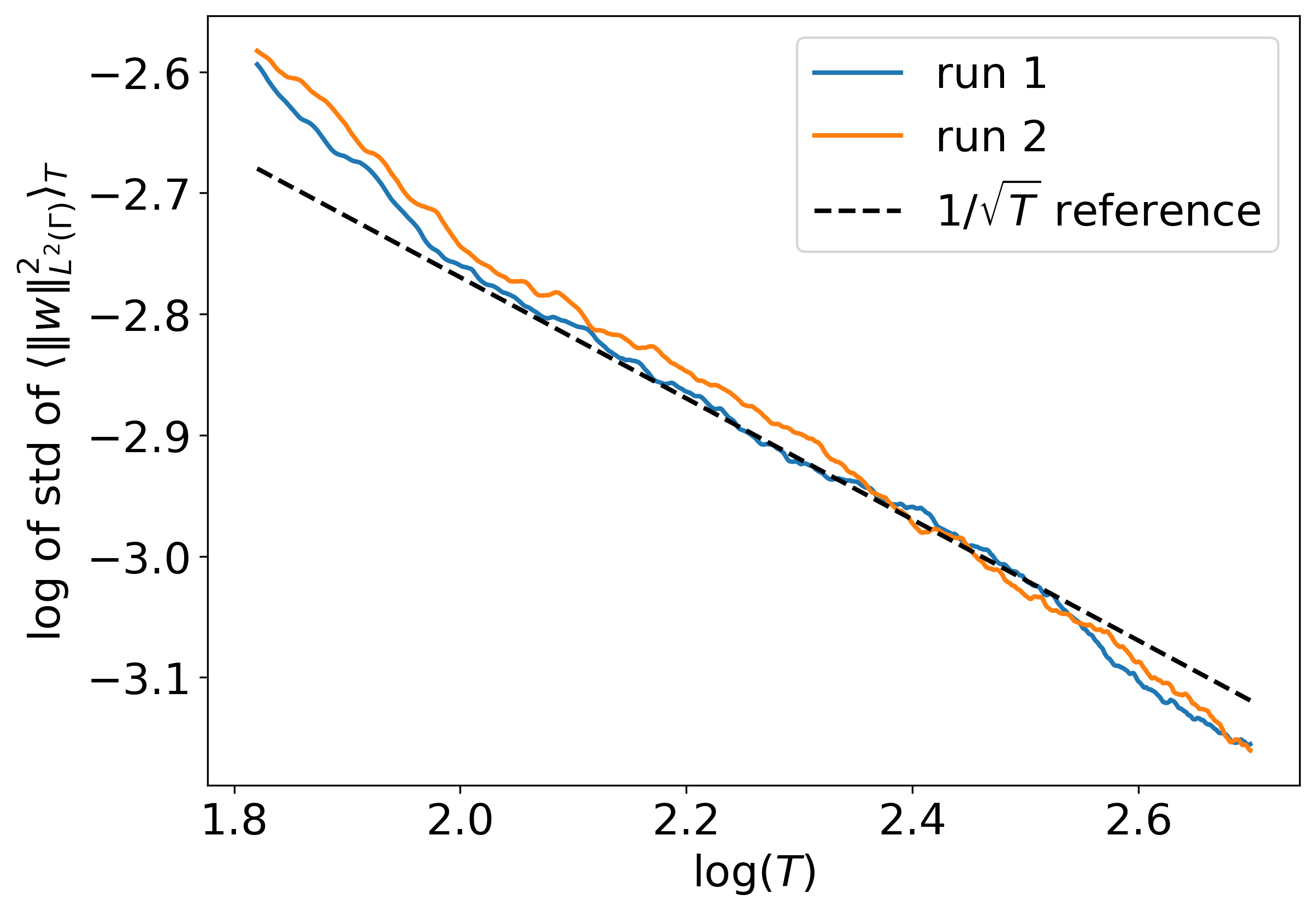}
\end{subfigure}
  \caption{Sample standard deviations of $\langle \|v\|_{L^2(\Gamma)}^2\rangle_{T}$ and $\langle \|w\|_{L^2(\Gamma)}^2\rangle_{T}$ in the case of additive noise for MC runs with 500 samples each, $T_{\rm burn-in}=30\, \mu$s, $T_{\rm final}=500\, \mu$s.}
  \label{fig:std_additive}
\end{figure}
In order to investigate numerically ergodicity of the system, we compute the time averages of $\|v(t,\cdot)\|_{L^2(\Gamma)}^2$ and $\|w(t,\cdot)\|_{L^2(\Gamma)}^2$:
\begin{align}
\label{eq:M_v_L2}
\begin{alignedat}{2}
&\langle \|v\|_{L^2(\Gamma)}^2\rangle_{T} = \frac{1}{T-T_{\rm burn-in}} \int_{T_{\rm burn-in}}^{T} \|v(t,\cdot)\|_{L^2(\Gamma)}^2\, \dd t,\\
&\langle \|w\|_{L^2(\Gamma)}^2\rangle_{T} = \frac{1}{T-T_{\rm burn-in}} \int_{T_{\rm burn-in}}^{T} \|w(t,\cdot)\|_{L^2(\Gamma)}^2\, \dd t.
\end{alignedat}
\end{align}
The time averages at the pole $\theta=\pi$ for $T_{\rm burn-in}=30\, \mu$s are shown in Figure \ref{fig:time-average} and are computed per realization.

The plots of the sample standard deviation of the time‑averaged $\|v\|_{L^2(\Gamma)}^2$ across trajectories (evaluated pointwise in time) for $\theta=\pi$ and the log-log plot are presented in Figure \ref{fig:std_additive}. One can see that the sample standard deviation decays, and the rate is approximately $-1/2$. Note that the sample standard deviation is random, and for independent Monte-Carlo runs we get clearly different curves, all of those share the same slope. 
These numerical simulations provide an indication for ergodicity.

\subsection{Multiplicative noise}
\begin{figure}[ht]
  \centering
  \begin{subfigure}[t]{0.45\textwidth}
    \centering
    \includegraphics[width=\linewidth]{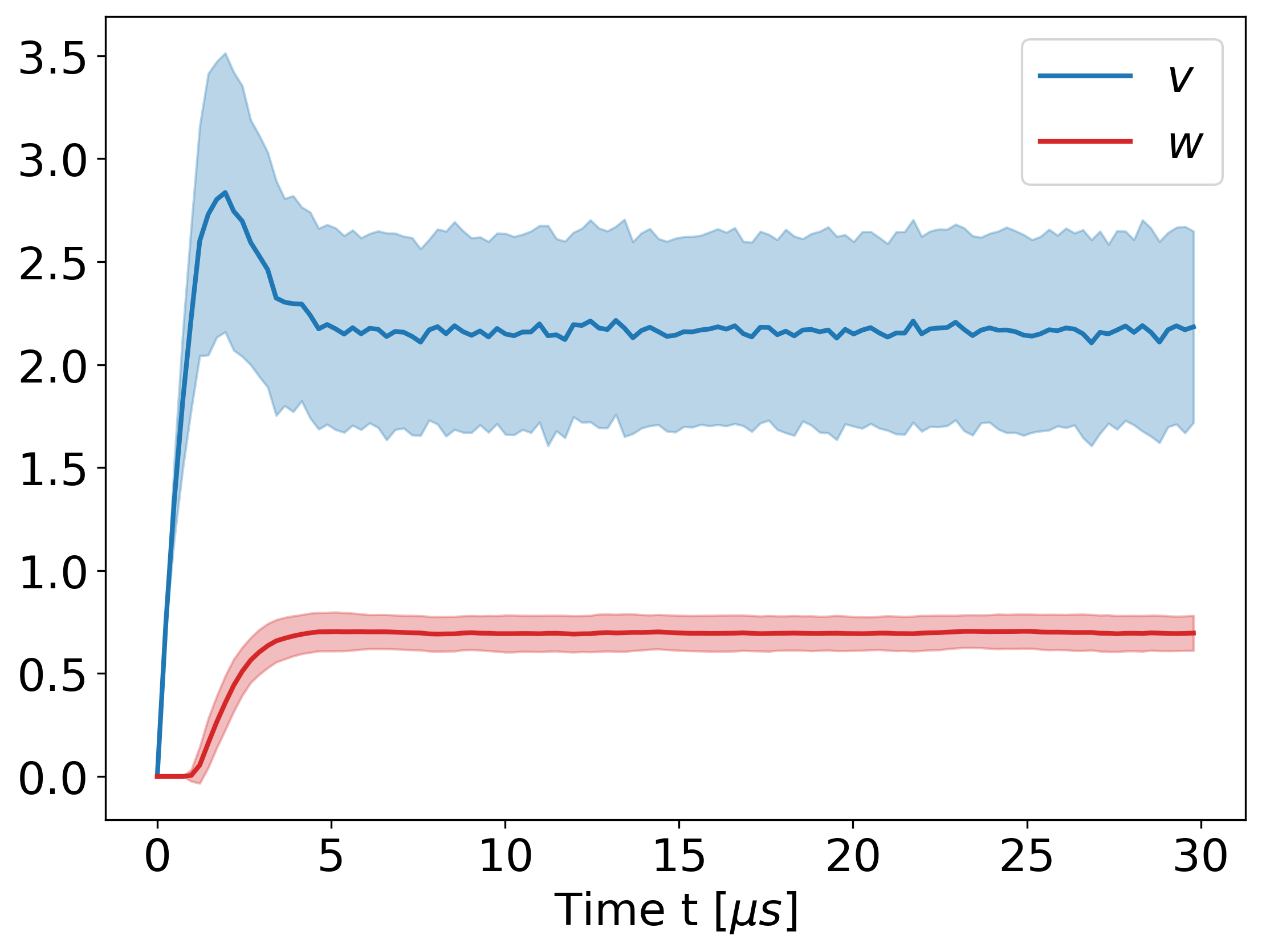}
    \caption{$v$ at the pole $\theta=\pi$.}
    \label{fig:v_vs_t_mult}
  \end{subfigure}\hfill
  \begin{subfigure}[t]{0.45\textwidth}
    \centering
    \includegraphics[width=\linewidth]{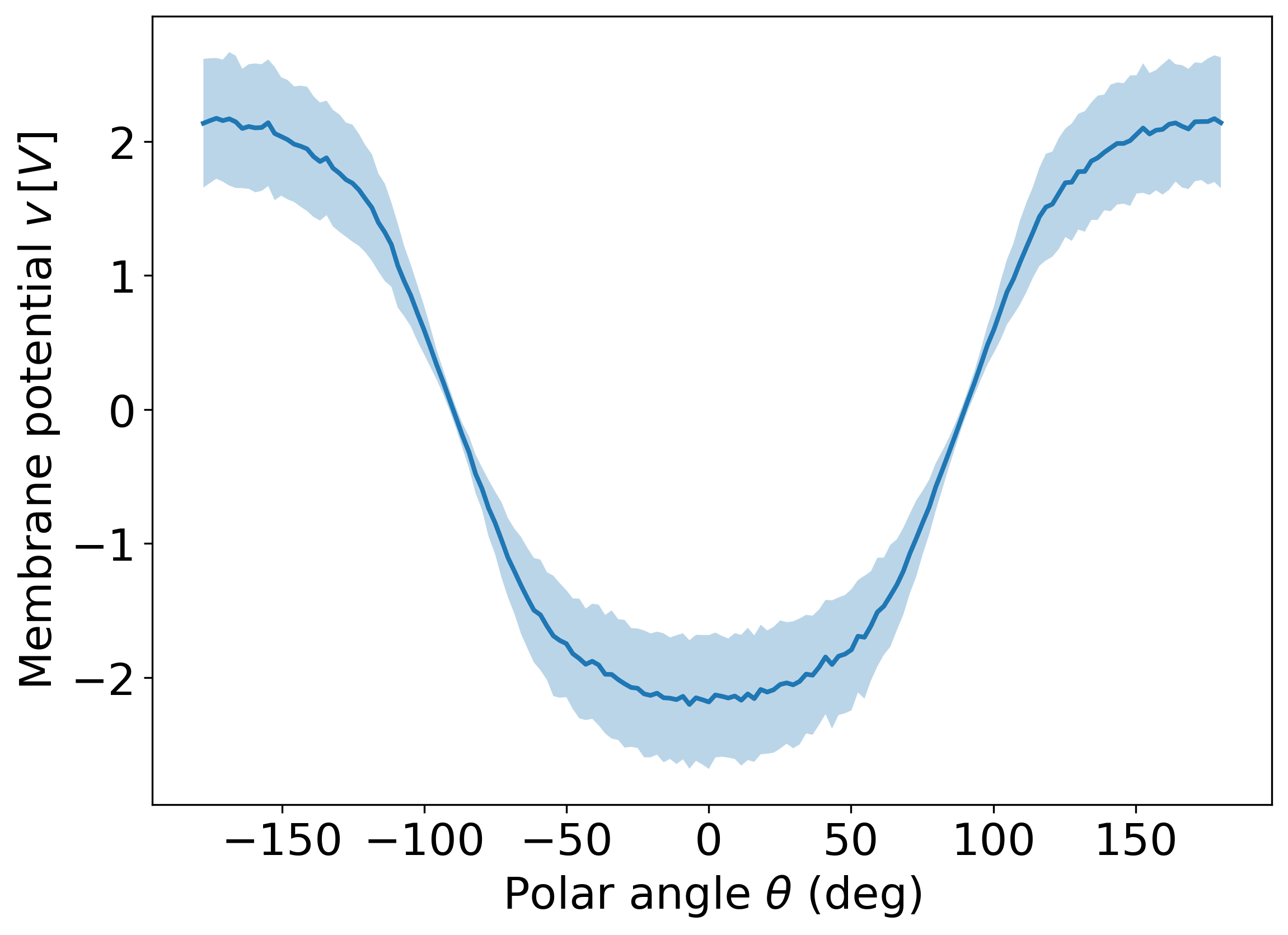}
    \caption{$v$ at $T_{\rm final}=500\, \mu$s.}
    \label{fig:v_vs_theta_mult}
  \end{subfigure}\hfill
  \caption{Membrane potential $v$ and degree of porosity $w$ in the case of linear multiplicative noise.}
  \label{fig:v_and_w_lin_multiplicative}
\end{figure}
\begin{figure}[!htb]
  \centering
  \begin{subfigure}[t]{0.45\textwidth}
    \centering
    \includegraphics[width=\linewidth]{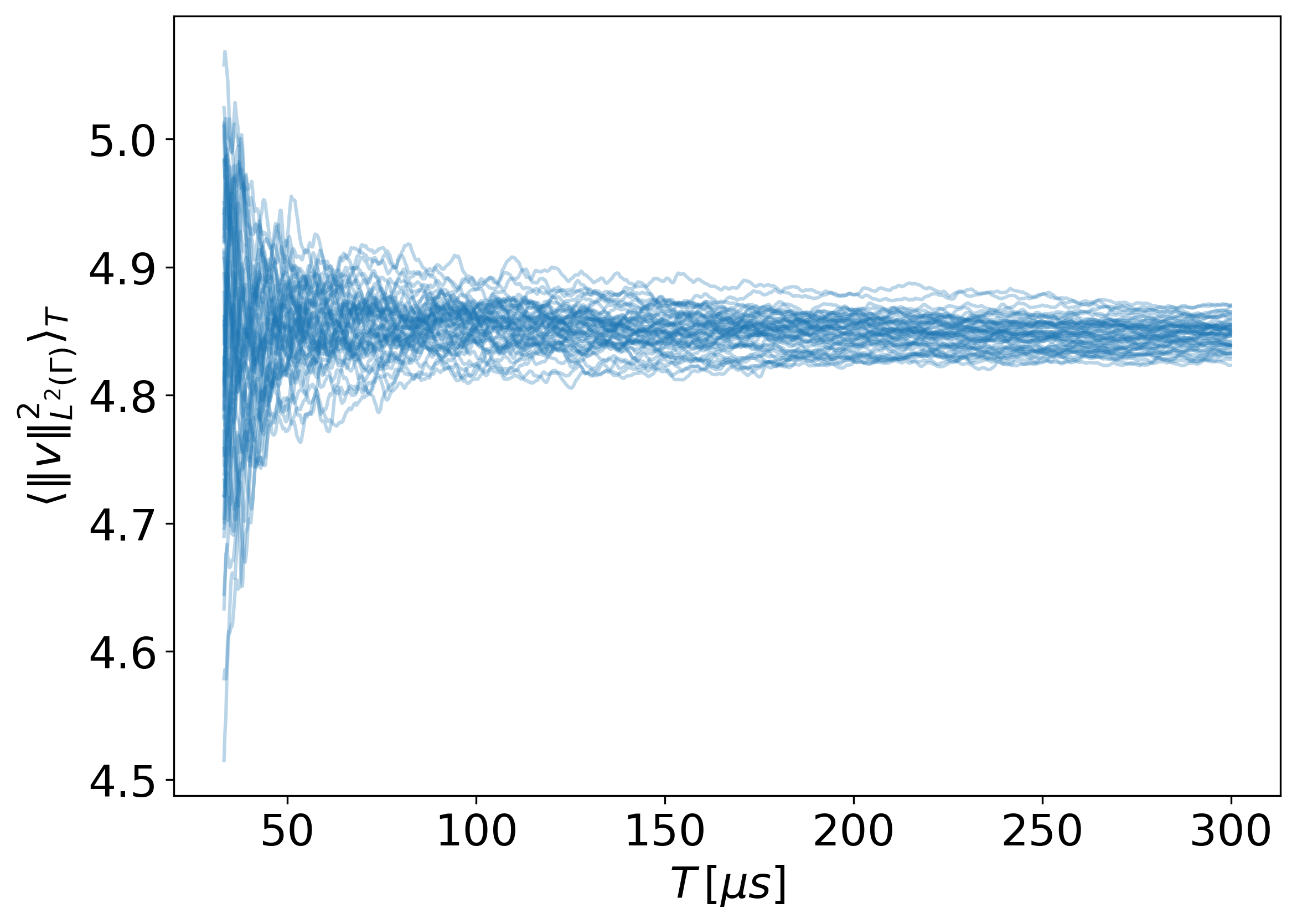}
    \caption{$\langle \|v\|_{L^2(\Gamma)}^2\rangle_{T}$ over $[1.1 T_{\rm burn-in},T_{\rm final}]$.}
    \label{fig:w_vs_t_mult}
  \end{subfigure}\hfill
 \begin{subfigure}[t]{0.45\textwidth}
    \centering
    \includegraphics[width=\linewidth]{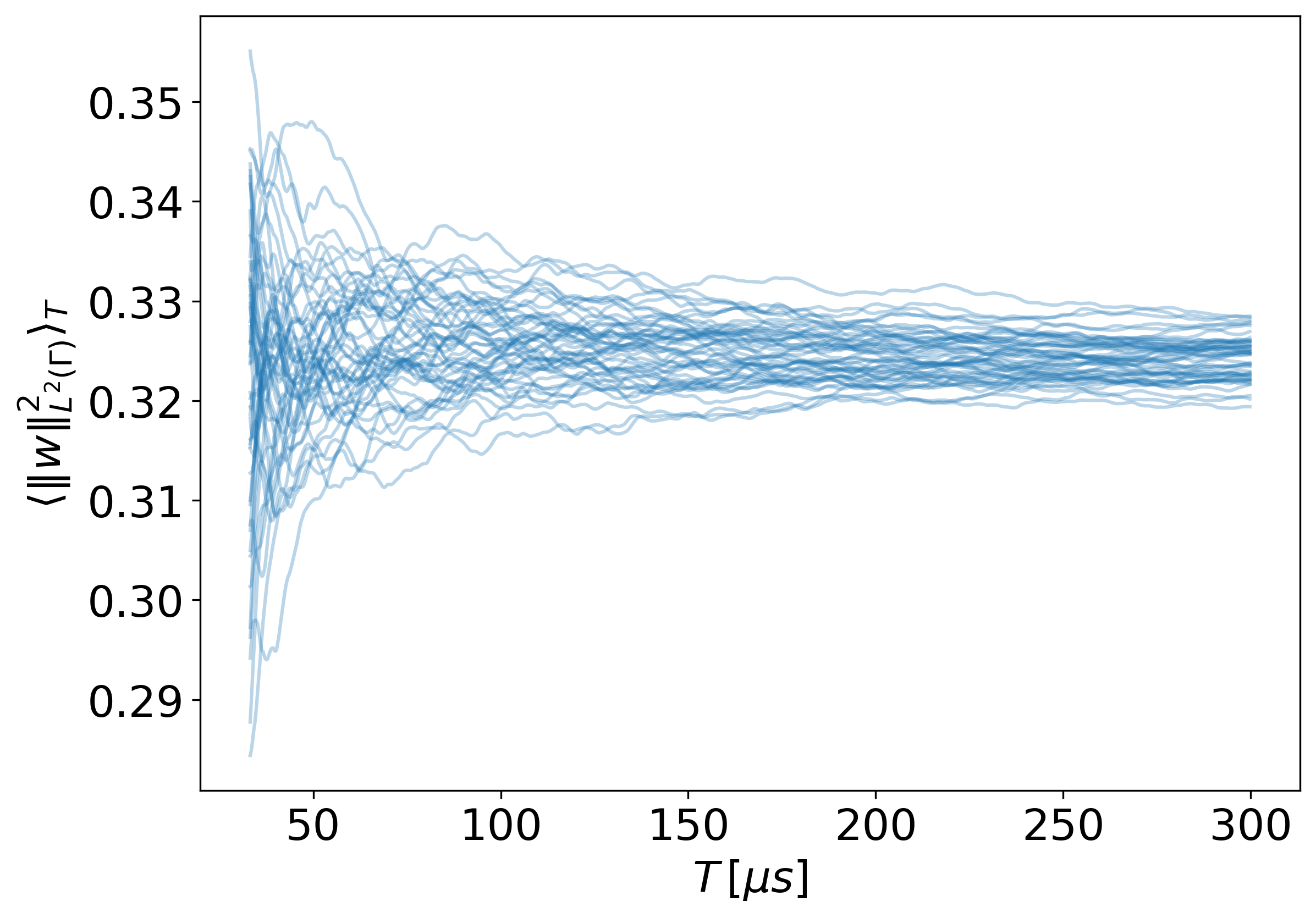}
    \caption{$\langle \|w\|_{L^2(\Gamma)}^2\rangle_{T}$ over $[1.1 T_{\rm burn-in},T_{\rm final}]$.}
\label{fig:time_averages_mult}
  \end{subfigure}
  \caption{Time averages in the case of linear multiplicative noise.}
\label{fig:time_avg_lin_multiplicative}
\end{figure}
We consider a multiplicative noise of the form $\dd W_t = \alpha \, v \, \dd \zeta(t)$, a linear multiplicative noise with $\alpha=0.5$, where $\zeta$ is a real-valued standard Brownian motion, as well as a multiplicative noise of the form $\dd W_t = \alpha \, \cos(v) \, \dd \zeta(t)$. We choose $T_{\rm burn-in}=50\, \mu$s and the final time $T_{\rm final}=300\, \mu$s. The solution $(v,w)$ is shown in Figure \ref{fig:v_and_w_lin_multiplicative} for the linear multiplicative noise and in Figure \ref{fig:v_and_w_multiplicative_cos} for $b(v)=\alpha \cos v$. The time averages for the linear multiplicative noise $\langle \|v\|_{L^2(\Gamma)}^2\rangle_{T}$ and $\langle \|w\|_{L^2(\Gamma)}^2\rangle_{T}$ defined by \eqref{eq:M_v_L2} are presented in Figure \ref{fig:time_avg_lin_multiplicative}. 
\begin{figure}[!htb]
    \centering
  \centering
  \begin{subfigure}[t]{0.45\textwidth}    \includegraphics[width=\linewidth]{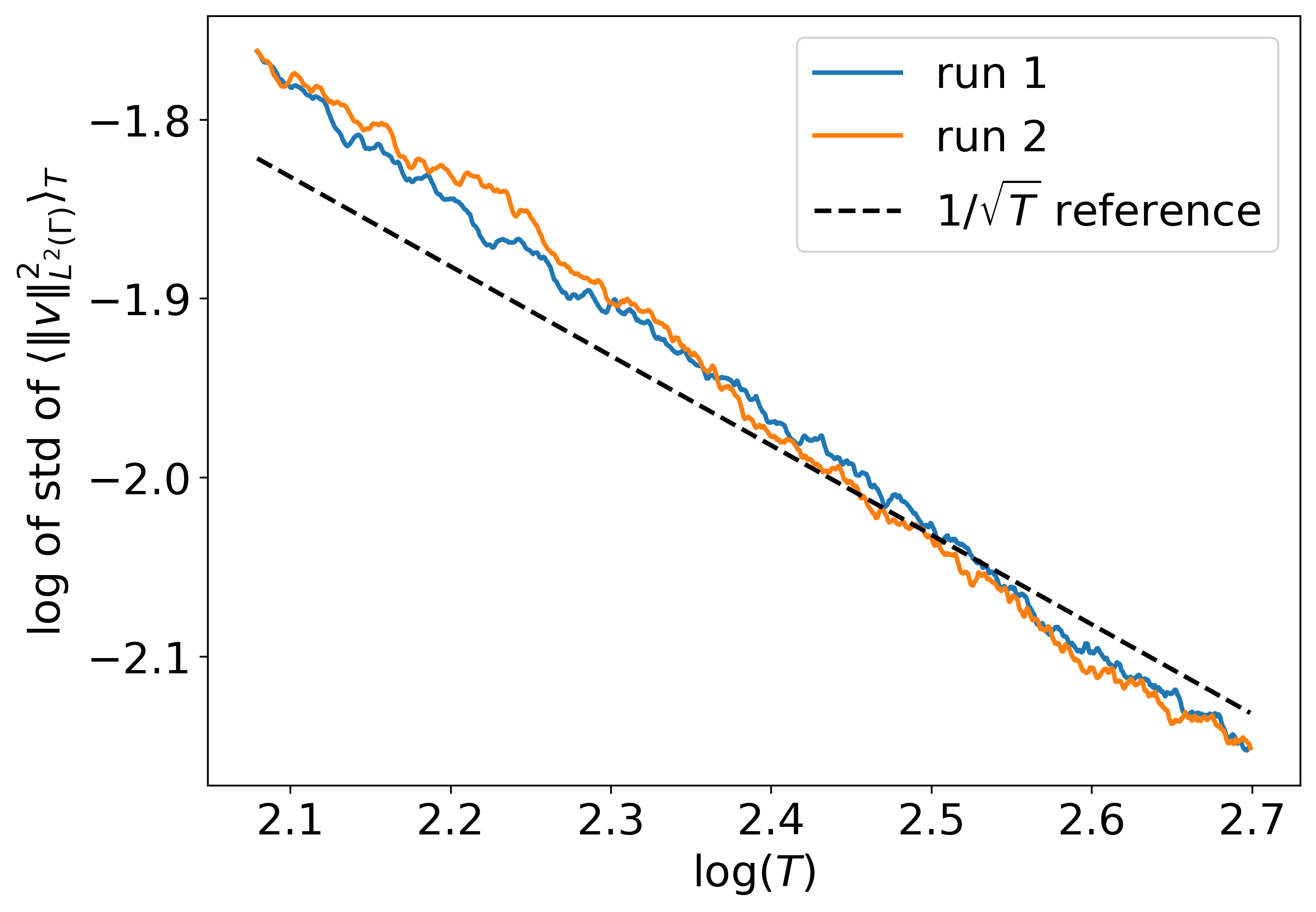}
  \end{subfigure}
\hspace{1cm}
  \begin{subfigure}[t]{0.45\textwidth}
    \includegraphics[width=\linewidth]{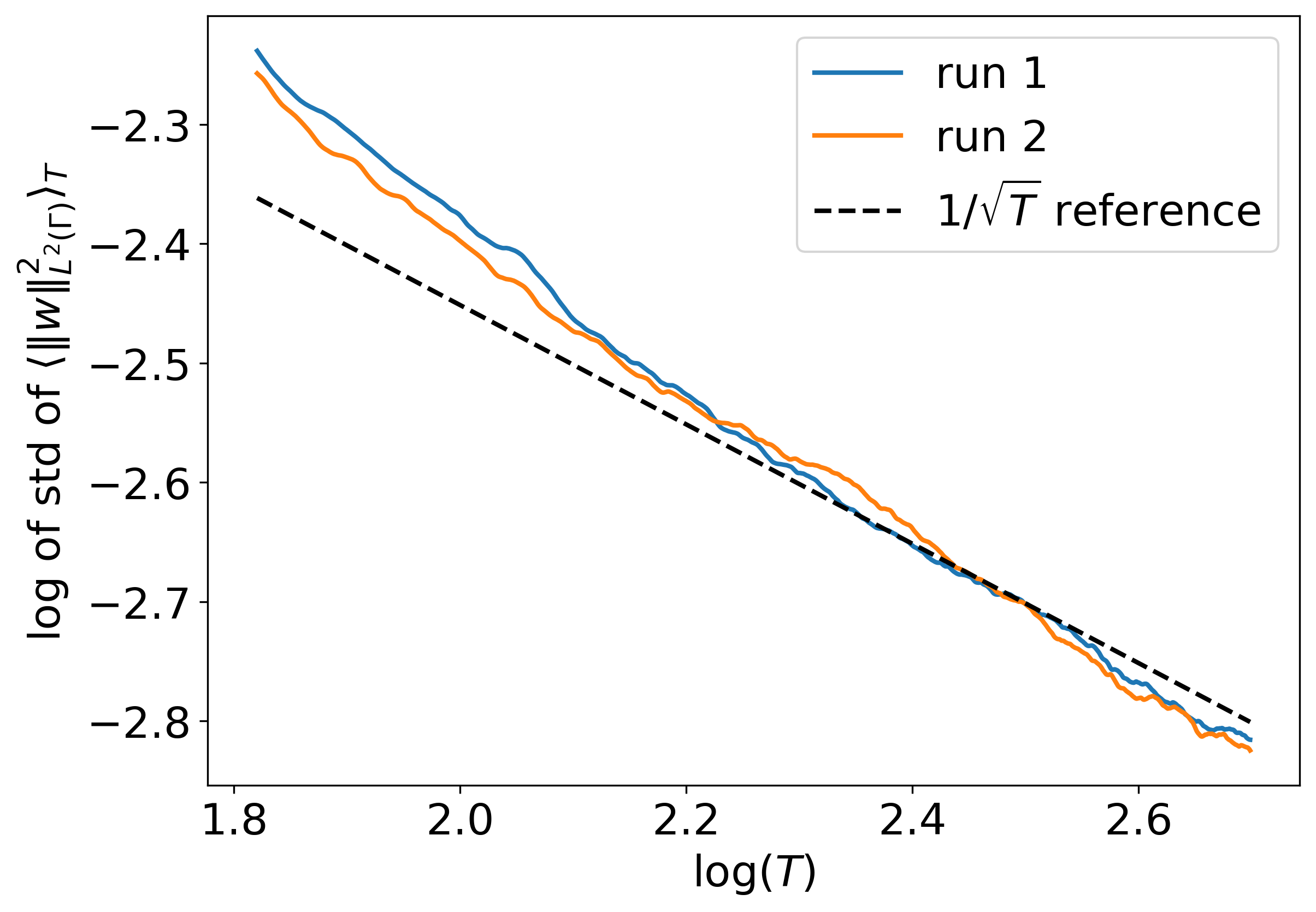}
    \end{subfigure}
  \caption{Sample standard deviation of $\langle \|v\|_{L^2(\Gamma)}^2\rangle_{T}$ and $\langle \|w\|_{L^2(\Gamma)}^2\rangle_{T}$ in the case of linear multiplicative noise for two MC runs, with 500 sample each, $T_{\rm burn-in}=30\, \mu$s, $T_{\rm final}=500\, \mu$s.}
  \label{fig:std_lin_multiplicative}
\end{figure}
\begin{figure}[!htb]
  \centering
  \begin{subfigure}[t]{0.45\textwidth}
    \centering
    \includegraphics[width=\linewidth]{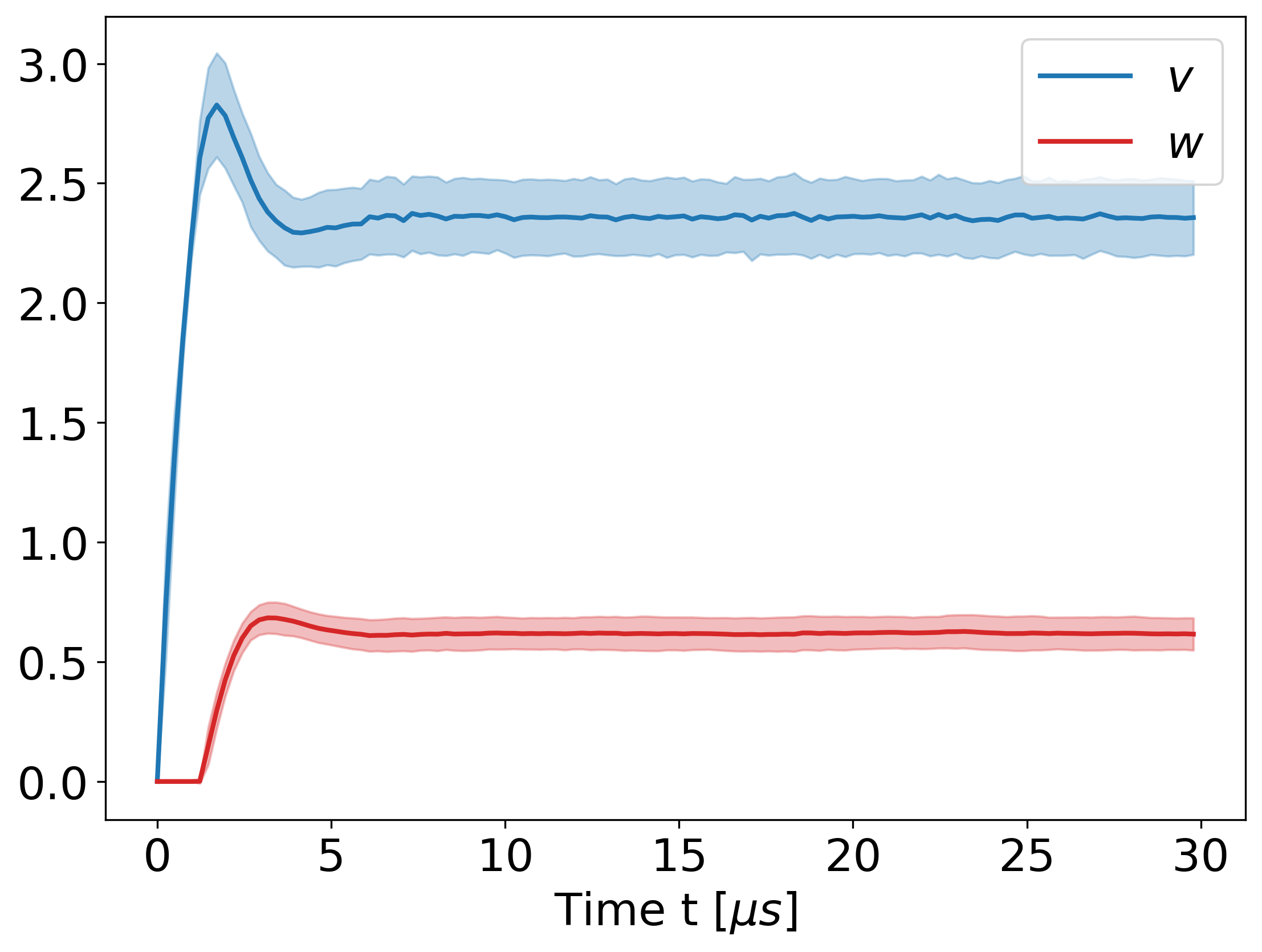}
    \caption{$v$ at the pole $\theta=\pi$.}
    \label{fig:v_vs_t_mult_cos}
  \end{subfigure}\hfill
  \begin{subfigure}[t]{0.45\textwidth}
    \centering
    \includegraphics[width=\linewidth]{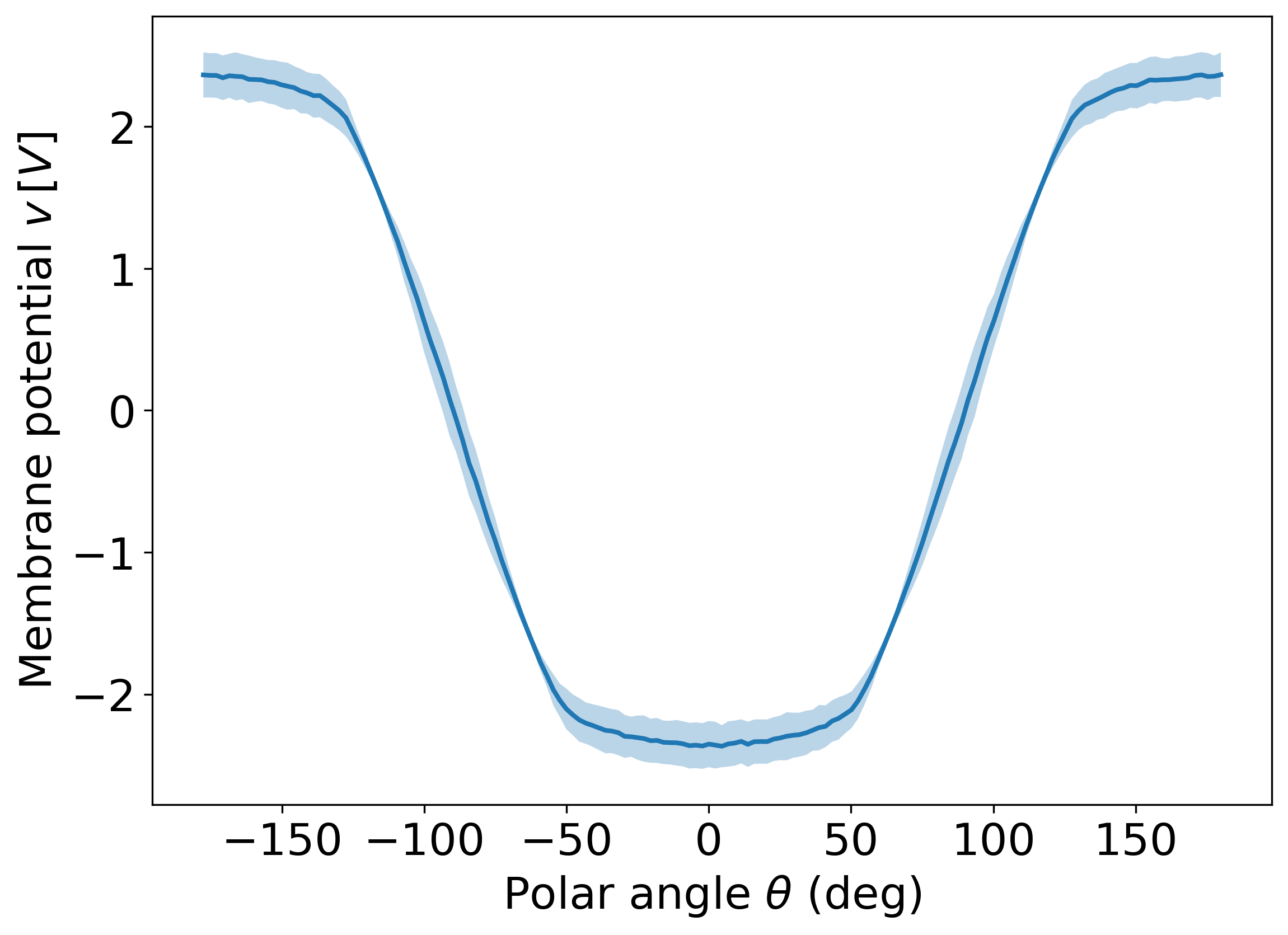}
    \caption{$v$ at $T_{\rm final}=500\, \mu$s.}
    \label{fig:v_vs_theta_mult_cos}
  \end{subfigure}\hfill
  \caption{Membrane potential $v$ and degree of porosity $w$ in the case of the multiplicative noise $b(v)=\cos(v)$.}
  \label{fig:v_and_w_multiplicative_cos}
\end{figure}
The $\log-\log$ plots of the sample standard deviations are shown in Figure~\ref{fig:std_lin_multiplicative} and Figure~\ref{fig:std_multiplicative_cos}. We see that, as for additive noise, the time averages stabilize, and the sample standard deviation decays with the rate $T^{-1/2}$, as $T\to \infty$, consistent with the behavior predicted by ergodic averaging. The curves corresponding to the bounded multiplicative noise $b(v) = \alpha \cos(v)$ appear smoother and exhibit better agreement between independent Monte Carlo runs than those obtained for linear multiplicative noise. This is likely due to the boundedness of the cosine coefficient, which restricts the magnitude of the stochastic forcing and reduces the impact of large fluctuations on finite-sample statistics. 
\begin{figure}[!htb]
    \centering
  \centering
  \begin{subfigure}[t]{0.45\textwidth}    \includegraphics[width=\linewidth]{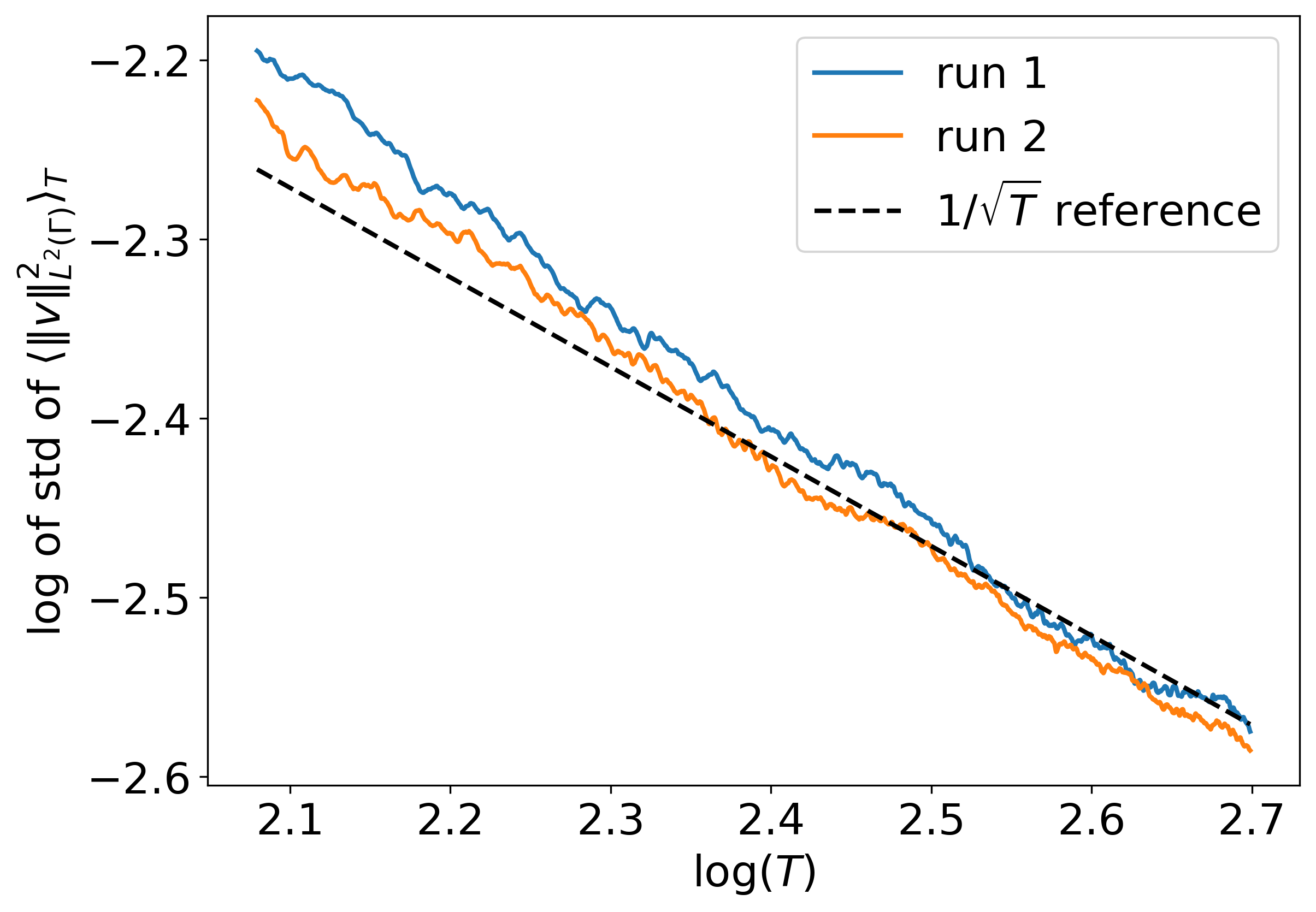}
  \end{subfigure}
\hspace{1cm}
  \begin{subfigure}[t]{0.45\textwidth}
    \includegraphics[width=\linewidth]{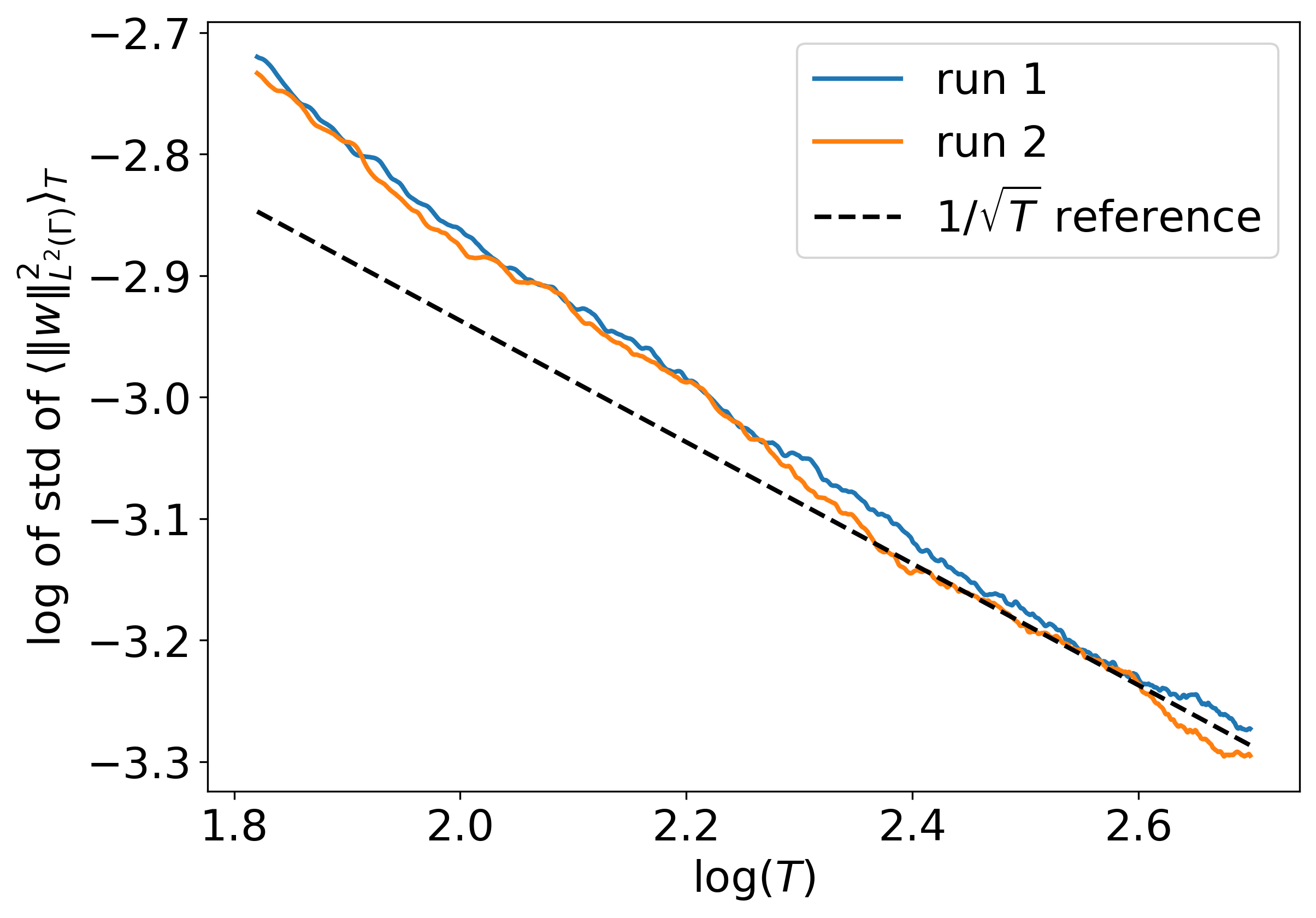}
    \end{subfigure}
  \caption{Sample standard deviation of $\langle \|v\|_{L^2(\Gamma)}^2\rangle_{T}$ and $\langle \|w\|_{L^2(\Gamma)}^2\rangle_{T}$ in the case of  multiplicative noise $b(v)=\cos(v)$ for two MC runs, with 500 sample each, $T_{\rm burn-in}=30\, \mu$s, $T_{\rm final}=500\, \mu$s.}
  \label{fig:std_multiplicative_cos}
\end{figure}

\section*{Acknowledgements}
A part of this work was done during the stay of I. Motschan Ulander and I. Pettersson at VCU, whose hospitality is greatly appreciated. The research stay at VCU was supported by the Barbro Osher Pro Suecia Foundation. O. Misiats visit to Chalmers was funded by Olle Engkvist's project (No. 227-0235). The research of  O. Misiats was supported by Simons Collaboration
Grant for Mathematicians No. 854856 and National Science Foundation Grant DMS-2408507.

The authors thank O. Stanzhytskii for stimulating discussions on the results. 

\vspace{-0.5cm}

\bibliographystyle{abbrv}
\bibliography{refs-noise}

\end{document}

%% file: multicells.pdf_tex
\begingroup%
  \makeatletter%
  \providecommand\color[2][]{%
    \errmessage{(Inkscape) Color is used for the text in Inkscape, but the package 'color.sty' is not loaded}%
    \renewcommand\color[2][]{}%
  }%
  \providecommand\transparent[1]{%
    \errmessage{(Inkscape) Transparency is used (non-zero) for the text in Inkscape, but the package 'transparent.sty' is not loaded}%
    \renewcommand\transparent[1]{}%
  }%
  \providecommand\rotatebox[2]{#2}%
  \newcommand*\fsize{\dimexpr\f@size pt\relax}%
  \newcommand*\lineheight[1]{\fontsize{\fsize}{#1\fsize}\selectfont}%
  \ifx\svgwidth\undefined%
    \setlength{\unitlength}{427.88975513bp}%
    \ifx\svgscale\undefined%
      \relax%
    \else%
      \setlength{\unitlength}{\unitlength * \real{\svgscale}}%
    \fi%
  \else%
    \setlength{\unitlength}{\svgwidth}%
  \fi%
  \global\let\svgwidth\undefined%
  \global\let\svgscale\undefined%
  \makeatother%
  \begin{picture}(1,0.80264325)%
    \lineheight{1}%
    \setlength\tabcolsep{0pt}%
    \put(0,0){\includegraphics[width=\unitlength,page=1]{multicells.pdf}}%
    \put(0.72150219,0.21682999){\color[rgb]{0,0,0}\makebox(0,0)[lt]{\lineheight{1.25}\smash{\begin{tabular}[t]{l}$\mathbf{n}$\end{tabular}}}}%
    \put(0.30106346,0.28246119){\color[rgb]{0,0,0}\makebox(0,0)[lt]{\lineheight{1.25}\smash{\begin{tabular}[t]{l}$G_i$\end{tabular}}}}%
    \put(0.44531815,0.10036929){\color[rgb]{0,0,0}\makebox(0,0)[lt]{\lineheight{1.25}\smash{\begin{tabular}[t]{l}$G_e$\end{tabular}}}}%
    \put(0,0){\includegraphics[width=\unitlength,page=2]{multicells.pdf}}%
    \put(0.80240703,0.69630568){\color[rgb]{0,0,0}\makebox(0,0)[lt]{\lineheight{1.25}\smash{\begin{tabular}[t]{l}$\partial G$\end{tabular}}}}%
    \put(0,0){\includegraphics[width=\unitlength,page=3]{multicells.pdf}}%
    \put(0.60042291,0.60424414){\color[rgb]{0,0,0}\makebox(0,0)[lt]{\lineheight{1.25}\smash{\begin{tabular}[t]{l}$\Gamma$\end{tabular}}}}%
  \end{picture}%
\endgroup%